\documentclass{scrartcl}

\usepackage{latexsym, amsmath, amssymb, amsthm, stmaryrd}
\usepackage{tikz}
\usetikzlibrary{cd}
\usepackage{hyperref}

\usepackage{enumitem}

\begin{document}

\newtheorem{theo}{Theorem}[section]
\theoremstyle{definition}
\newtheorem{defi}[theo]{Definition}
\newtheorem{theo1}[theo]{Theorem}
\newtheorem{prop}[theo]{Proposition}
\newtheorem{lemm}[theo]{Lemma}
\newtheorem{coro}[theo]{Corollary}
\newtheorem{prob}[theo]{Problem}
\newtheorem{remark}[theo]{Remark}
\newtheorem{example}[theo]{Example}
\newtheorem{problem}[theo]{Problem}

\newcommand{\elt}{\ensuremath{v \in G}}
\newcommand{\elg}{\ensuremath{\phi \in \Gamma}}
\newcommand{\gstab}{\ensuremath{G_v}}
\newcommand{\caa}{\ensuremath{C_1(\alpha)}}
\newcommand{\cia}{\ensuremath{C_i(\alpha)}}
\newcommand{\ima}{\ensuremath{\mathrm{Im}}}
\newcommand{\de}{\ensuremath{{\mathrm{deg}}}}
\newcommand{\dist}{\ensuremath{\mathrm{dist}}}
\newcommand{\parent}{\ensuremath{\mathrm{par}}}
\newcommand{\wre}{\ensuremath{\,\mathrm{Wr}}}
\newcommand{\proj}{\ensuremath{\mathrm{proj}}}
\newcommand{\anc}{\ensuremath{\mathrm{anc}}}
\newcommand{\res}{\ensuremath{\mid}}
\newcommand{\nin}{\ensuremath{\ \not\in\ }}
\newcommand{\sym}{\ensuremath{\mathrm{Sym}}}
\newcommand{\aut}{\ensuremath{\mathrm{Aut}}}
\newcommand{\autsc}{\ensuremath{\mathrm{Aut}_{\mathrm{sc}}}}
\newcommand{\auttn}[0]{\ensuremath{\aut(T_n)}}
\newcommand{\autx}[0]{\ensuremath{\aut(X)}}
\newcommand{\sign}{\ensuremath{\mathrm{sign}}}
\newcommand{\Wr}{\ensuremath{\,\mathrm{Wr}}\,}

\newcommand{\Pk}{\ensuremath{\mathrm{(P}}_k\mathrm{)}}

\newcommand{\Star}{\ensuremath{\mathrm{Star}}}
\newcommand{\TStar}{{\ensuremath{\mathrm{TStar}}}}
\newcommand{\cL}{{\cal L}}
\newcommand{\cU}{{\cal U}}
\newcommand{\cG}{{\cal G}}
\newcommand{\cV}{{\cal V}}

\newcommand{\tg}{\tilde{g}}
\newcommand{\tf}{\tilde{f}}
\newcommand{\tih}{\tilde{h}}
\newcommand{\tG}{\tilde{G}}
\newcommand{\tx}{\tilde{x}}
\newcommand{\tX}{\tilde{X}}

\newcommand\rev[1]{\mathop{\overline{#1}}}
\newcommand{\ba}{\rev{a}}
\newcommand{\bb}{\rev{b}}
\newcommand{\bc}{\rev{c}}
\newcommand{\be}{\rev{e}}

\newcommand{\V}{\mathrm V}
\newcommand{\VX}{{\mathrm V}X}
\newcommand{\E}{\mathrm E}
\newcommand{\A}{\mathrm A}
\newcommand{\TA}{\mathrm{TA}}
\newcommand{\TE}{\mathrm{TE}}
\newcommand{\id}{\mathrm {id}}
\newcommand{\re}{\mathrm {re}}
\newcommand{\pr}{\mathrm {pr}}
\newcommand{\con}{\mathrm {cr}}

\newcommand{\RR}{\ensuremath{\mathbb R}}
\newcommand{\ZZ}{\ensuremath{\mathbb Z}}
\newcommand{\QQ}{\ensuremath{\mathbb Q}}
\newcommand{\NN}{\ensuremath{\mathbb N}}
\newcommand{\SRR}{\mbox{${^*}{\mathbb{R}}$}}

\newcommand\restr[2]{{
  \left.\kern-\nulldelimiterspace 
  \vphantom{|}#1
  \right|_{#2}
  }}

  \newcommand{\ggawo}{{ggawo}}
\newcommand{\gga}{{gga}}

  \newcommand{\bun}[1]{\llbracket #1 \rrbracket} %for vertex and edge bundles
\newcommand{\f}[1]{\textcolor{red}{Florian: #1}}
\newcommand{\ro}[1]{\textcolor{green}{Roggi: #1}}

\title{Graphs of group actions and group actions on trees}

\medskip

\author{Florian Lehner, Christian Lindorfer, R\"ognvaldur G.~M\"oller, \\and Wolfgang Woess}
\smallskip

\maketitle

\begin{abstract}
    Bass--Serre theory provides a powerful framework for studying group actions on trees. While extremely effective for structural questions in group theory, it is less suited to the systematic construction of group actions with prescribed local behaviour. Motivated by local‑to-global constructions such as the Burger--Mozes universal groups and local action diagrams, we develop an analogue of Bass--Serre theory for group actions. 

    The central object of study in our are graphs of group actions, combinatorial structures similar to graphs of groups from Bass--Serre theory, encoding compatible local permutation actions on a base graph. From these we can construct groups which act on tree-like graphs called scaffoldings and hence also on trees. 
    
    We prove uniqueness and universality results for the resulting groups and show that our framework unifies and generalises (among other known constructions) both graphs of groups and local action diagrams. Remarkably, we are able to encapsulate the full generality of the former while still allowing for efficient construction of groups with certain local properties like in the latter.

\end{abstract}

\section*{Introduction}

Combining simpler pieces in a tree-like way can be a very powerful tool, as evidenced by beautiful results in 
group theory \cite{MullerSchupp83,ReidSmith2020,Serre2003,Stallings68,Stallings71}, 
graph theory \cite{Carmesinetal14,Carmesinetal19,DunwoodyKron15,RobertsonSeymour86,RobertsonSeymour91},  and many other mathematical fields such as
computational complexity \cite{ArnborgLagergrenSeese91,Courcelle90}, and
probability theory \cite{BenjaminiSchramm96}. 

Free products with amalgamation of two groups and HNN-extensions are important examples of such constructions in geometric group theory; Bass--Serre theory \cite{Serre2003} uses the concept of a \emph{graph of groups} to describe groups that one gets by repeated use of these two constructions. The fundamental theorem of Bass--Serre theory states that we can associate to every graph of groups its fundamental group, a group which naturally acts on a tree. Conversely, every group acting on a tree is the fundamental group of a graph of groups.
This correspondence has been uniquely useful when it comes to analysing groups acting on trees and in theory it can be used to construct any group acting on a tree by giving its corresponding graph of groups. However, as observed by Reid and Smith \cite{ReidSmith2020}, the construction of groups with certain properties using the graph-of-groups construction can be difficult.

To overcome these difficulties, a complementary approach of defining groups by local actions (that is, actions of vertex stabilisers on the neighbourhood of the fixed vertex) has emerged. Two important constructions following this principle are the universal group construction of Burger and Mozes \cite{BurgerMozes2000} and Smith's Box product construction \cite{Smith2017}. Similarly to amalgamated free products and HNN-extensions in Bass--Serre theory, these two constructions can be combined to define a local-action complement to classical Bass–Serre theory, see \cite{ReidSmith2020}. The central object in the resulting theory are \emph{local action diagrams} which have some striking similarities to graphs of groups. To every local action diagram we can associate a universal group which , similarly to the fundamental group of Bass--Serre theory, acts on a tree. It is worth pointing out that this universal group always has Tits' independence property (P), and conversely, every group acting on a tree with property (P) is the universal group of some local action diagram.

The aim of this work is to develop an analogue of Bass--Serre theory for group actions.  This unifies and generalises known constructions including the aforementioned graph of groups \cite{Serre2003} and local action diagram \cite{ReidSmith2020} constructions and many others \cite{BanksElderWillis2015,BurgerMozes2000,Hamannetal2022,Mohar2006,Smith2017}.

The central notion in our theory is a combinatorial structure called a \emph{graph of group actions}. This is directly modelled on the concept of a graph of groups but the definition of the \emph{universal group} of a graph of group actions is more closely related to the local-to-global approach of local action diagrams. More precisely, a graph of group actions consists of a connected graph $\Delta$, called the base graph,  together with permutation group action at each vertex and each arc together with an embedding of the group action at an arc into the group action at the terminal vertex of the arc. To each graph of group actions we can associate its universal group which acts on a tree-like graph (called the \emph{scaffolding} for the graph of group actions); an action of the universal group on a tree can be obtained by taking a quotient of the scaffolding.

Within this framework, one can recover both the fundamental group of a graph of groups in Bass--Serre theory by letting all groups act regularly on themselves, as well as the theory of local action diagrams by taking all arc groups to be trivial groups acting on a one-element set. We note that in the case of Bass--Serre theory, the scaffolding is closely related (but not quite equal) to a Cayley graph of the resulting fundamental group.

It is worth pointing out that we can also add more structure to the sets the vertex and edge groups act on. For instance, if at each vertex we have an action of a group of automorphisms on some graph, we can recover the notion of \emph{tree amalgamation of graphs} introduced by Mohar in \cite{Mohar2006}.  A variant of Mohar's tree amalgamation which respects given subgroups of the automorphism groups of the constituent graphs was introduced by Hamann et al.~\cite[Section 5]{Hamannetal2022}.  Our construction extends the constructions in \cite{Mohar2006} and \cite{Hamannetal2022} in very much the same way as the graph of groups and their fundamental groups in Bass--Serre theory extend the construction of free products with amalgamations in group theory.  These connections will be explored in a forthcoming paper.

In another forthcoming paper the theory presented in this paper will be used to analyse group action on tree that have property (P$_k$) that was defined by Banks, Elder and Willis in \cite{BanksElderWillis2015}.

The rest of this paper is structured as follows. In the first section basic terminology and basic concepts are explained.  Our terminology for graphs is similar to the terminology used in Serre's book \cite{Serre2003}, except that we allow edges that are self reverse, i.e.\ $\ba=a$.   

Section~\ref{SDefinition} contains the definition of graphs of group actions, scaffoldings and the universal group of a graph of group actions.  At first the definition of a graph of group actions is limited to the case where the base graph has no self-reverse arcs; the general definition where edge inversions are permitted is deferred to Section~\ref{SInversions} where it is also shown that results from the special case without inversions carry over to the general case by a subdivision argument.  

In Section~\ref{SExistence} it is shown that scaffolding graphs do exist for such graphs of group actions and that the definition of the universal group makes sense.   In Section~\ref{SExtension}, it is then shown that any two scaffoldings are isomorphic in a strong sense which also implies that the group (and its action on the scaffolding graph) defined by a graph of group actions is unique up to conjugation.  The technicalities of proving that are made simpler by excluding the possibility of thee being self-reverse arcs in the base graph.  

As mentioned above, in Section~\ref{SInversions} graphs of group actions where the base graph has self-reverse arcs are defined. Concepts defined in Section~\ref{SDefinition} are extended to this case and the main result of Section~\ref{SExtension} is generalised.  This gives us Theorem~\ref{thm:acceptable-isomorphism} that is used to show that any two scaffoldings for the same graph of group actions give isomorphic actions of the universal groups on the respective scaffoldings.  

In Section~\ref{SPermutations} basic properties of the actions of the universal group on a scaffolding and the associated tree are discussed.  

As mentioned above, the definition of a graph of group actions and their universal groups combine ideas from Bass--Serre theory and ideas from local-to-global constructions.  In Section~\ref{SBass-Serre} it is shown how a graph of groups can in a simple way be turned into a graph of group actions so that the action of the universal group of this graph of group actions on the associated tree is the same as the action of the fundamental group of the graph of groups on the associated tree.  In the subsequent section it is then shown that the universal groups of Burger and Mozes can be defined in terms of graphs of group actions.  We note that our construction is more general, because if one thinks in terms of actions on trees then the universal groups of Reid and Smith's local action diagram correspond to (P)-closed groups, whereas our construction also gives groups which do not satisfy property (P); some examples are given in the appendix.  

The final section of this paper gives an alternative view of the universal group. A graph of groups actions can contain a lot of information that is immaterial when it comes to considering the action on the associated tree.  In Section~\ref{SReduced} it is shown how it is possible to remove some of that surplus information to get a simpler graph of group actions that gives the same group action on a tree as the original graph of group actions.  In addition, it is shown that one could define the action of the universal group on the associated tree by considering a \lq\lq tree-like\rq\rq\ graph 
that looks like the associated tree but might have multiple parallel edges.

Finally, Appendix~\ref{SExamples} contains simple examples of graphs of group actions yielding interesting actions on trees.   The motivation behind local action diagrams of Reid and Smith is to construct and study group actions on trees that have Tits' Property (P) (see Section~\ref{SPropertyP}).  The examples in Appendix~\ref{SExamples} show that with graphs of group actions one can construct complicated groups actions on trees that don't have Property (P).

\section{Preliminaries}

\subsection{Graphs}\label{SGraphs}

Our notation for graphs is similar to the notation in Serre's book \cite{Serre2003}.  
A graph $\Gamma$ consists of a set of \emph{vertices}  $\V\Gamma$, a set of \emph{arcs} $\A\Gamma$ and maps $\A\Gamma\to \V\Gamma\times\V\Gamma;a\mapsto (o(a), t(a))$ and $\A\Gamma\to \A\Gamma; a\to \ba$ such that ${\rev{\rev{a}}}=a$  and $o(a)=t(\ba)$ for every $a\in \A\Gamma$.   For an arc $a$ the vertex $o(a)$ is called the {\em origin} of $a$, the vertex $t(a)$ is called the \emph{terminus} of $a$ and the arc $\ba$ is called the \emph{reverse} of the arc $a$.   An arc $a$ such that $o(a)=t(a)$ is called a {\em loop}.   If $a$ is a loop then it is possible that $a\neq\ba$, but also that $a=\ba$; in the latter case we call the arc $a$ \emph{self-reverse}.   The arcs $a$ and $\ba$ form together an {\em edge} in the graph; the set of edges will be denoted by $\E\Gamma$. In the case where $a=\ba$ the corresponding edge is said to be \emph{self-reverse}.

The same notation is used for {\em digraphs} except that then we do not have the map taking an arc to its reverse.  A digraph $\Gamma$ is defined as a quadruple $(\V\Gamma, \A\Gamma, o, t)$, where $\V\Gamma$ is the set of {\em vertices}, $\A\Gamma$ is the set of {\em arcs}, and $o$ and $t$ are maps $\A\Gamma\to\V\Gamma$.  As above we call $o(a)$ the \emph{origin} of the arc $a$ and $t(a)$ is called the \emph{terminus} of $a$.  Note that every graph is also a digraph.

We say that distinct vertices $u$ and $v$ (in a graph or a digraph) are \emph{adjacent}, or \emph{neighbours}, if there is an arc $a$ such that $\{o(a), t(a)\}=\{u,v\}$.   The set of vertices adjacent to a vertex $v$ in a graph $\Gamma$ is denoted by $N_\Gamma(v)$, or just $N(v)$ if there is no danger of confusion.  

A graph is called {\em simple} if it has no loops and the map $\A\Gamma\to \V\Gamma\times\V\Gamma; a\mapsto (o(a), t(a))$ is injective.  An arc $a$ in a simple graph can be represented by an ordered pair of vertices $(u,v)$ where $u=o(a)$ and $v=t(a)$, and the corresponding edge can be represented by an unordered pair $\{u, v\}$.   
For a digraph (or a graph), we can also define the \emph{underlying simple graph} as the simple graph that has the same vertex set as the digraph and two distinct vertices are adjacent if and only if they are adjacent in the digraph.

A {\em path} of length $n$ in $\Gamma$ (a graph or a digraph) is a sequence $v_0, \ldots, v_n$ of distinct vertices such that $v_i$ and $v_{i+1}$ are adjacent for $i=0, \ldots, n-1$.  If we do not insist on the vertices being distinct then we talk about a {\em walk}.   A graph is said to be {\em connected} if for any two vertices $u$ and $v$ there is a path starting with one of them and ending with the other.  A {\em simple cycle} of length $n$ is a walk $v_0, \ldots, v_n$ such that $v_0=v_n$ and the vertices $v_0, \ldots, v_{n-1}$ are all distinct.  

If the graph $\Gamma$ is connected one can define a metric $d_\Gamma$ on $\V\Gamma$ by defining the distance, $d_\Gamma(x,y)$, between two vertices $x$ and $y$ as the minimal length of path starting with one and ending with the other.  Let $v$ be a vertex.  The {\em ball of radius $n$ with center in $v$} in $\Gamma$ is the set $B_\Gamma(v,n)=\{u\in \V\Gamma\mid d_\Gamma(v,u)\leq n\}$ and the {\em sphere of radius $n$ with center in $v$} in $\Gamma$ is the set $S_\Gamma(v,n)=\{u\in \V\Gamma\mid d_\Gamma(v,u)=n\}$.

A {\em tree} is a simple connected graph with no simple cycles of length $\geq 3$.   A \emph{leaf} in a tree is a vertex with at most one neighbour.  Sometimes it is convenient to grant some specific vertex in a tree a special status; such a vertex is called the \emph{root} of the tree.  

\begin{defi}\label{Droot}
Suppose $T$ is a rooted tree with root $v_0$.  Define the \emph{parent}, $\parent(u)$, of $u \in \V T \setminus \{v_0\}$ to be the neighbour of $u$ closest to $v_0$, and set $p(v_0)= v_0$.   Conversely, a vertex $w$ is a \emph{child} of a vertex $v$ if $w$ is adjacent to $v$ and $d(v_0, w)=d(v, v_0)+1$.  %The \emph{depth} of a vertex vertex $w \in \V T$ is its distance to $v_0$.
\end{defi}

A graph $\Gamma'$ with vertex set $\V\Gamma'$ and arc set $\A\Gamma'$, is a {\em subgraph} of $\Gamma$ if $\V\Gamma'\subseteq \V\Gamma$ and $\A\Gamma'\subseteq \A\Gamma$ and the maps $o'$, $t'$ and $a\mapsto \ba$ for $\Gamma'$ equal the restrictions to $\A\Gamma'$ of the corresponding maps for $\Gamma$. 
If $V'\subseteq \V\Gamma$ then we define the \emph{subgraph spanned by} $V'$ as the subgraph having $V'$ as a vertex set and the set of arcs is the set of all arcs $a\in \A\Gamma$ such that both $o(a)$ and $t(a)$ are in $V'$.  Suppose $A'\subseteq \A\Gamma$ and that if $a\in \A\Gamma'$  then $\ba\in A'$.  The \emph{subgraph spanned by} $A'$ has $A'$ as a set of arcs and the set of vertices is the set of all vertices $v$ such that there exists an arc $a\in A'$ with $v=o(a)$ or $v=t(a)$.
{\em Subdigraphs} are defined similarly.

For a vertex $v$ define $\Star_\Gamma(v)$ (here $\Gamma$ is a graph or a digraph) as the set $t^{-1}(v)$ and $\Star_\Gamma^{-1}(v)=o^{-1}(v)$.   Thus $\Star_\Gamma(v)$ is the set of arcs in $\Gamma$ that have $v$ as their terminal vertex and $\Star_\Gamma^{-1}(v)$ is the set of arcs that have $v$ as their vertex of orgin.   
The {\em in-degree} of a vertex $v$ in a digraph $\Gamma$ is the cardinality of $\Star_\Gamma(v)$ and the \emph{out-degree} is the cardinality of $\Star_\Gamma^{-1}(v)$.  For a graph these two numbers will be equal and this number is the \emph{degree} of $v$.
%Furthermore define $\OStar_\Gamma(v)$ as the subdigraph of $\Gamma$ spanned by $\Star_\Gamma(v)$.  
%A graph is said to be {\em regular} if all vertices have the same degree $d$; in this case $d$ is called the {\em degree of the graph}.
If the degree of every vertex is finite, we say that the graph is {\em locally finite} (a digraph is locally finite if both the in- and out-degrees of every vertex are finite).

Let $\Gamma$ be a graph.   A \emph{graph equivalence relation} on $\Gamma$ consists of a pair of equivalence relations $\sim_{\V}$ and $\sim_{\A}$ where $\sim_{\V}$ is an equivalence relation on $\V \Gamma$ and $\sim_{\A}$ is an equivalence relation on $\A\Gamma$ such that if $a\sim_{\A}b$ then $o(a)\sim_{\V}o(b)$ and $t(a)\sim_{\V}t(b)$, and if $a\sim_{\A}b$ then $\ba\sim_{\A}\bb$.   Given an equivalence relation $\sim_{\V}$ on the vertex set $\V\Gamma$ one can define an equivalence relation $\sim_{\A}$ on the arc set by saying that two arcs $a, a'$ are equivalent, $a\sim_{\A}  a'$ if  $o(a)\sim_{\V} o(a')$ and $t(a)\sim_{\V}t(a')$ and thus get a graph equivalence relation.

If $\sim$ is a graph equivalence relation on a graph $\Gamma$, then the \emph{quotient graph} $\Gamma/\!\!\sim$ has the set of equivalence classes of $\sim_{\V}$ as a vertex set and the set of equivalence classes of $\sim_{\A}$ as an arc set.  Furthermore, if $a$ is an arc in $\Gamma/\!\!\sim$ then find an arc $b$ in $\Gamma$ belonging to the equivalence class $a$ and define $o(a)$ as the equivalence class of $o(b)$ and $t(a)$ as the equivalence class of $t(b)$.  The conditions above on $\sim_{\A}$ guarantee that the maps $o, t: \A(\Gamma/\!\!\sim)\to \V (\Gamma/\!\!\sim)$ and the arc reversal map are well defined.  

Suppose that $B$ is a set of arcs in $\Gamma$ such that if $a\in B$ then $\ba\in B$.  Consider the graph $\Delta$ that has $\V\Gamma$ as a vertex set and $B$ as a set of arcs.  Define an equivalence relation $\sim_{\V}$ on $\V\Gamma$ such that $u\sim v$ if and only if $u$ and $v$ are in the same connected component of $\Delta$.  Let $\sim_{\A}$ denote the trivial equivalence relation on $\A\Gamma\setminus B$.  The graph $\Gamma/B=(\Gamma\setminus B)/\!\!\sim$ is the graph one gets by \emph{contracting} the arcs in $B$.   For a subgraph $\Gamma'$ in $\Gamma$ we define $\Gamma/\Gamma'$ to denote the graph $\Gamma/\A\Gamma'$.  

If $\Gamma$ and $\Delta$ are digraphs then a {\em digraph morphism} $g\colon  \Gamma\to\Delta$ consists of a pair of maps $g_\V\colon \V\Gamma\to \V\Delta$ and $g_\A\colon \A\Gamma\to \A\Delta$ such that  $g_\V(o(a))=o(g_\A(a))$ and $g_\V(t(a))=t(g_\A(a))$. To simplify notation, in what follows we use $g$ to denote both $g_{\V}$ and $g_{\A}$. 
If $\Gamma$ and $\Delta$ are graphs then we call a digraph morphism a {\em graph morphism} if additionally $\rev{g(a)}=g(\ba)$ for every arc $a\in \A\Gamma$; note that if the graphs $\Gamma$ and $\Delta$ are simple, then this is always satisfied. 
An {\em automorphism} $g$ of a graph or digraph $\Gamma$ is a bijective (that is, both $g_\V$ and $g_\A$ are bijective) morphism $\Gamma\to\Gamma$. An automorphism maps any self-reverse loop to a self-reverse loop.  The automorphisms of $\Gamma$ form the {\em automorphism group} $\aut(\Gamma)$.  Morphisms and automorphisms of digraphs are defined similarly.  In the case that $\Gamma$ is a digraph and $\Gamma/\!\!\sim$ is a quotient graph as defined above then the map $p:\Gamma\to \Gamma/\!\!\sim$ given by the natural projections $p_{\V}: \V\Gamma\to \V\Gamma/\!\!\sim_\V$ and $p_{\A}: \A\Gamma\to \A\Gamma/\!\!\sim_\A$ is a digraph morphism.

\subsection{Group actions}\label{SGroup}

For a set $X$ we denote the \emph{symmetric group} on $X$, the group of all permutations of $X$, by $\sym(X)$.   %In the case that our set $X$ is finite with $n$ elements the symmetric group will be denoted by $S_n$. 
A {\em group action} $(G, X)$ consists of a group $G$, a (non-empty) set $X$ and a homomorphism $\rho\colon G\to \sym(X)$.  For $g\in G$ and $x\in X$ we write $g(x)=(\rho(g))(x)$.
A group action is \emph{faithful} if the kernel of the action, the subgroup $K=\{g\in G\mid g(x)=x\mbox{ for all }x\in X\}$, is trivial, or equivalently, if the homomorphism $\rho$ is injective.  A faithful group action $(G, X)$ is called a {\em permutation group action}; in this case we call $G$  a {\em permutation group on} $X$ and think of $G$ as a subgroup of $\sym(X)$. 

The \emph{orbit} of a point $x\in X$ is the set $Gx= \{g(x) \mid g \in G\}$.
The action is said to be {\em transitive} if $Gx = X$ for one, and hence every, $x \in X$ or, in other words, the action is transitive if for any two points $x, y\in X$ there exists an element $g\in G$ such that $g(x)=y$.  The {\em stabiliser} of $x\in X$ is the subgroup $G_x=\{g\in G\mid g(x)=x\}$.   For a set $Y\subseteq X$  the {\em pointwise stabiliser} of $Y$ is 
$G_{(Y)}=\{g\in G\mid g(y)=y\mbox{ for all }y\in Y\}$ and the {\em setwise stabiliser} is $G_{\{Y\}}=\{g\in G\mid g(Y)=Y\}$. 

 In the case the action of $G$ on $X$ is not faithful we can regard $G^X=G/G_{(X)}$ as a permutation group on $X$ and say that it is the {\em permutation group on $X$ induced by $G$}.  For a subset $Y$ of $X$ we define $G^{Y}$ as the permutation group on $Y$ induced by $G_{\{Y\}}$.

Two group actions $(H,Y)$ and $(G, X)$ are \emph{isomorphic} if there exists a group isomorphism $\varphi\colon  H\to G$ and bijective map $\Phi\colon  Y\to X$ such that $\varphi(h)(\Phi(y))=\Phi(h(y))$, for every $h\in H$ and every $y\in Y$, that is, the diagram in  Figure~\ref{fig:isomorhisms} commutes. 
\begin{figure}
\centering
    \begin{tikzcd}[column sep=large]
        {Y}
        \arrow [r,"h"]
        \arrow [d,"\Phi"]
        &
        {Y} 
        \arrow [d,"\Phi"]
        \\
        {X}
        \arrow [r,"g=\varphi(h)"]
        &
        {X}
    \end{tikzcd}
        \caption{Diagram describing isomorphisms of group actions.}
    \label{fig:isomorhisms}
\end{figure}

Note that if $h\in H$ then 
\[
    \varphi(h) (x) = \varphi(h) (\Phi \Phi^{-1}(x)) = \Phi(h(\Phi^{-1}(x))) = \Phi  h   \Phi^{-1}(x)
\]
for every $y \in Y$.  Thus $\varphi(h)$ and $\Phi  h   \Phi^{-1}$ induce the same permutation on $X$.  In particular, if the group action $(G,X)$ is faithful then the map $\varphi$ is completely determined by $\Phi$.  We call the pair $(\varphi, \Phi)$ (or,  sometimes just the map $\Phi$) a \emph{permutation isomorphism}.  

Let $(H,Y)$ and $(G, X)$ be group actions.  An {\em embedding of group actions} $(H,Y)\hookrightarrow(G,X)$ is a 
permutation isomorphism $(\phi, \Phi)$ from $(H,Y)$ to $(G^{\Phi(Y)},\Phi(Y))$.  Since the group action $(G^{\Phi(Y)},\Phi(Y))$ is faithful we see that the permutation isomorphism is completely determined by $\Phi$, and $\Phi$ can be thought of as an embedding $Y\hookrightarrow X$.   

The concept of an embedding of permutation group actions can also be explained in the following way:  Suppose $\Phi: Y\hookrightarrow X$ is an embedding.  Then $\Phi$ defines an embedding of group actions if given $h \in H$ there exists $g = \varphi(h) \in G^{\Phi(Y)}$ such that the diagram in Figure~\ref{fig:isomorhisms} commutes and, conversely, given $g$ in the setwise stabiliser of $\Phi(Y)$ in $G$ we can find some $h \in H$ such that the diagram commutes.

When $G$ is a group acting by automorphisms on a graph $\Gamma$ we define $\Gamma/G$ as the quotient graph $\Gamma/\!\!\sim$ where $\sim_{\V}$ is the equivalence relation on $\V\Gamma$ that has the orbits of $G$ on $\V\Gamma$ as equivalence classes and $\sim_{\A}$ is the equivalence relation on $\A\Gamma$ that has the orbits of $G$ on $\A\Gamma$ as equivalence classes.   Together $\sim_{\V}$ and $\sim_{\A}$ form a graph equivalence relation.

\subsection{Topological groups and the permutation topology}\label{STopology}

When $G$ is a group acting on a set $X$ it possible to endow $G$ with the {\em permutation topology}, see for instance \cite{Woess1991} and \cite{Moller2010}.   In this topology the  family of all subgroups of the form $G_{(F)}$, where $F$ is a finite subset of $X$, forms a neighbourhood basis for the identity.  A subgroup is thus open if and only if it contains the pointwise stabilizer of a finite subset of $X$.  Another way to describe this topology is to think of $X$ as having the discrete topology and the permutation topology as the compact-open topology on $G$.  The compact-open topology has the property that the action map $G\times X\to X; (g, x)\mapsto g(x)$ is continuous.   If the action of $G$ on $X$ is faithful then $G$ is totally disconnected.  

A sequence $\{g_i\}$ of elements in $G$ converges to an element $g\in G$ in the permutation topology if and only if for each element $x\in X$ there is a number $N_x$ such that if $i\geq N_x$ then $g_i(x)=g(x)$.  We say that $G$ is {\em closed in the permutation topology} if $G^X$ is closed in $\sym(X)$ with respect to the permutation topology on $\sym(X)$.  The topological properties of compactness and co-compactness have natural descriptions in terms of the action of $G$ on $X$,

\begin{lemm}\label{LCompact-cocompact}
Let $G$ be a group acting on a set $X$. Assume $G$ is closed in the permutation topology.  Assume also that all suborbits (the orbits of stabilizers of points) are finite.   Then:
\begin{enumerate}
    \item {\textrm (\cite[Lemma~1]{Woess1991}, cf.~\cite[Lemma~2.2]{Moller2010})}  The stabilizer of a point $x\in X$ is compact.
    \item {\textrm (\cite[Lemma~2]{Woess1991}, cf.~\cite[Lemma~2.2]{Moller2010})}  A subset $A$ in $G$ has compact closure if and only if all orbits of $A$ are finite.
    \item {\textrm (\cite[Proposition~1]{Nebbia2000}, cf.~\cite[Lemma~2.3]{Moller2010})}  Assume that the action of $G$ on $X$ is transitive A subgroup $H$ in $G$ is co-compact, i.e.~$G/H$ is compact, if and only if $H$ has only finitely many orbits on $X$.
\end{enumerate}
\end{lemm}

\begin{remark}
Parts 1 and 2 of the above lemma are stated in \cite{Woess1991} and \cite{Moller2010} with the additional assumption that the action is transitive, but that assumption is not used in the proofs.  Also note that part 1 is stated in \cite{Woess1991} with the assumption that $X$ is countable, but the proof in \cite{Moller2010} shows that this assumption is not needed.
\end{remark}

It follows that if $G$ is a closed subgroup of the automorphism group of a connected locally finite graph $\Gamma$, then $G$ with the permutation topology is a totally disconnected, locally compact group.   

\subsection{Wreath products}

Consider permutation group actions $(G, X)$ and $(H,A)$.  Define $B=\prod_{\xi\in X} H_\xi$ where $H_\xi= H$ for all $\xi\in X$.  
%$H_\xi= H$? or $H_\xi \simeq H$? 
The group $G$ has an obvious action on the set $X\times A$ such that if $g\in G$ and $(x,a)\in X\times A$ then $g(x,a)=(g(x),a)$ and the group $B$ also has a permutation action on $X\times A$, such that $(h_\xi)_\xi(x,a)=(x, h_x(a))$.  The \emph{(unrestricted) wreath product} of the two permutation group actions is defined as the subgroup of $\sym(X\times A)$ generated by $G$ and $B$.  The wreath product can also be defined as a semidirect product $G\wre_X H=G\ltimes B$ where $g(h_\xi)_\xi g^{-1}=(h'_{\xi})_\xi$ with $h'_\xi=h_{g^{-1}(\xi)}$.   Thus we can think of elements in $G\wre_X H$ as being represented by a pair $[g; b]$ where $g\in G$ and $b\in B$.

Let $G$ be a permutation group acting on a set $X$.  An equivalence relation $\sim$ on $X$ is called a \emph{$G$-congruence} if whenever $x,y\in X$ and $g\in G$ then $x\sim y$ if and only if $g(x)\sim g(y)$.  Define $\tX$ as the set of equivalence classes.  The group $G$ has a natural action on $\tX$ and we define $\tG$ as the permutation group that $G$ induces on $\tX$.   For $g\in G$ let $\tg$ denote the permutation $g$ induces on $\tX$.  Assume that $\tG$ acts transitively on $\tX$.   Let $A$ denote one of the $\sim$-classes and let $\xi_A$ be the corresponding element in $\tX$.  Set $H=G^A$. 

The next step is to show, following \cite[p.~6--7]{Neumann1976}, that $G$ embeds into the wreath product $\tG\wre_{\tX} H$.  
For each element $\xi\in \tX\setminus\{\xi_A\}$ choose an element $\gamma_\xi$ in $G$ such that $\gamma_\xi(\xi)=\xi_A$ and set also $\gamma_{\xi_A}=1$.   For an element $\xi\in \tX$ note that $\gamma_{\tg(\xi)}g\gamma_\xi^{-1}$ is an element in $G_{\{A\}}$ and we define $h_\xi$ as its restriction to $A$.  It is left to the reader to check that the map $G\to \tG\wre_{\tX} H$ defined by $g\mapsto [\tg; (h_\xi)_\xi]$ is a group embedding.   If $x\in X$ we let $\xi_x$ be the equivalence class containing $x$ and set $a_x=\gamma_{\xi_x}^{-1}(x)$.  The map $\iota\colon X\to \tX\times A; x\mapsto (\xi_x, a_x)$ is bijective and allows us to identify $X$ with $\tX\times A$.  Via the embedding above we see that the group $G$ acts on both $X$ and $\tX\times A$.  It is easy to show that $\iota(g(x))=g(\iota(x))$ for every $g\in G$ and every $x\in X$ and the two actions are therefore isomorphic.    

We also see that given an element $h\in H$ and two elements $\xi, \xi'\in \tX$ there always is an element $g\in G$ such that for every $a\in A$ we have $g(\xi,a)=(\xi', h(a))$.    

In what follows we deal with a group $G$ acting on a set $X$ with a $G$-congruence but we will not assume that the action on the set of congruence classes is transitive (and thus the action on $X$ may be intransitive).  

Let $G$ be a permutation group acting on a set $X$.  Denote the family of $G$-orbits by $\{O_i\}_{i\in I}$.  For each $i\in I$ let $(H_i, A_i)$ be a permutation group action.  Set $Y=\bigsqcup_{i\in I} O_i\times A_i$.  Define \(W=G\wre_{X} (H_i)_{i\in I}\) as the permutation group acting on $Y$ that is generated by the action of $G$ in the first coordinate and the action of $\prod_{i\in I} \big(\prod_{O_i} H_i\big)$ in the second coordinate.   The group $W$ can be called an \emph{inhomogeneous wreath product}.  The construction described here can also be found in \cite{KlavikNedelaZeman2022}.

Conversely, suppose that we are given a permutation action of a group $G$ on a set $X$ and that we have a $G$-congruence on $X$ but the action of $G$ on the congruence classes is not transitive.  Let $\tX$ denote the set of congruence classes and let $\tG$ denote the action of $G$ on the set of equivalence classes.  Suppose $(O_i)_{i\in I}$ denotes the set of orbits of $\tG$ on $\tX$ choose a family of equivalence classes $(A_i)_{i\in I}$ containing exactly one representative from each $G$-orbit $O_i$.  Let $H_i=G^{A_i}$ be the permutation group $G$ induces on $A_i$.  The above argument now shows that there is an embedding of $G$ into $\tG\wre_{\tX} (H_i)_{i\in I}$.  For an element $\xi$ in $\tX$ that belongs to $O_i$ choose an element $\gamma_\xi\in G$ such that $\xi=\gamma_\xi A_i$.
An element $g\in G$ can now be represented as a $[\tg; (g_\xi)_{\xi\in \tX}]$ where $\tg$ is the permutation $g$ induces on $\tX$ and for $\xi\in O_i$ we define $g_\xi$ as  the restriction of $\gamma_{g(\xi)}^{-1}g\gamma_\xi$ to $A_i$ and note that $g_\xi$ belongs to $H_i$.

\section{Graphs of group actions and their universal groups}\label{SDefinition}

%Parts of the intruductory discussion should probably be in the introduction to the paper.

In this section the main ingredients of the theory of graphs of group actions are introduced. First there is the concept of a \emph{graph of group actions}, then \emph{scaffoldings} and finally \emph{universal groups of graphs of group actions}.  

%The central objects of this theory are graphs of group actions. From a  \ggawo\ one constructs its universal group which acts on a tree-like graph called a scaffolding; this action also gives rise to an action of a tree. By (imperfect) analogy, in Bass-Serre theory we may think of the fundamental group of a graph of groups acting on its Cayley graph; this graph is tree-like in the sense that the cosets of the vertex groups in the graph of groups can be seen as the vertices of a tree.

\subsection{Graphs of group actions}

\begin{defi}\label{DGraph-of-group-action}
A \emph{graph of group actions without self-reverse arcs} (or a \emph{\ggawo}) consists of
\begin{enumerate}[label=(\alph*)]
    \item a connected graph $\Delta$ without any self-reverse arcs, called the \emph{base graph},
    \item for every vertex $v \in \V\Delta$ a permutation group action $(G(v),X(v))$, called \emph{the action at the vertex} $v$, and
    \item for every arc $a \in \A \Delta$
    \begin{itemize}
    \item a permutation group action $(H(a),Y(a))$, called  \emph{the action at the arc} $a$, such that $(H(\ba),Y(\ba)) = (H(a),Y(a))$, and 
    \item an embedding  $(\phi_a, \Phi_a)\colon (H(a),Y(a)) \hookrightarrow (G(t(a)),X(t(a)))$ of permutation actions (the embedding is fully determined by the embedding $\Phi_a: Y(a)\hookrightarrow X(t(a))$, as explained in Section~\ref{SGroup}).
    \end{itemize}
\end{enumerate}
\end{defi}

%To avoid pathological cases it is assumed that the sets $X(v)$ and $Y(a)$ are non-empty. 
Throughout this paper, when talking about a \ggawo\ $\cal G$ we will always assume that
\[{
\cal G}=(\Delta, (G(v), X(v))_{v\in \V\Delta}, (H(a), Y(a))_{a\in \A\Delta}, (\phi_a, \Phi_a)_{a\in \A\Delta})
\]
unless explicitly stated otherwise.

Let $a$ be an arc in $\Delta$.  A translate of the set $\Phi_a(Y(a))$ under the action of $G(t(a))$ on $X(t(a))$ is called an \emph{$a$-adhesion set}.  If $a$ is not relevant, or clear from the context, we omit it and write \emph{adhesion set} instead of $a$-adhesion set. A bijection $\Psi$ from $Y(a)$ to an $a$-adhesion set is called an \emph{$a$-adhesion map} if there is $\gamma \in G(t(a))$ such that $\Psi = \gamma \Phi_a$ (i.e.\ the diagram in Figure \ref{fig:diagram-adhesionmap} commutes).

\begin{remark}
In the definition of a \ggawo\ it is assumed that the base graph has no self-reverse arcs.  This restriction will be removed in Section~\ref{SInversions}.
\end{remark}

\begin{figure}
    \centering
    \begin{tikzcd}[column sep=large]
        {Y(a)}
        \arrow [r,hook,"\Psi"]
        \arrow [dr,bend right,hook,"\Phi_a"]
        &
        {X(t(a))}  
        \\
        &
        {X(t(a))}
        \arrow [swap,u,"\gamma"]
    \end{tikzcd}
    \caption{If $\gamma \in G(t(a))$, then $\Psi$ in this commutative diagram is an $a$-adhesion map and $\Psi(Y(a))$ is an $a$-adhesion set.}
    \label{fig:diagram-adhesionmap}
\end{figure}

\subsection{Scaffoldings}

The universal group of a \ggawo\ $\cG$ is defined in terms of an action on a graph which we call a \emph{scaffolding} for $\cG$.  Below is the definition of a scaffolding.  A construction demonstrating the existence of scaffoldings will follow in Section~\ref{sec:canonical}. 

Suppose $\Sigma$ is a graph and $\sim$ an equivalence relation on $\V\Sigma$. Extend the equivalence relation $\sim$ to $\A\Sigma$ by saying that if $a, a'\in \A\Sigma$ then $a\sim a'$ if and only if $o(a)\sim o(a')$ and $t(a)\sim t(a')$.
In what follows,  $[v]$ is used to denote the equivalence class of a vertex $v$ in $\Sigma$ as a subset of $\V\Sigma$, but when the equivalence class $[v]$ is thought of as a vertex in $\Sigma/\!\!\sim$ it will be denoted by $\llbracket v\rrbracket$.  Similarly $[a]\subseteq \A\Sigma$ denotes the equivalence class of an arc $a$ in $\Sigma$, but $\bun{a}$ denotes the corresponding arc in $\Sigma/\!\!\sim$.  An equivalence class $[v]$ for $v\in \V\Sigma$ is called a \emph{vertex bundle} and if $a\in\A\Sigma$ the equivalence class $[a]$ is called an \emph{arc bundle}.
We will use $o$ and $t$ to denote the origin and terminus maps on both $\Sigma$ and $\Sigma/\!\!\sim$.  If $a\in \A\Sigma$ then $o([a])=\{o(b)\mid b\in [a]\}\subseteq \V\Sigma$, but $o(\llbracket a\rrbracket)=\llbracket o(a)\rrbracket$ and is a vertex in $\Sigma/\!\!\sim$.

\begin{defi}  \label{DScaffolding}
    A \emph{scaffolding} for a \ggawo\  ${\cal G}$ is a graph  $\Sigma$ together with a graph equivalence relation $\sim$ on $\V\Sigma$ such that: 
\begin{enumerate}[label=(\roman*)]
    \item The quotient graph $T_\Sigma=\Sigma/\!\!\sim$ is a tree. 
    \item Each vertex bundle $[v]$ in $\Sigma$ with respect to $\sim$ induces an empty graph (i.e.\ a graph with no arcs).
    \item There is a graph morphism $\pi \colon \Sigma/\!\!\sim \to \Delta$.
    \item  There exists a map $$p\colon \V\Sigma\cup\A\Sigma \to \left(\bigcup _{v \in \V\Delta} X(v)\right)\cup \left(\bigcup _{a \in \A\Delta} Y(a)\right),$$ called a \emph{legal colouring} such that for every vertex $\llbracket v\rrbracket\in \V T_\Sigma$ and every arc $\llbracket a\rrbracket\in \A T_\Sigma$: 
    \begin{itemize}
        \item The restriction of $p$ to $[v]$ is a bijection $p_{\llbracket v\rrbracket}: [v]\to X(\pi(\llbracket v\rrbracket))$.
        \item The restriction of $p$ to $[a]$ is a bijection $p_{\llbracket a\rrbracket}: [a]\to Y(\pi(\llbracket a\rrbracket))$.
        \item For $y \in [a]$ we have $p_{\llbracket a\rrbracket}(y) = p_{\llbracket \ba\rrbracket} (\rev y)$.
        \item The unique map $\Psi_{\llbracket a\rrbracket}:Y(\pi(\llbracket a\rrbracket))\to X(\pi(t(\llbracket a\rrbracket)))$ satisfying $\Psi_{\llbracket a\rrbracket}p(b)=pt(b)$ for every $b\in [a]$ is a $\pi(\llbracket a\rrbracket)$-adhesion map.
        \item Suppose $\llbracket v \rrbracket \in \V T$ and $b$ is an arc in $\Delta$ with $t(b)=\pi(\llbracket v\rrbracket)$.  If $S$ is a $b$-adhesion set in $X(\pi(\llbracket v\rrbracket))$  then there is exactly one arc $\llbracket a\rrbracket \in \A T_\Sigma$ such that $\pi(\llbracket a\rrbracket) = b$, $t(\llbracket a\rrbracket)=\llbracket v\rrbracket$ and $\Psi_{\llbracket a\rrbracket} p([a]) = S$.
    \end{itemize}
\end{enumerate}
\end{defi}

\begin{remark}
The map $\Psi_{\llbracket a\rrbracket}$ above always exists and is unique because the map $p_{\bun a}$ is a bijection.
\end{remark}

\begin{remark}
    \label{rmk:arcbundlematching}
    Every arc bundle in a scaffolding connects two different vertex bundles and thus corresponds to an edge of the tree $T_\Sigma$. Since adhesion maps are embeddings of actions at arcs into actions at vertices, the second last bullet point in the above definition moreover implies that each arc bundle is a matching between two subsets of the respective vertex bundles (corresponding to adhesion sets).
\end{remark}

\begin{remark}
    Intuitively we may think of the first two conditions in in Definition \ref{DScaffolding} as the defining properties of the scaffolding as a relational structure. The maps $\pi$ and $p$ in the last two conditions connect this structure to the given graph of group actions, ensuring that the local (in terms of $\sim$-equivalence classes) structure of $\Sigma$ corresponds to the structure given by the graph of group actions.

The conditions imply that the diagram in Figure \ref{fig:diagram-legallabelling} commutes for every $\llbracket a\rrbracket \in \A T_\Sigma$. This ensures the compatibility of the legal colouring of the arcs with the legal colouring of the vertices.
\begin{figure}
    \centering
    \begin{tikzcd}[column sep=large]
        {[v]}
        \arrow [d,"p_{\llbracket v\rrbracket}"]
        &
        {[a]}
        \arrow [d,"p_{\llbracket a\rrbracket}"]
        \arrow [l,swap,"t"]
        \arrow [<->, r,"\rev{\,\cdot\,}"]
        &
        {[\ba]}   
        \arrow [d,"p_{\llbracket \ba\rrbracket}"]
        \arrow [r,"t"]
        &
        {[w]}
        \arrow [d,"p_{\llbracket w\rrbracket}"]
        \\
        X(\pi(\llbracket v\rrbracket))
        &
        Y(\pi(\llbracket a\rrbracket)))
        \arrow [l,swap,hook',"\Psi_{[a]}"]
        \arrow [<->, r,"="]
        &
         Y(\pi(\llbracket \ba\rrbracket))
        \arrow [r,hook,"\Psi_{[\ba]}"]
        &
        X(\pi(\llbracket w\rrbracket))
    \end{tikzcd}
    \caption{The compatibilty conditions for a legal colouring at the arc $\llbracket a\rrbracket \in \A T_\Sigma$ with $t(\llbracket a\rrbracket) = \llbracket v\rrbracket$, $o(\llbracket a\rrbracket) = \llbracket w\rrbracket$ say that the above diagram commutes. The arrow marked with $\rev{\,\cdot\,}$ denotes the action of reversing an arc and the arrow labelled by $=$ is the identity map.}
    \label{fig:diagram-legallabelling}
\end{figure}
\end{remark}

When discussing a scaffolding $\Sigma$ of a \ggawo\ $\cG$, we will always assume that a map $p$ as above is given.  Sometimes, particularly when we are working simultaneously with two scaffoldings for the same \ggawo, we will refer to a scaffolding as a quadruple $(\Sigma, \sim, \pi, p)$.

The first property in the definition of a scaffolding encapsulates the tree-likeness of the graph $\Sigma$. If a group $G$ of automorphisms of $\Sigma$ preserves the equivalence relation $\sim$, then we immediately get a natural action of $G$ on $T_{\Sigma}$. 

\subsection{The universal group of a graph of group actions}

Suppose $\Sigma$ is a scaffolding for a \ggawo\ $\cG$.  Let us call an automorphism $g$ of the graph $\Sigma$ a \emph{scaffolding automorphism} if both $\sim$ and $\pi$ are invariant under the action of $g$. The set of all scaffolding automorphisms forms a group denoted by $\autsc (\Sigma)$.

Let $g \in \autsc (\Sigma)$ and $v \in \V\Sigma$. Since $g$ preserves $\sim$ we see that $g([v]) = [g(v)]$. Moreover, since $g$ preserves $\pi$ we also see that $\pi(\bun v) = \pi (\bun {gv})$. The \emph{local action} of $g$ at $\bun v$ is the unique permutation $g_{\bun v}\colon X(\pi(\bun v)) \to X(\pi(\bun v))$ such that $pg(w) = g_{\bun v} p(w)$ for every $w \in [v]$; in other words, $g_{\bun v}$ is the unique map which makes the diagram in Figure~\ref{fig:diagram-localaction} commute.  Since $p_{\bun v}$ is a bijection, we see that $g_{\bun v}=pg(p_{\bun v})^{-1}$.

\begin{figure}
    \centering
    \begin{tikzcd}[column sep=huge]
        {[v]}
        \arrow [r,"p"] 
        \arrow [d,"g"] & 
        X(\pi(\llbracket v\rrbracket))   
        \arrow [d,"g_{\bun v}"]
        \\
        {f([v])}
        \arrow [r,"p"] 
        & 
        X(\pi(\llbracket v\rrbracket))  
    \end{tikzcd}
    \caption{Commutative diagram for the local action $g_{\bun v}$ of $g \in \autsc(\Sigma)$ at $\bun v \in \V T_{\Sigma}$.}
    \label{fig:diagram-localaction}
\end{figure}

\begin{defi}
    \label{DUniversalgroup}
    The \emph{universal group} of a \ggawo\ $\cG$, denoted by $\cU(\cG)$, is the group consisting of all scaffolding automorphisms of $\Sigma$ such that $g_{\bun v} \in G(\pi(\bun v))$ for every $\bun v \in \V T_\Sigma$.
\end{defi}
 
Of course we have to show that scaffoldings exist and that $\cU(\cG)$ is indeed a group. This will be done in the next section. Moreover, in Section~\ref{SExtension} it is shown that the universal group a \ggawo\  does not depend on the choice of a scaffolding.

\section{Existence of scaffoldings and universal groups}\label{SExistence}

In this section it is shown that every \ggawo\ $\cG$ has a scaffolding.  This is done by constructing the so-called \emph{canonical scaffolding}. The construction starts with a tree $T$ (a $\Star$-covering tree, as defined in Section~\ref{SStar-covering}, for the augmented base graph $\Delta_+$ defined in Section \ref{SAugmented}) and a graph morphism $\pi \colon T \to \Delta$. Each vertex $v\in \V T$ is then replaced by 
an empty graph whose vertex set is a copy of $X(\pi(v))$ and each edge $\{a, \ba\} \in \E T$ is replaced by a matching between some $\pi(a)$- and $\pi(\ba)$-adhesion sets.

\subsection{The augmented base digraph}\label{SAugmented}

For an arc $a$ in $\Star_\Delta(v)$ define {\em the index of $a$}, denoted by $i(a)$, as the number of distinct $a$-adhesion sets.  That is to say, the index is the number of distinct $G(v)$-translates in $X(v)$ of the set $\Phi_a(Y(a))$ and is thus equal to the index of the subgroup $G(v)_{\{\Phi_a(Y(a))\}}$ in $G(v)$.  Note that we do not assume that the indices $i(a)$ and $i(\ba)$ are equal. Construct a digraph $\Delta_+$ that has the same vertex set as $\Delta$ and the arcs in $\Delta$ are also arcs in $\Delta_+$, but for each arc $a$ we add $i(a)-1$ new arcs, all with the same origin and terminus as $a$.  Thus, $\Delta$ is a subdigraph of $\Delta_+$.   
The digraph $\Delta_+$ is completely determined by $\cG$ and is called the \emph{augmented base digraph} for $\cG$.    

We say that $a\in \A\Delta$ is the \emph{original arc} for itsefl and the $i(a)-1$ arcs added because of $a$.   Define a map $\rho\colon   \A\Delta_+\to \A\Delta$ so that $\rho(b)=a$ if $a$ is the original arc for $b$.  

\subsection{Star-covering trees}\label{SStar-covering}

\begin{defi}\label{DStar-covering-tree}
A {\em $\Star$-covering tree} for a connected digraph $\Gamma$ is a tree $T$ that admits a surjective digraph morphism $\pi\colon  T\to \Gamma$ such that for every $x \in \V T$ the restriction of $\pi$ to $\Star_T(x)$ is a bijection to $\Star_{\Gamma}(\pi(x))$.  A map $T\to \Gamma$ satisfying the above conditions is called a \emph{$\Star$-covering map}. 
\end{defi}

%\textcolor{purple}{RGM:  There are equivalent results in the literature.  Would it be better to quote one of them and cut out a page or two from our paper?  Parts of the proof of Corollary 3.3 could also be omitted.}

%\textcolor{red}{FL: I think the self-contained treatment is more valuable than cutting a page or two; the paper will be very long either way.}

\begin{lemm}\label{LExistence-Star-covering-tree}
Let $\Gamma$ be a connected digraph. There exists a $\Star$-covering tree for $\Gamma$ if and only if for every $a \in \A \Gamma$ there is $a' \in \A \Gamma$ such that $o(a') = t(a)$ and $t(a') = o(a)$ (i.e.\ there exists a \lq\lq reverse\rq\rq\  for every arc).
Moreover, if there exists a $\Star$-covering tree for $\Gamma$, then any two $\Star$-covering trees for $\Gamma$ are isomorphic.
\end{lemm}

\begin{proof}
First recall that a tree $T$ is by definition an undirected graph. Suppose $\pi: T\to \Gamma$ is a digraph morphism.  If $a$ is an arc in $\Gamma$ and $a=\pi(b)$ we set $a'=\pi(\bb)$ and then $o(a') = t(a)$ and $t(a') = o(a)$, i.e.\ $a$ has a \lq\lq reverse\rq\rq.    So, if $a$ is an arc in $\Gamma$ and there is no arc $a'$ in $\Gamma$ such that $o(a') = t(a)$ and $t(a') = o(a)$ then $a$ cannot be contained in the image of a digraph morphism $\pi\colon T \to \Gamma$ and thus a $\Star$-covering tree for $\Gamma$ cannot exist.  Whence the condition that there exists a \lq\lq reverse\rq\rq\ for every arc is necessary.

Assume that there exists a \lq\lq reverse\rq\rq\ for every arc in $\Gamma$.
We use induction to construct a $\Star$-covering tree $T$ and a $\Star$-covering map $\pi$. For a vertex $v$ in $\Gamma$ define $\TStar_\Gamma(v)$ as a digraph that consists of a vertex $v_*$ and for each arc $a$ in $\Star_{\Gamma}(v)$ we have one vertex $v_a$ and one arc $b_a$ such that $o(b_a)=v_a$ and $t(b_a)=v_\ast$.
Define a labelling $\ell$ on the vertices and arcs of $\TStar_{\Gamma}(v)$ so that $\ell(v_\ast) = v$, $\ell(v_a) = o(a)$ and $\ell(b_a) = a$ for every $a \in \Star_\Gamma(v)$.

Choose a vertex $v_0$ in $\Gamma$.  Define $T_0$ as the labelled digraph $\TStar_\Gamma(v_0)$. Our induction hypothesis is that a labelled digraph $T_i$ has been defined (both vertices and arcs labelled) such that the underlying undirected graph is a tree and that if $v$ is a vertex in $T_i$ such that $\Star_{T_i}(v)$ is non-empty then the labelling gives an isomorphism between the subdigraph spanned by $\Star_{T_i}(v)$ and $\TStar_\Gamma(\ell(v))$, but if $u$ is a vertex in $T_i$ such that $\Star_{T_i}(u)$ is empty then there is a unique arc $b$ in $T_i$ such that $o(b)=u$. 

Let $u$ be a vertex of $T_i$ such that $\Star_{T_i}(u)$ is empty, let $b$ be the unique arc in $T_i$ such that $o(b)=u$, and let $v=t(b)$.  Set $u'=\ell(u)$, $b' = \ell(b)$ and $v'=\ell(v) = t(b')$.  Pick an arc $b'' \in \A \Gamma$ such that $o(b'') = v'$ and $t(b'') = u'$.  (In the proof of Corollary~\ref{CCompatible-Star-covering-map} below we return to this point.)  Take a copy $T_{u'}$ of $\TStar_\Gamma(u')$, with the labelling,  and \lq\lq glue\rq\rq\ to $T_i$ identifying $v_{b''}$ with $v$ and $u'_\ast$ with $u$.  Let the new vertices and arcs keep their labels from $T_{u'}$.  Repeat this for every arc $b$ of $T_i$ for which there is no reverse arc in $T_i$ to obtain a digraph $T_{i+1}$.  It is obvious that $T_{i+1}$ satisfies the conditions in the induction hypothesis.

Using this method we construct an increasing sequence $T_0, T_1, T_2, \ldots$ of labelled digraphs. The underlying simple graph of each of these digraphs is a tree. Moreover, for every $b \in \A T_i$ there is $b' \in \A T_{i+1}$ such that $o(b') = t(b)$ and $t(b') = o(b)$, and this arc $b'$ is unique. 
%$o(a)$ is not a leaf in $T_i$ then there exists a unique arc $\ba$ in $T_i$ such that $o(\ba)=t(a)$ and $t(\ba)=o(a)$.    
Define $T$ as the union of this sequence.  From the construction it is clear that $T$ is a tree.

Define a map $\pi \colon T \to \Gamma$ such that if $x$ is a vertex or an arc in $T$ then $\pi(x) = \ell(x)$.  This map is clearly a digraph morphism and if $v$ is a vertex in $T$ then the restriction of $\pi$ to $\Star_T(v)$ induces a bijection $\Star_T(v)\to \Star_\Gamma(\pi(v))$.  To show that $\pi$ is surjective, note that if $v \in \V \Gamma$ appears as a label in $T_i$ then $v$ is in the image of $\pi$ and hence every arc in $\Star_\Gamma(v)$ is also in the image of $\pi$. Since $\Gamma$ is assumed to be connected, this implies that all of $\Gamma$ is in the image of $\pi$.

Hence $\Star$-covering trees for $\Gamma$ exist, and it only remains to show that any two $\Star$-covering trees are isomorphic. Suppose $T$ and $T'$ are two $\Star$-covering trees for $\Gamma$ with $\Star$-covering maps $\pi$ and $\pi'$, respectively. Label each $v \in \V T$ with label $\pi(v)$ and each $v' \in \V T'$ with label $\pi'(v')$ (no need to label the arcs). For $u,v \in \V \Gamma$ denote by $i(u,v)$ the number of arcs in $\Gamma$ with $o(a) = u$ and $t(a) = v$. Since $\pi$ and $\pi'$ are $\Star$-covering maps, every vertex in $T$ and $T'$ with label $v$ has precisely $i(u,v)$ neighbours with label $u$. 
Pick $v_0 \in \V T$ and $v_0' \in \V T'$ with the same label. Think of the vertices $v_0$ and $v_0'$ as the roots of the trees $T$ and $T'$, respectively.  Let $\parent$ and $\parent'$ denote the respective parent maps as in Definition~\ref{Droot}.  Suppose we have defined a label preserving isomorphism $\alpha_i: B_T(v_0, i)\to B_{T'}(v'_0, i)$. Extend $\alpha_i$ to a label preserving map $\alpha_{i+1}: B_T(v_0, i+1)\to B_{T'}(v'_0, i+1)$ using the following procedure: If $v$ is a leaf of the subtree spanned by $B_T(v_0, i)$ then we define $\alpha_{i+1}$ on $N(v)$ (the set of vertices in $T$ adjacent to $v$) by taking any label preserving bijection from $N(v)$ to $N(\alpha_i(v))$ mapping $\parent(v)$ to $\parent(\alpha_i(v))$.  It is not hard to see that in this way we get a label preserving isomorphism $\alpha \colon T \to T'$.
\end{proof}

Now let $\cG$ be a \ggawo.
The next step is to find a $\Star$-covering tree $T$ of $\Delta_+$ which also takes into account the graph structure of $\Delta$. A $\Star$-covering map $\pi_+ \colon T \to \Delta_+$ is said to be \emph{compatible} with $\mathcal G$ if the digraph morphism $\pi=\rho \pi_+\colon  T\to \Delta$ is a graph morphism, in other words, $\pi(\rev a) = \rev{\pi(a)}$ for every $a \in  \A T$. A minor modification of the proof of Lemma \ref{LExistence-Star-covering-tree} gives a $\Star$-covering tree with a compatible $\Star$-covering map.

\begin{coro}
    \label{CCompatible-Star-covering-map}
    For every \ggawo\  $\mathcal G$ there exists a $\Star$-covering tree $T$ of $\Delta_+$ and a $\Star$-covering map $\pi_+ \colon T \to \Delta_+$ which is compatible with $\mathcal G$.
    Moreover, if $\pi_+$ and $\pi_+'$ are two $\Star$-covering maps $T\to \Delta_+$ compatible with $\mathcal G$, then there is an automorphism $\alpha$ of $T$ such that $\rho\pi_+ = \rho\pi_+'\alpha$.
\end{coro}

\begin{proof}
    For the first part, note that one can choose $b''\in\A\Delta_+$ in the proof of Lemma \ref{LExistence-Star-covering-tree} such that $\pi(b'') = \rev{\pi(b)}$. 

    Now we prove the second claim in the corollary.  Let  $\pi_+$ and $\pi_+'$ be two $\Star$-covering maps compatible with $\mathcal G$. Since $\pi_+$ is a $\Star$-covering map, for every $w \in \V T$  with $\pi_+(w) = v$ and every arc $a \in \Star_\Delta(v)$ there are exactly $i(a)$ arcs $b \in \Star_T(w)$ such that $\rho\pi_+(b)=a$. Similarly, for every $w \in \V T$  with $\pi_+'(w) = v$ and every arc $a \in \Star_\Delta(v)$ there are exactly $i(a)$ arcs $b \in \Star_T(w)$ such that $\rho\pi_+'(b)=a$.

    Pick $v_0,v_0' \in \V T$ with $\pi_+(v_0) = \pi_+'(v_0')$. Let $\parent, \parent'$ denote the parent maps defined by rooting $T$ at $v_0$ and $v_0'$, respectively, see Definition~\ref{Droot}.  Set $\alpha(v_0) = v_0'$.  Assume that $\alpha$ has been defined on $B_T(v, n)$ and that $B_T(v_0, n)$ is mapped bijectively to $B_T(v_0', n)$ such that $\rho\pi_+ (x)= \rho\pi_+'\alpha(x)$ for every $x\in B_T(v_0, n)$.  Extend the definition of $\alpha$ to $B_T(v_0, n+1)$ so that if $d_T(v_0, v)=n$ then the restriction of $\alpha$ to $N(v)$ is any bijection from $N(v)$ to $N(\alpha(v))$ such that $\alpha (\parent(v)) = \parent'(\alpha(v))$ and $\rho\pi_+(b) = \rho\pi'_+\alpha(b)$ for every $b\in \Star_T(v)$. It is not hard to see that in this way we get an automorphism $\alpha \colon T \to T$ with $\rho\pi_+ = \rho\pi_+'\alpha$.
\end{proof}

\begin{remark}\label{RCompatible}
\begin{enumerate}[label=(\roman*)]
\item In Section~\ref{SInversions} we remove the restriction that the base graph of a graph of group actions has no self-reverse arcs.  If it so happens in the construction of the $\Star$-covering tree that $\rho(b')$ is a self-reverse arc in $\Delta$ then we can simply choose $b''=b'$.   Thus Corollary~\ref{CCompatible-Star-covering-map} easily generalizes to the case when our base graph has self-reverse arcs.
\item An astute reader familiar with Bass--Serre theory will realize that the method described above could be used to construct the Bass--Serre tree for a graph of groups, see \cite[Section 5.3]{Serre2003} and \cite[Definition 1.16]{Bass1993}.   
The constructions given in \cite{Serre2003} and \cite{Bass1993} start with a graph of groups and use group theory.  Similar concepts as $\Star$-covering trees can also be found in e.g.\ \cite[(5.3)]{Tits1970},  and \cite[Definition 3.4]{ReidSmith2020} and again the constructions can be described as \lq\lq algebraic\rq\rq.   The aim   with the different formulation and the different proof here is to highlight the graph theoretical nature of this concept.  
\end{enumerate}
\end{remark}

The augmented base digraph $\Delta_+$ for a \ggawo\ $\cG$ is uniquely determined by $\cG$.  Thus a $\Star$ covering tree for $\Delta_+$ is uniquely determined by $\cG$.

\begin{defi}
Let $\cG$ be a \ggawo.   A $\Star$-covering tree for $\Delta_+$ is called the \emph{associated tree} for $\cG$ and denoted by $T_{\cG}$.     
\end{defi}
The following lemma relates scaffoldings and associated tree.

\begin{lemm}\label{lem:scaffolding-tree}
Suppose $\Sigma$ is a scaffolding for a \ggawo\  $\cG$.   Then the tree $T_\Sigma$ is a $\Star$-covering tree for $\Delta_+$ and thus $T_{\Sigma}=T_{\cG}$. Moreover, the $\Star$-covering map can be chosen to be compatible with $\cG$.
\end{lemm}

\begin{proof}  To show that $T_\Sigma=T_{\cG}$ we just have to define a $\Star$-covering map $\pi_+: T_\Sigma\to \Delta_+$.  Since $\V\Delta_+ = \V \Delta$ we can define $\pi_+$ on the vertex set of $T_\Sigma$ to be equal to the graph homomorphism $\pi\colon T_\Sigma \to \Delta$ which exists by the definition of a scaffolding.  

Suppose now that $\llbracket a\rrbracket$ is an arc in $T_\Sigma$ and $\llbracket v\rrbracket=t(\llbracket a\rrbracket)$.  Then $p(t([a]))$ is a $\pi(\llbracket a\rrbracket)$-adhesion set in $X(\pi(\llbracket v\rrbracket))$ and we define $\pi_+(\llbracket a\rrbracket)$ as the arc in $\Delta_+$ corresponding to that adhesion set.  It follows from the fifth item in condition (iv) in Definition~\ref{DScaffolding} that the restriction of $\pi_+$ to $\Star_{T}(\llbracket v\rrbracket)$ maps $\Star_T(\llbracket v\rrbracket)$ bijectively to $\Star_{\Delta_+}(\pi(\llbracket v\rrbracket))$.  Since $\Delta_+$ is connected this allows us to conclude that $\pi_+$ is surjective.  Hence, $\pi_+:  T_\Sigma\to \Delta_+$ is a $\Star$-covering map. 

It follows from the definition of $\pi_+$ above that $\pi=\rho\pi_+$.  As $\pi$ is a a graph homomorphism it is also easy to see that $\pi_+$ is compatible with~$\cG$ as claimed.
\end{proof}

\subsection{Canonical scaffoldings}\label{sec:canonical}

We now define a special type of scaffoldings for a \ggawo\ $\cG$, the so-called \emph{canonical scaffoldings},  and thus show that scaffoldings do exist.   

Let $T=T_{\cG}$ be a $\Star$-covering tree for $\Delta_+$ with $\pi_+:T\to \Delta_+$ a $\Star$-covering map compatible with $\cG$, and set $\pi=\rho\pi_+$. 
For an arc $a$ in $\Delta$, pair the arcs in $\rho^{-1}(a)$ with the  $a$-adhesion set in $X(t(a))$, and call an $a$-adhesion map onto the adhesion set paired with $b\in\rho^{-1}(a)$ an $a$-adhesion map \emph{of type b}.   For each arc $b\in \rho^{-1}(a)\setminus\{a\}$ fix an adhesion map $\Phi_b: Y(a)\to X(t(a))$ of type $b$. Choose $\gamma_b \in G(t(a))$ so that $\Phi_b=\gamma_b\Phi_a$.  We may assume that if $a\in \A\Delta$ then $a$ is paired with $\Phi_a(Y(a))$ and $\gamma_a$ is the identity.  The set $\{\gamma_b\}_{b\in \rho^{-1}(a)}$ is a complete set of coset representatives for $G(v)_{\{\Phi_a(Y(a))\}}$ in $G(v)$.   
The \emph{augmented \ggawo} is the digraph $\Delta_+$ with the vertex and arc actions together with the chosen maps $\Phi_b$ for $b\in \A\Delta_+$.

Define $\Sigma$ as a graph with 
\[\V\Sigma=\bigsqcup_{v\in \V\Delta} \pi_+^{-1}(v)\times X(v)\] 
and 
\[\A\Sigma=\bigsqcup_{a\in \A\Delta} \pi_+^{-1}(a)\times Y(a).\] 
For an arc  $(a, y)$ in $\Sigma$ set $o(a, y)=(o(a), \Phi_{\pi_+(\ba)}(y))$ and  $t(a,y)=(t(a), \Phi_{\pi_+(a)}(y))$.      

Define a graph equivalence relation $\sim$ on $\Sigma$ by saying that two elements from $\V\Sigma\cup\A\Sigma$ are equivalent if and only if their first coordinates are equal.  With this definition it is clear that $T=\Sigma/\!\!\sim$. If $v\in \V T$ and $a\in \A T$ the corresponding vertex and arc bundles are $\Sigma_v=\{(v,x)\in \V\Sigma\mid x\in X(\pi(v))\}$ and $\Sigma_a=\{(a,y)\in \A\Sigma\mid y\in Y(\pi(a))\}$, respectively.

Finally, define $$p\colon \V\Sigma\cup\A\Sigma \to \left(\bigcup _{v \in \V\Delta} X(v)\right)\cup \left(\bigcup _{a \in \A\Delta} Y(a)\right)$$ 
such that if $(s,z)$ is in $\V\Sigma\cup\A\Sigma$ then $p(s,z)=z$. In other words, $p$ is the projection onto the second coordinate.

\begin{lemm}\label{lem:canonical-scaffolding}
The graph $\Sigma$ together with the the equivalence relation $\sim$
is a scaffolding for $\cG$ in the sense of Definition~\ref{DScaffolding} (witnessed by the maps $\pi$ and $p$ defined above).  This scaffolding is called a \emph{canonical scaffolding} for $\cG$.  
\end{lemm}

\begin{proof}
Conditions (i), (ii), (iii) and the first three conditions in (iv) in Definition~\ref{DScaffolding} follow immediately. 

Suppose $a$ is an arc in $T$.  
The map $\Psi_{a}:Y(\pi(a))\to X(\pi(t(a)))$ defined by the formula $\Psi_{ a}=pt(p_{a})^{-1}$ is clearly equal to the map $\Phi_{\pi_+(a)}$ and is thus a $\pi(a)$-adhesion map and we have show that the fourth condition in (iv) holds..

Suppose $v$ is a vertex in $T$ and $b$ an arc in $\Delta$ such that $t(b)=\pi(v)$ and that $S$ is a $b$-adhesion set in $X(\pi(v))$.   Then there is a unique arc $c$ in $\Delta_+$ such that $\rho(c)=b$ and $S$ is equal to the image of $\Phi_c$.   In $T$ there is a unique arc $a$ such that $t(a)=v$ and $\pi_+(a) = c$.  Then
$\Psi_{a} p_{a}([a])=\Psi_{a}(Y(\pi(a)))=\Phi_{\pi_+(a)}(Y(\pi(a)))=S$ and we have shown that the last condition in part (iv) of Definition~\ref{DScaffolding} holds.
\end{proof}

Using only the definition of canonical scaffoldings and the fact that $\Star$-covering trees are isomorphic,  it is easy to prove that any two canonical scaffoldings for a given \ggawo\  are isomorphic as graphs. In Section~\ref{SExtension} we show something significantly stronger: any two scaffoldings for a \ggawo\ are isomorphic by an isomorphism that is compatible with the actions at the vertices of the \ggawo. This allows us in many cases to replace an arbitrary scaffolding with a canonical scaffolding.  

\subsection{Acceptable isomorphisms and the universal group}\label{SAcceptable}

\begin{defi}\label{DScaffolding-automorphism}
Let $(\Sigma, \sim, \pi, p)$ and $(\Sigma', \sim', \pi', p')$ be scaffoldings for the same \ggawo\  $\cG$.  A \emph{scaffolding isomorphism} is a graph isomorphism $f \colon \Sigma \to \Sigma'$ such that $x\sim y$ if and only of $f(x)\sim'f(y)$, and $\pi(\llbracket x\rrbracket) = \pi' (\llbracket f(x)\rrbracket)$ for all $x, y\in \V\Sigma\cup \A\Sigma$. A \emph{scaffolding automorphism} is a scaffolding isomorphism $\Sigma\to \Sigma$.   The group of all scaffolding automorphisms is denoted by $\autsc(\Sigma)$.     
\end{defi}

Focussing on the special case of scaffolding automorphisms $\Sigma\to \Sigma$ we see that the equivalence relation $\sim$ is an $\autsc(\Sigma)$-congruence and the group $\autsc(\Sigma)$ has an action on $T_\Sigma$ by automorphisms.     If $h\in \autsc(\Sigma)$ then the induced automorphism $T_{\Sigma}\to T_{\Sigma}$ will be denoted by $\tih$.   

The restriction $p_{\llbracket v\rrbracket}$ of $p$ to a vertex bundle $[v]$ is a bijection $[v]\to X(\pi(\llbracket v\rrbracket))$.  Thus the vertex set of $\Sigma$ can be identified with the set $\bigsqcup_{v\in \V\Delta} \pi^{-1}(v)\times X(v)$ so that $w\in \V\Sigma$ is identified with $(\pi(\bun{w}), p(w))$ (compare with the construction of canonical scaffoldings).   Using this identification, it is possible to think of the group $\autsc(\Sigma)$ as a subgroup of the inhomogeneous wreath product $\autsc(\Sigma)^{\V T}\wre_{\V T}\, (\sym(X(v)))_{v\in \V \Delta}$. 
Hence one can write 
$h = [\tih;(h_{\llbracket v\rrbracket})_{\llbracket v\rrbracket \in \V T_\Sigma}]$, where $h_{\llbracket v\rrbracket}$ is a permutation on $X(\pi(\llbracket v\rrbracket))$ defined by $h_{\llbracket v\rrbracket}(x) =  p(h(\llbracket v\rrbracket,x))$ for every $x\in X(\pi(\llbracket v\rrbracket))$.

The following definition extends this idea to the more general setting of scaffolding isomorphisms.

\begin{defi}\label{DAcceptble}  Let $(\Sigma, \sim, \pi, p)$ and $(\Sigma', \sim', \pi', p')$ be scaffoldings for a \ggawo\ $\cG$.
   Suppose that $f \colon \Sigma \to \Sigma'$ is a scaffolding isomorphism. The \emph{local map} of $f$ at $\llbracket v\rrbracket \in V T_\Sigma$ is the permutation $f_{\llbracket v\rrbracket}$ of $X(\pi(\llbracket v\rrbracket))$ such that $f_{\llbracket v\rrbracket}(p(x)) = p'(f(x))$ for every $x \in [v]$, see Figure \ref{fig:diagram-localmap}.  
   For a scaffolding automorphism $f$ of $\Sigma$ the local map at $\llbracket v\rrbracket$ is called the \emph{local action} at $\llbracket v\rrbracket$.      
      We call a scaffolding isomorphism \emph{acceptable} if the local map at $\llbracket v\rrbracket$ is in $G(\pi(\llbracket v\rrbracket))$ for every $\llbracket v\rrbracket \in \V T_{\Sigma}$. 
\end{defi}
\begin{figure}
    \centering
    \begin{tikzcd}[column sep=huge]
        {[v]}
        \arrow [r,"p"] 
        \arrow [d,"f"] & 
        X(\pi(\llbracket v\rrbracket))   
        \arrow [d,"f_{\llbracket v\rrbracket}"]
        \\
        {f([v])}
        \arrow [r,"p'"] 
        & 
        X(\pi(\llbracket v\rrbracket))  
    \end{tikzcd}
    \caption{Commutative diagram defining the local map $f_{[v]}$ of $f$ at $[v]\in \V T_{\Sigma}$.}
    \label{fig:diagram-localmap}
\end{figure}

\begin{remark}
The restriction of $p_{\bun v}$ of $p$ to a vertex bundle $[v]$ in $\Sigma$ is a bijection.   Thus the condition that $f_{\llbracket v\rrbracket}(p(x)) = p'(f(x))$ for every $x \in [v]$, determines $f_{\llbracket v\rrbracket}$ completely and $f_{\llbracket v\rrbracket}=p'f(p_{\llbracket v\rrbracket})^{-1}$.
\end{remark}

Suppose $\Sigma, \Sigma'$ and $\Sigma''$ are scaffoldings for a \ggawo\ $\cG$ and $h:\Sigma\to \Sigma'$ and $h'=\Sigma'\to \Sigma''$ are scaffolding isomorphisms.  Let $p, p'$ and $p''$ denote  legal colourings on $\Sigma, \Sigma'$ and $\Sigma''$, respectively.   Then the local map of $h'h$ at $\bun{v}\in \V T_\Sigma$ is 
\[(h'h)_{\bun{v}}=p''(h'h)(p_{\bun{v}})^{-1}=\big(p''h'(p'_{h(\bun{v})})^{-1}\big)\big(p'h(p_{\bun{v}})^{-1})=h'_{h(\bun{v})}h_{\bun{v}}.\]
If $h$ and $h'$ are both acceptable, then $h_{\bun{v}}\in G(\pi(\bun{v})$ and $h'_{h(\bun{v})}\in G(\pi(h(\bun{v}))=G(\pi(\bun{v})$ and we see that $(h'h)_{\bun{v}}\in G(\pi(\bun{v})$.
Thus the composition of acceptable scaffolding isomorphism yields an acceptable scaffolding isomorphism.  It is also easy to see that the inverse of an acceptable isomorphism is acceptable. Since the identity map $\Sigma\to \Sigma$ is an acceptable scaffolding isomorphism, this confirms that the universal group $\cU(\cG)$ defined in Definition~\ref{DUniversalgroup} is a group under the operation of composition of mappings.  The argument above is the same as one would use to prove that the Burger-Mozes groups defined in \cite[Section~3.2]{BurgerMozes2000} are indeed groups.

\begin{prop}    
    The set $\cU(\cG)$ of acceptable scaffolding automorphisms is a group.
\end{prop}

If we think of the group of scaffolding automorphisms as a subgroup of the inhomogeneous wreath product
$\autsc(\Sigma)^{\V T}\wre_{\V T}\, (\sym(X(v)))_{v\in \V \Delta}$ then $\cU(\cG)$ is a subgroup of $\cU(\cG)^{\V T}\wre_{\V T}\, (G(v))_{v\in \V \Delta}\leq \autsc(\Sigma)^{\V T}\wre_{\V T}\, (\sym(X(v)))_{v\in \V \Delta}$.

In the next section  we show that that $\cU(\cG)$ does not depend on which scaffolding $(\Sigma,\sim,\pi,p)$ is used to define it. More precisely, we prove that any two scaffoldings $\Sigma$ and $\Sigma'$ for a \ggawo\ $\cG$ are isomorphic by an acceptable isomorphism and thus their universal groups and the actions of these groups on $\Sigma$ and $\Sigma'$ are isomorphic.  

In the example below scaffolding automorphisms and acceptable automorphisms are explained in the case when the scaffolding $\Sigma$ is a canonical scaffolding.

\begin{example}
Suppose that $\cG$ is a \ggawo\ and $\Sigma$ is a canonical scaffolding as defined in Section~\ref{sec:canonical}.  Let $g$ be a scaffolding automorphism of $\Sigma$ and denote the corresponding automorphism of $T=T_\Sigma$ by $\tilde{g}$.  If $v$ is a vertex in $T$, then $g$ maps the vertex bundle $\Sigma_v$ to the vertex bundle $\Sigma_{\tilde{g}(v)}$ and the local action of $g$ at the vertex bundle is the permutation $g_{v}$ of $X(\pi(v))$ such that $g(v,x)=(\tilde{g}(v), g_{v}(x))$ for all $x\in X(\pi(v))$.  If each of the permutations $g_{v}$ belongs to $G(\pi(v))$ then $g$ is acceptable.  

A scaffolding automorphism also maps each arc bundle to an arc bundle.   For an arc $a$ in $T$ define a permutation $g_{a}$ of $Y(\pi(a))$ such that if $y\in Y(\pi(a))$ then  $g(a,y)=(\tilde{g}(a), g_a(y))$.   Set $v=t(a)\in \V T$ and $b=\tg(a)$.  If $(a,y)\in \Sigma_a$ then $t(a,y)=(v, \Phi_{\pi_+(a)}(y))$.  There are elements $\gamma_a, \gamma_b\in G(\pi(v))$ such that $\Phi_{\pi_+(a)}=\gamma_a\Phi_{\pi(a)}$ and $\Phi_{\pi_+(b)}=\gamma_b\Phi_{\pi(b)}=\gamma_b\Phi_{\pi(a)}$.  For an arc bundle $\Sigma_a$ define $t_{\Sigma_a}: \Sigma_a\to t(\Sigma_a)$ as the map one gets by restricting $t$ to $\Sigma_a$.  Note that $t_{\Sigma_a}$ is invertible.  We can now determine $g_a(y)$ by looking at $g(t(a,y))$. Thus 
\begin{align*}
g_a(y)&=pgp_{a}^{-1}(y)\\
&=pg(a,y)\\
&=p(t_{\Sigma_b})^{-1}gt(a,y)\\
&=p(t_{\Sigma_b})^{-1}g(v,\Phi_{\pi_+(a)}(y))\\
&=p(t_{\Sigma_b})^{-1}(\tg(v),g_v(\Phi_{\pi_+(a)}(y)))\\
&=p(b, (\Phi_{\pi_+(b)})^{-1}\tg(v),g_v(\Phi_{\pi_+(a)}(y)))\\
&=(\Phi_{\pi_+(b)})^{-1}g_v(\Phi_{\pi_+(a)}(y))\\
&=(\gamma_b\Phi_{\pi(a)})^{-1}g_v\gamma_a\Phi_{\pi_(a)}(y)\\
&=(\Phi_{\pi(a)})^{-1}\gamma_b^{-1}g_v\gamma_a\Phi_{\pi_(a)}(y).
\end{align*}
So $g_a=(\Phi_{\pi(a)})^{-1}\gamma_b^{-1}g_v\gamma_a\Phi_{\pi(a)}$. If $g$ is acceptable, then $\gamma_b^{-1}g_v\gamma_a\in G(\pi(v))$.  As $\Phi_{\pi(a)}$ is a permutation group embedding of $(G(\pi(a)), Y(\pi(a)))$ into $(G(\pi(v)), X(\pi(v)))$ it follows that $g_a\in G(\pi(a))$.  
\end{example}

   More generally, any scaffolding isomorphism $f: \Sigma\to \Sigma'$ maps arc bundles to arc bundles.  One can define the \emph{local map} at an arc bundle $\llbracket a\rrbracket\in \A T_{\Sigma}$ in the same way as the local map at a vertex bundle, i.e.\ as the unique permutation $f_{\bun{a}}$ of $Y(\pi(a))$ such that $f_{\bun{a}}(p(x))=p'(f(x))$.  When $f$ is a scaffolding automorphism $\Sigma\to \Sigma$ we call $f_{\bun{a}}$ the \emph{local action} at $\llbracket a\rrbracket$.     

   The argument in the example above can be extended to show the following.
   
\begin{lemm}\label{LArcBundle}
Suppose $\Sigma$ and $\Sigma'$ are scaffoldings for a \ggawo\ $\cG$.
    If $f: \Sigma\to \Sigma'$ is an acceptable scaffolding isomorphism then the local map $f_{\bun{a}}=p'f(p_{\bun{a}})^{-1}$ at each arc bundle $\bun{a}$ is in $G(\pi(a))$. 
\end{lemm}

\section{The extension property}\label{SExtension}

The \emph{extension property} encapsulated in Theorem \ref{thm:acceptable-isomorphism-ggawo} below is the key to proving many of the basic properties of \ggawo's and their universal groups.   The most fundamental application is to show that any two scaffoldings for a \ggawo\ $\cG$ are isomorphic by means of an acceptable isomorphism.   From this it follows that the universal group of a \ggawo\ does not depend on which scaffolding is used to define it. 

For a scaffolding $(\Sigma, \sim, \pi, p)$ for a \ggawo\ $\cG$ and a subtree $S$ of $T_{\Sigma}$ define $\Sigma_S$ to be the subgraph of $\Sigma$ induced by $\bigcup_{\llbracket v\rrbracket \in \V S} [v]$. In the case where $S$ has just a single vertex $\llbracket v\rrbracket$ we will  write $\Sigma_{\llbracket v\rrbracket}$ instead of $\Sigma_S$ and, similarly, if $S$ has only two vertices $\llbracket v\rrbracket$ and $\llbracket w\rrbracket$ we will denote the associated subgraph of $\Sigma$ by $\Sigma_{\llbracket v\rrbracket, \llbracket w\rrbracket}$.  

Let $(\Sigma', \sim', \pi', p')$ be a scaffolding for the same \ggawo\ $\cG$ and let $S'$ be a subtree of $\V T_{\Sigma'}$. An isomorphism $f \colon \Sigma_S \to \Sigma'_{S'}$ such that $x\sim y$ if and only if $f(x)\sim'f(y)$, and $\pi(\llbracket x\rrbracket) = \pi' (\llbracket f(x)\rrbracket)$ for every $x, y\in \V\Sigma_S\cup \A\Sigma_S$ is called a \emph{partial scaffolding isomorphism}.  A partial scaffolding isomorphism $\Sigma_S \to \Sigma'_{S'}$ also gives an isomorphism $S\to S'$.   For  $\llbracket v\rrbracket \in \V S$ define the local map $f_{\llbracket v\rrbracket}$ in the same way as in Section~\ref{SAcceptable} and say that $f$ is  an \emph{acceptable partial scaffolding isomorphism} if $f_{\llbracket v\rrbracket} \in G(\pi(\llbracket v\rrbracket))$ for every $\llbracket v\rrbracket \in \V S$. 
%Note that the inverse of an acceptable partial scaffolding isomorphism is also an acceptable partial scaffolding isomorphism.

Our aim is to prove Theorem~\ref{thm:acceptable-isomorphism-ggawo}   that says that any acceptable partial scaffolding isomorphism $\Sigma_s\to \Sigma_{S'}$ can be extended to an acceptable scaffolding isomorphism $\Sigma\to \Sigma'$. The first step is to show that if $f$ is an acceptable partial scaffolding isomorphism defined on $\Sigma_{\bun{v}}$ for some vertex $\bun{v}\in \V T_\Sigma$ and $\bun{w}$ is some vertex in $T_\Sigma$ adjacent to $\bun{v}$ then $f$ can be extended to $\Sigma_{\bun{v}, \bun{w}}$ (see Lemma~\ref{lem:extension}).  This is then used to show that if $S$ is a subtree of $T_\Sigma$ and $\llbracket w\rrbracket$ is a vertex not in $S$ but adjacent to a vertex $\bun{v}$ in $S$ then any acceptable partial isomorphism $f$ defined on $\Sigma_S$ can be extended to $\Sigma_{\hat{S}}$ where $\hat{S}$ is the subtree of $T_{\Sigma}$ we get if the vertex $\llbracket w\rrbracket$ and the edge connecting $\bun{w}$ and $\bun{v}$  is added to the subtree $S$.  This is achieved in Lemma~\ref{lem:bigextension}.  Finally Zorn's lemma is applied to show that any acceptable partial scaffolding isomorphism defined on a subgraph $\Sigma_S$ can be extended to an acceptable scaffolding isomorphism defined on all of $\Sigma$.  

\begin{lemm}
    \label{lem:neighbourhood}
    Let $(\Sigma, \sim, \pi, p)$ and $(\Sigma', \sim', \pi', p')$ be scaffoldings for a \ggawo\ $\cG$ and let $S$ and $S'$ be subtrees of $T_\Sigma$ and $T_{\Sigma'}$, respectively.  Assume $f:\Sigma_S\to \Sigma'_{S'}$ is an 
    acceptable partial scaffolding isomorphism. Suppose $\llbracket v\rrbracket \in \V S$.  
    
    The restriction of $\tilde  f$ to $\Star_{T_S}(\llbracket v\rrbracket)$ is completely determined by $\tilde f(\llbracket v\rrbracket)$ and the local map $f_{\llbracket v\rrbracket}$ of $f$ at $\llbracket v\rrbracket$. In other words, if $g: \Sigma_S\to \Sigma'_{S'}$ is an acceptable partial scaffolding isomorphism such that $\tilde g(\llbracket v\rrbracket)=\tilde f(\llbracket v\rrbracket)$ and $g_{\llbracket v\rrbracket}=f_{\llbracket v\rrbracket}$ then the restrictions of $\tilde f$ and $\tilde g$ to $\Star_{T_S}(\llbracket v\rrbracket)$ are equal.
\end{lemm}

\begin{proof} 
Let $a$ be an arc in $\Sigma_S$ such that $t(\llbracket a\rrbracket)=\llbracket v\rrbracket \in \V T_\Sigma$.
Recall that  $A=t([a]) = \{t(b) \mid b \in [a]\}\subseteq [v]\subseteq \V\Sigma$. 

Then $p(A)=\Psi_{\bun a}p_{\bun a}([a])$ is a $\pi(\llbracket a\rrbracket)$-adhesion set in $X(\pi(\llbracket v\rrbracket))$.  The group $G(\pi(\llbracket v\rrbracket))$ permutes the $\pi(\llbracket a\rrbracket)$-adhesion sets in $X(\pi(\llbracket v\rrbracket))$ and thus $f_{\llbracket v\rrbracket}p(A)$ is also a $\pi(\llbracket a\rrbracket)$-adhesion set in $X(\pi(\llbracket v\rrbracket))$.  By Definition~\ref{DScaffolding} there is a unique arc $\llbracket a'\rrbracket\in T_{\Sigma'}$ such that $t(\llbracket a'\rrbracket)=\llbracket v'\rrbracket$, $\pi(\llbracket a'\rrbracket)=\pi(\llbracket a\rrbracket)$ and the image of $[a']$ under $\Psi_{\bun{a'}}p$ is $f_{\llbracket v\rrbracket}p(A)$.   Since $f$ is a partial scaffolding isomorphism, it follows that $\tilde f(\llbracket a\rrbracket)=\llbracket a'\rrbracket$.  

Since $\tilde g(\llbracket v\rrbracket)=\tilde f(\llbracket v\rrbracket)$ and $g_{\llbracket v\rrbracket}=f_{\llbracket v\rrbracket}$, the same argument shows that $\tilde g(\llbracket a\rrbracket)=\llbracket a'\rrbracket$ and the result follows.
\end{proof}

As an immediate corollary we obtain the following result which is of particular interest when some of the neighbours of $\llbracket v \rrbracket$ are not contained in $S$.

\begin{coro}\label{cor:extension-star}
    Let $(\Sigma, \sim, \pi, p)$ and $(\Sigma', \sim', \pi', p')$ be scaffoldings for a \ggawo\ $\cG$ and let $S$ and $S'$ be subtrees of $T_\Sigma$ and $T_{\Sigma'}$, respectively.  Assume $f:\Sigma_S\to \Sigma'_{S'}$ is an acceptable partial scaffolding isomorphism.  Suppose $\hat{S}$ is a subtree of $T_\Sigma$ containing $S$ and that $F$ and $G$ are acceptable partial scaffolding isomorphisms defined on $\hat{S}$ that extend $f$.  If $\llbracket v\rrbracket$ is a vertex in $S$ then the restrictions of $F$ and $G$ to $\Star_{\hat{S}}(\llbracket v\rrbracket)$ are equal.
\end{coro}

%This is already stated verbatim in the Lemma above.

% \begin{coro}
%     Let $(\Sigma, \sim, \pi, p)$ and $(\Sigma', \sim', \pi', p')$ be scaffoldings for a \ggawo\ $\cG$.   Assume $f:\Sigma\to \Sigma'$ is an acceptable  scaffolding isomorphism.    If $\llbracket v\rrbracket$ is a vertex in $\Sigma$ then the restriction of $f$ to $\Star_{T_\Sigma}(\llbracket v\rrbracket)$ is completely determined by the local map $f_{[v]}$.
% \end{coro}

\begin{lemm}[One-step extension property]
    \label{lem:extension}
        Let $(\Sigma, \sim, \pi, p)$ and $(\Sigma', \sim', \pi', p')$ be scaffoldings for a \ggawo\ $\cG$. Let $\llbracket v\rrbracket \in \V T_{\Sigma}$ and $\llbracket v'\rrbracket \in \V T_{\Sigma'}$ with $\pi(\llbracket v\rrbracket) = \pi'(\llbracket v'\rrbracket)$ and assume that $f: \Sigma_{\llbracket v\rrbracket}\to \Sigma'_{\llbracket v'\rrbracket}$ is an acceptable partial scaffolding ismorphism.  
      Suppose $\llbracket w\rrbracket$ is a neighbour of $\llbracket v\rrbracket$ in $T_\Sigma$.  
      \begin{enumerate}
          \item There is a neighbour $\llbracket w'\rrbracket$ of $\llbracket v'\rrbracket$ so that it is possible to extend $f$ to an acceptable isomorphism $F: \Sigma_{\llbracket v\rrbracket, \llbracket w\rrbracket}
    \to \Sigma'_{\llbracket v'\rrbracket, \llbracket w'\rrbracket}$. 
    \item  The vertex $\bun{w'}$ is unique; if $\llbracket w''\rrbracket\neq \llbracket w'\rrbracket$ is a neighbour of $\llbracket v'\rrbracket$ then it is not possible to extend $f$ to an acceptable partial scaffolding isomorphism $\Sigma_{\llbracket v\rrbracket, \llbracket w\rrbracket}
    \to \Sigma'_{\llbracket v'\rrbracket, \llbracket w''\rrbracket}$. 
      \end{enumerate} 
    \end{lemm}
    
\begin{proof}
Let $f_{\llbracket v \rrbracket}$ denote the local map of $f$ at $\llbracket v \rrbracket$.  Furthermore, let $\llbracket a \rrbracket \in \A T_{\Sigma}$ be such that $t(\llbracket a\rrbracket)=\llbracket v\rrbracket$ and $o(\llbracket a\rrbracket)=\llbracket w\rrbracket$. Define $\hat{S}$ as the subtree of $T_\Sigma$ containing just the vertices $\llbracket v\rrbracket$ and $\llbracket w\rrbracket$ (and the arcs connecting them).   From Lemma~\ref{lem:neighbourhood} and Corollary~\ref{cor:extension-star} it follows that there is a unique vertex $\llbracket w'\rrbracket \in V$ such that if $F$ is an extension of $f$ to $\Sigma_{\bun{v}, \bun{w}}$ then $F(\llbracket w\rrbracket)=\llbracket w'\rrbracket$. This shows the uniqueness of $\bun{w'}$.

Let $\bun{a'} \in \A T_{\Sigma'}$ such that $o(\bun{a'})=\bun{w'}$ and $t(\bun{a'})=\bun{v'}$.   Then $\pi(\bun{v})=\pi(\bun{v'})$, $\pi(\bun{w})=\pi(\bun{w'})$ and $\pi(\bun{a})=\pi(\bun{a'})$.  Therefore $X(\pi(\bun{v}))=X(\pi(\bun{v'}))$, $X(\pi(\bun{w}))=X(\pi(\bun{w'}))$ and $Y(\pi(\bun{a}))=Y(\pi(\bun{a'}))$.

It remains to show that we can find a map $g\colon [w] \to [w']$ that can be used to extend $f$.   Intuitively, this should be clear since all we need is that the adhesion set $t([\ba])$ in $[w]$ is mapped to the adhesion set $t([\bar{a'}])$ in $[w']$ ``in the same way'' as $t([a])$ is mapped to $t([a'])$. This intuition is made rigorous by showing that there exist maps indicated by dashed lines in the diagram in Figure \ref{fig:diagram-extensionlemma} so that the diagram is commutative.
\begin{figure}
    \centering
    \begin{tikzcd}%[column sep=large]
        {[v]}
        \arrow[ddddd,bend right=70,swap,"f"]
        \arrow[d,"p_{\bun v}"]
        &
        {[a]}
        \arrow [d,"p_{\bun a}"]
        \arrow [l,hook',swap,"t"]
        \arrow[r,leftrightarrow,"\bar{\cdot}"]
        &
        {[\ba]}
        \arrow[d,swap,"p_{\bun{\ba}}"]
        \arrow[r,hook,"t"]
        &
        {[w]}
        \arrow[d,swap,"p_{\bun w}"]
        \arrow[ddddd,dashed,bend left=70,"g"]
        \\
        X(\pi(\llbracket v\rrbracket))
        \arrow[ddd,swap,bend right=50,"f_{\llbracket v\rrbracket}"]
        &
        Y(\pi(\llbracket a\rrbracket))
        \arrow [l,swap,hook',"\Psi_{\bun{a}}"]
        \arrow[dl,hook',bend left,swap,"\Phi_{\pi(\llbracket a\rrbracket)}"]
        \arrow[ddd,dashed,"r_{\llbracket v\rrbracket}"]
        \arrow [r,leftrightarrow,"="]
        &
        Y(\pi(\llbracket \ba\rrbracket))
        \arrow [r,hook,"\Psi_{\bun{\ba}}"]
        \arrow[dr,hook,bend right,"\Phi_{\pi(\llbracket \ba\rrbracket})"]
        \arrow[ddd,dashed,swap,"r_{\llbracket w \rrbracket}"]
        &
        X(\pi(\llbracket w\rrbracket))
        \arrow[ddd,dashed,bend left=50,"g_{\llbracket w\rrbracket}"]
        \\
        X(\pi(\llbracket v\rrbracket)) 
        \arrow[u,swap,"h_{\llbracket v\rrbracket}"]
        \arrow[d,dashed,"q_{\llbracket v\rrbracket}"]
        &
        &
        &
        X(\pi(\llbracket w\rrbracket))
        \arrow[u,"h_{\llbracket w\rrbracket}"]
        \arrow[d,dashed,swap,"q_{\llbracket w\rrbracket}"]
        \\
        X(\pi(\llbracket v'\rrbracket)) 
        \arrow[d,"h_{\llbracket v'\rrbracket}"]
        &
        &
        &
        X(\pi(\llbracket w'\rrbracket))
        \arrow[d,swap,"h_{\llbracket w'\rrbracket}"]
        \\
        X(\pi(\llbracket v'\rrbracket)) 
        &
        Y(\pi(\llbracket a\rrbracket))
        \arrow [l,swap,hook',"\Psi_{\bun{a'}}"]
        \arrow[ul,hook',bend right,"\Phi_{\pi(\llbracket a'\rrbracket)}"]
        \arrow [r,leftrightarrow,"="]
        &
        Y(\pi(\llbracket \ba\rrbracket))
        \arrow [r,hook,"\Psi_{\bun{\ba'}}"]
        \arrow[ur,hook,bend left,swap,"\Phi_{\pi(\llbracket \ba'\rrbracket)}"]
        &
        X(\pi(\llbracket w'\rrbracket))
        \\
        {[v']}
        \arrow[u,swap,"p'_{\bun{v'}}"]
        &
        {[a']}
        \arrow [u,swap,"p'_{\bun{a'}}"]
        \arrow [hook',l,swap,"t"]
        \arrow[r,leftrightarrow,"\rev{\,\cdot\,}"]
        &
        {[\ba']}
        \arrow[u,"p'_{\bun{\ba'}}"]
        \arrow[r,hook,"t"]
        &
        {[w']}
        \arrow[u,"p'_{\bun{w'}}"]
    \end{tikzcd}
    \caption{Commutative diagram for the one-step extension property (Lemma~\ref{lem:extension})}
    \label{fig:diagram-extensionlemma}
\end{figure}

We first focus on the maps shown by solid lines in the diagram. 
Let $p_{\bun a}:[a]\to Y(\pi(\llbracket a\rrbracket))$ and adhesion maps $\Psi_{\bun a}: Y(\pi(\llbracket a\rrbracket))\to X(\pi(\llbracket v\rrbracket))$ and $\Psi_{\bun{\ba}}: Y(\pi(\llbracket\ba\rrbracket))\to X(\pi(\llbracket w\rrbracket))$ be as in Definition \ref{DScaffolding}.  Then the top two rows of the diagram in Figure \ref{fig:diagram-extensionlemma} commute; see Figure \ref{fig:diagram-legallabelling}. By the definition of adhesion maps, there are elements $h_{\llbracket v\rrbracket} \in G(\pi(\llbracket v\rrbracket))$ and $h_{\llbracket w\rrbracket} \in G(\pi(\llbracket w\rrbracket))$ such that the top half of the diagram commutes, see  the diagram in Figure \ref{fig:diagram-adhesionmap}. In the same way we see that the bottom half of the diagram commutes.  Note that so far we have not referred to the maps $f$ or $f_{\llbracket v\rrbracket}$ in any way.

The only choice for $q_{\llbracket v \rrbracket}$ that makes the first column of the diagram in Figure \ref{fig:diagram-extensionlemma} (together with the maps $f$ and $f_{\llbracket v\rrbracket}$) commutative is $q_{\llbracket v\rrbracket} = (h_{\llbracket v'\rrbracket})^{-1}f_{\llbracket v\rrbracket}h_{\llbracket v\rrbracket}$.   By assumption $f_{\llbracket v\rrbracket} \in G(\pi(\llbracket v\rrbracket))$ and as $G(\pi(\llbracket v\rrbracket))=G((\pi\llbracket v'\rrbracket))$ we see that $q_{\llbracket v \rrbracket} \in G(\pi(\llbracket v\rrbracket))$. 

Let $B$ denote the image of $\Phi_{\pi(\bun{a})}$.  Using the commutativity of the diagram with solid arrows, we see that 
$$p_{\bun{v}}t([a])=\Psi_{\bun{a}}p_{\bun{a}}([a])=h_{\bun{v}}\Phi_{\bun{a}}p_{\bun{a}}([a])=h_{\bun{v}}(B).$$ 
Again, using the above and the commutativity of the diagram we find that 
$$p_{\bun{v'}}ft([a])=f_{\bun{v}}\Psi_{\bun{a}}p_{\bun{a}}([a])=f_{\bun{v}}h_{\bun{v}}(B).$$
From the choice of $\bun{w'}$ it follows that $f(t[a])=t([a'])$.  Then, using the fact that $\Phi_{\pi(\bun{a'})}=\Phi_{\pi(\bun{a'})}$, it follows that 
$$p_{\bun{v'}}ft([a])=p_{\bun{v'}}t([a'])=h_{\bun{v'}}\Phi_{\pi(\bun{a'}}p'_{\bun{a'}}([a'])=h_{\bun{v'}}(B).$$
Hence $f_{\bun{v}}h_{\bun{v}}(B)=h_{\bun{v'}}(B)$ implying that $q_{\bun{v}}$ stabilizes setwise the set $B$, the image of $\Phi_{\pi(\bun{a})}$.

Since $\Phi_{\pi(\bun{a})}$ induces an embedding of permutation groups there is $r_{\llbracket v \rrbracket}\in G(\pi(\llbracket a\rrbracket))$ such that $q_{\llbracket v\rrbracket} \Phi_{\pi(\llbracket a\rrbracket)} = \Phi_{\pi(\llbracket a\rrbracket)} r_{\llbracket v \rrbracket}$.  As $Y(\pi(\llbracket \ba\rrbracket))=Y(\pi(\llbracket a\rrbracket))$ we may set $r_{\llbracket w\rrbracket}=r_{\llbracket v\rrbracket}$. 

Since the map $\Phi_{\pi(\llbracket \ba\rrbracket)} = \Phi_{\pi(\llbracket \ba'\rrbracket)} $ induces an embedding of permutation groups there is an element $q_{\llbracket w\rrbracket} \in G(\pi(\llbracket w\rrbracket))$ stabilizing the image of this embedding such that $\Phi_{\pi(\llbracket \ba\rrbracket)} r_{\llbracket w \rrbracket} = q_{\llbracket w \rrbracket} \Phi_{\pi(\llbracket \ba\rrbracket)}$. 

We set $g_{\llbracket w \rrbracket}=h_{ \llbracket w'\rrbracket}q_{\llbracket w\rrbracket}(h_{\llbracket w\rrbracket})^{-1}$ and $g = (p_{\bun{w'}})^{-1}g_{\llbracket w \rrbracket}p_{\bun{w}}$.  Finally, define $F\colon \Sigma_{\bun{v}, \bun{w}}\to \Sigma_{\bun{v'}, \bun{w'}}$ so that if $x\in [v]$ then $F(x)=f(x)$ and if $x\in [w]$ then $F(x)=g(x)$.

We still need to check that $F$ is a graph morphism.  Let $b$ be an arc in $[a]$.  Since the diagram commutes we see that if we set $b'=(p'_{\bun{a'}})^{-1}r_{\llbracket v \rrbracket}p_{\bun{a}}(b)$ then $F(t(b))=t(b')$.  On the other hand we see that $\bb'=(p'_{\bun{\ba'}})^{-1}r_{\llbracket w \rrbracket}p_{\bun{a}}(\bb)$.  Thus 
$F(o(b))=F(t(\bb))=g(t(\bb))=t(\bb')=o(b')$.  Since scaffoldings are simple graphs and the only arcs in the subgraph of $\Sigma$ induced by $[v] \cup [w]$ are those in $[a]$ and $[\ba]$, this suffices to conclude that $F$ is a graph morphism.
\end{proof}

\begin{remark}\label{RUniquelyDetermined}
With the notation used in the proof above it is clear that $r_{\llbracket v\rrbracket}$ and thus also $r_{\llbracket w\rrbracket}$ are uniquely determined by $q_{\llbracket v\rrbracket}$, but $q_{\llbracket w\rrbracket}$ is in general not uniquely determined by $r_{\llbracket w\rrbracket}$.  The element $q_{\llbracket w\rrbracket}$ could be replaced by any element in the setwise stabiliser of the image of $\Phi_{\pi(\llbracket \ba\rrbracket)}$ that induces the same permutation as $q_{\llbracket w \rrbracket}$ on this set.  Thus we see that $q_{\llbracket w\rrbracket}$, and thus also $g$, is uniquely determined if and only the subgroup of $G(\pi(\llbracket w\rrbracket))$ fixing pointwise the image of $\Phi_{\pi(\llbracket \ba\rrbracket)}$ is trivial.
\end{remark}   

\begin{lemm}
    \label{lem:bigextension}
     Let $(\Sigma, \sim, \pi, p)$ and $(\Sigma', \sim', \pi', p')$ be scaffoldings for a \ggawo\ $\cG$. 
    Suppose $S$ and $S'$ are subtrees of $T_\Sigma$ and $T_{\Sigma'}$, respectively, and $f\colon \Sigma_{S} \to \Sigma'_{S'}$ is an acceptable partial scaffolding isomorphism. Suppose $\llbracket w\rrbracket$ is a vertex in $T_{\Sigma}$ that is not in $S$ but is adjacent to some vertex $\llbracket v\rrbracket\in S$.  Define $\hat S$ as the subtree of $T_{\Sigma}$ that we get by adding $\llbracket w\rrbracket$ to $S$.  
    Then there is a subtree $\hat{S'}$ of $T_{\Sigma'}$ such that there exists an acceptable partial scaffolding isomorphism $F\colon \Sigma_{\hat S} \to \Sigma'_{\hat S'}$ whose restriction to $\Sigma_{S}$ is $f$.  The subtree $\hat{S'}$ is uniquely determined by the property that it is the image of an extension of $f$ to $\hat S$.
\end{lemm}

\begin{proof}
    Set $\llbracket v'\rrbracket=f(\llbracket v\rrbracket)$.  Let $f_0$ denote the restriction of $f$ to $[v]$.   Then $f_0$ is an acceptable partial scaffolding isomorphism $\Sigma_{\bun v}\to \Sigma'_{\bun{v'}}$.  Lemma \ref{lem:extension} yields an extension $F_0$ of $f_0$ to an acceptable partial scaffolding isomorphism $\Sigma_{\bun{v}, \bun{w}}\to \Sigma_{\bun{v'}, \bun{w'}}$.   By considering the adhesion sets in $[v]$ as in the proof og Lemma~\ref{lem:neighbourhood} we see that $\bun{w'}$ is not in $S'$.  Define $\hat{S'}$ as the subtree of $T_{\Sigma'}$ that we get by adding the vertex $\llbracket w'\rrbracket$ to $S'$.  If $\llbracket w'\rrbracket \notin S'$, then $F_0$ can be used to extend $f$ to a map $F \colon \Sigma_{\hat S}\to \Sigma_{\hat{S'}}$ that is an acceptable partial scaffolding isomorphism.  The statement that $\hat{S'}$ is unique follows from the fact that $\bun{w'}$ is unquely determined. 
\end{proof}

The final step in showing that any acceptable partial scaffolding isomorphism can be extended to an acceptable scaffolding isomorphism is a simple application of Zorn's lemma.

\begin{theo}[Full extension property for \ggawo]
    \label{thm:acceptable-isomorphism-ggawo}
    Let $(\Sigma, \sim, \pi, p)$ and $(\Sigma', \sim', \pi', p')$ be scaffoldings for a \ggawo\ $\cG$.  Suppose $S_0$ and $S_0'$ are subtrees of $T_\Sigma$ and $T_{\Sigma'}$, respectively, and that $f \colon \Sigma_{S_0} \to \Sigma'_{S_0'}$ is an acceptable partial scaffolding isomorphism.
Then $f$ can be extended to an acceptable scaffolding isomorphism $F: \Sigma\to \Sigma'$.
\end{theo}

\begin{proof}
 Define $X$ as the set of all pairs $(S, f_S)$ where $S$ is a subtree of $T_\Sigma$ containing $S_0$ and $f_S:  \Sigma_S\to \Sigma'_{S'}$ is an acceptable partial scaffolding isomorphism extending $f$.  For two such pairs $(S_1, f_{S_1})$ and $(S_2, f_{S_2})$ write  $(S_1, f_{S_1})\leq(S_2, f_{S_2})$ if $S_1\subseteq S_2$ and $f_{S_2}$ is an extension of $f_{S_1}$.  

 Suppose $(S_i, f_{S_i})_{i\in I}$ is a chain in $X$.  Set $S=\bigcup S_i$ and $S'=\bigcup S'_i$ where $S_i'=f_i(S_i)$.  Then clearly $S$ is a subtree of $T_\Sigma$.  Define $F: \Sigma_S\to \Sigma_{S'}$ so that if a vertex $v$ is contained in $\Sigma_{S_i}$ for some $i$ then $F(v)=f_i(v)$.  We show that the value of $F(v)$ is independent of the choice of $i$.  If $v\in \Sigma_{S_j}$ for some $j$ then $(S_j, f_{S_j})\leq (S_i, f_{S_i})$ or $(S_i, f_{S_i})\leq (S_j, f_{S_j})$.   In the first case $f_{S_i}$ is an extension of $f_{S_j}$ and in the second case $f_{S_j}$ is an extension of $f_{S_i}$.  Thus $f_i(v)=f_j(v)$.   We also see that if $\llbracket v\rrbracket$ is a vertex in $S$ then there is an $i$ such that $\llbracket v\rrbracket$ is in $S_i$ and the local action of $F$ at $\llbracket v\rrbracket$ is equal to the local action of $f_{S_i}$ at $\llbracket v\rrbracket$ and hence $F$ is an acceptable partial scaffolding isomorphism defined on $\Sigma_S$.  Clearly $F$ extends $f$.   Thus $(S, F)\in X$ is an upper bound for this chain.

 Hence every chain in the partially ordered set $X$ has an upper bound.   By Zorn's lemma the set $X$ then has a maximal element which we denote by $(S, F)$.  If $S$ is not equal to $T_\Sigma$ then we could find a vertex $\llbracket w\rrbracket$ in $T_\Sigma$ that is not in $S$ but is adjacent to a vertex in $S$.  Then Lemma~\ref{lem:bigextension} says that $F$ could then be extended to the subtree we get by adding $\llbracket w\rrbracket$ to $S$.  This contradicts the maximality of $(S, F)$.  Set $S'=F(S)$.  Then $S'=T_{\Sigma'}$.  This follows from the fact that if $[v]$ is a vertex bundle in $\Sigma$ then $F$ maps the adhesion sets in $[v]$ bijectively to the adhesion sets in $[v']=F([v])$ and if $\llbracket w'\rrbracket$ is adjacent to $\llbracket v'\rrbracket$ in $T_{\Sigma'}$ then $[w']$ is in the image of $F$ and thus $\llbracket w'\rrbracket$ is in the image of the map $T_\Sigma\to T_{\Sigma'}$ induced by $F$.  Since $T_{\Sigma'}$ is connected we see that the image of the map $T_\Sigma\to T_{\Sigma'}$ is the whole of $T_{\Sigma'}$ and thus the image of $F$ is $\Sigma'$.   
\end{proof}

\begin{coro}
    \label{cor:acceptable-isomorphism-ggawo}
      Let $(\Sigma, \sim, \pi, p)$ and $(\Sigma', \sim', \pi', p')$ be scaffoldings for a \ggawo\ $\cG$.  Let $\llbracket v\rrbracket \in \V T_{\Sigma}$ and $\llbracket v'\rrbracket\in \V T_{\Sigma'}$ with $\pi(\llbracket v\rrbracket) = \pi(\llbracket v'\rrbracket)$, and let $f_0 \in G(\pi(\llbracket v\rrbracket))$.
          There is an acceptable scaffolding isomorphism $f \colon \Sigma \to \Sigma'$ such that $\tf (\llbracket v\rrbracket) = \llbracket v'\rrbracket$ and $f_{\llbracket v \rrbracket} = f_0$.
          In particular, there exists an acceptable scaffolding isomorphism $\Sigma\to \Sigma'$.  
\end{coro}

\begin{proof}
    Apply Theorem \ref{thm:acceptable-isomorphism} to the case where $S$ and $S'$ each consists of single vertex $\llbracket v\rrbracket$ and $\llbracket v'\rrbracket$, respectively, and the acceptable partial scaffolding isomorphism $\Sigma_{\llbracket v\rrbracket}\to \Sigma_{\llbracket v\rrbracket}$ has local map $f_0$ at $\llbracket v \rrbracket$.
\end{proof}

\section{Inversions}\label{SInversions}

Above  we have followed the precedent in Bass--Serre theory and not allowed self-reverse arcs in the base graph of a \ggawo. 
In this section, our theory is extended to include the possibility of self-reverse arcs in the base graph, thus allowing the possibility of edge inversions in the resulting group action on a tree. 

\subsection{Graphs of group actions (revisited)}
The first step is to modify the definition of a \ggawo, Definition~\ref{DGraph-of-group-action}, by including special conditions for self-reverse arcs in the base graph.   These conditions are in part (d).

\begin{defi}\label{DEGraph-of-group-action}
A \emph{graph of group actions} (or a \emph{\gga}) consists of

\begin{enumerate}[label=(\alph*)]
    \item a connected graph $\Delta$ called the \emph{base graph},
    \item for every vertex $v \in \V\Delta$ a permutation group action $(G(v),X(v))$, called \emph{the action at the vertex} $v$, and
    \item for every arc $a \in \A \Delta$ 
    \begin{itemize}
    \item a permutation group action $(H(a),Y(a))$, called  \emph{the action at the arc} $a$, such that $(H(\ba),Y(\ba)) = (H(a),Y(a))$, and 
    \item an embedding  $(\phi_a, \Phi_a)\colon (H(a),Y(a)) \hookrightarrow (G(t(a)),X(t(a)))$ of permutation actions.
    \end{itemize}
    \item for every arc $a \in \A \Delta$ that is self-reverse there is
 a permutation $h_a$ (called an \emph{inversion agent}) of $Y(a)$ such that $h_aH(a)h_a^{-1}=H(a)$ and $h_a^2$ is contained in $H(a)$. 
\end{enumerate}
\end{defi}

\subsection{Scaffoldings (revisited)}

The nest step is to modify the definition of a scaffolding to take into account self-reverse arcs in the base graph, $\Delta$, of our \gga\ $\cG$. If $a$ is an arc in $\Delta$ then, as before, an $a$-adhesion map is a map $\Psi: Y(a)\to X(t(a))$ such that there exists $g\in G(t(a))$ such that $\Psi=g\Phi_a$.  If the arc $a$ is self-reverse we must also consider \emph{twisted adhesion map}, that is, maps $\Psi: Y(a)\to X(t(a))$ such that there exists $g\in G(t(a))$ such that $\Psi=g\Phi_ah_a$.  

%addition
%Perhaps the proof can be left out?
\begin{remark}
    The condition that $h_aH(a)h_a^{-1}=H(a)$ ensures that $\Phi_ah_a: Y(a)\to X(t(a))$ is an embedding of permutation actions.  
    
    This can be seen from the following argument:  Let $(H, Y)$ and $(G,X)$ be permutation actions.  Suppose that $\phi: H\to G$ is a group isomorphism and $\Phi:Y\to X$ is a bijection so that the two form an isomorphism of group actions, i.e.\ $\phi(h)\Phi(y)=\Phi(\phi(y))$  for every element $h\in H$ and every element $y\in Y$.  Suppose that $h_a$ is a permutation of $Y$ such that $h_aHh_a^{-1}=H(a)$.  Then $h\mapsto h_ahh_a^{-1}$ is an isomorphism of $H$ and $\psi: H\to G$ defined by $\psi(h)=\phi(h_ahh_a^{-1})$ is a group isomorphism.  The map $\Psi=\Phi h_a: Y\to X$ is a bijection.  If $h\in H$ and $y\in Y$ we get
    $$\Psi(h(y))=\Phi(h_ahh_a^{-1}h_a(y))=\phi(h_ahh_a^{-1})\Phi(h_a(y))
    =\psi(h)\Psi(y).$$
    Thus the bijection $\Psi=\Phi h_a: Y\to X$ gives an isomorphism of permutation actions.
\end{remark}

%change-simplification
\begin{defi}\label{DScaffolding-gga}
        A \emph{scaffolding} for a \gga\  ${\cal G}$ is a graph  $\Sigma$ together with a graph equivalence relation $\sim$ on $\V\Sigma$ such that all the conditions in Definition~\ref{DScaffolding} are satisfied except that the fourth condition in part (iv) is replaced by the following condition
\begin{itemize}
    \item  Define $\Psi_{\llbracket a\rrbracket}:Y(\pi(\llbracket a\rrbracket))\to X(\pi(t(\llbracket a\rrbracket)))$ as the unique map satisfying $\Psi_{\llbracket a\rrbracket}p(b)=pt(b)$ for every $b\in [a]$.  
    
    If the arc $\pi(\llbracket a\rrbracket)$ is not self-reverse then $\Psi_{\llbracket a\rrbracket}$ is a $\pi(\llbracket a\rrbracket)$-adhesion map.   
    
    If $\pi(\llbracket a\rrbracket)$ is self-reverse then $\Psi_{\llbracket a\rrbracket}$ is a $\pi(\llbracket a\rrbracket)$-adhesion map and $\Psi_{\llbracket \ba\rrbracket}$ is a twisted $\pi(\llbracket a\rrbracket)$-adhesion map or vice versa.
\end{itemize}
\end{defi}

%addition

Very little needs to be changed in the construction of a canonical scaffolding in Section~\ref{sec:canonical}.  For a self reverse arc $a$ in the base graph $\Delta$  we define the index $i(a)$ in the same way as before as the number of distinct $G(t(a))$-translates of the set $\Phi_a(Y(a))$.  When it comes to constructing the augmented base digraph $\Delta_+$ we add $i(a)-1$ new arcs that are all self-reverse and have the same origin as $a$.  Associate each of these new arcs and $a$ with distinct translates of $\Phi_a(Y(a))$ so that $a$ is associated with $\Phi_a(Y(a))$.  Then if $b$ is one of these arcs in $\Delta_+$ fix an $a$-adhesion map $\Phi_b$ such that the image of $\Phi_b$ is the translate of $\Phi_a(Y(a))$ associated with $b$.  The construction of a $\Star$-covering tree works equally well in the case when the graph in question has self-reverse arcs, see Remark~\ref{RCompatible}.  Let $T$ be a $\Star$-covering tree for $\Delta_+$.  The vertex set and the arc set of a canonical scaffolding $\Sigma$ are defined as before and the origin and terminus maps for arcs $(c,y)$ in $\Sigma$ where $\pi(c)$ is not self-reverse are defined in the same way as before.  For each pair $c, \bc$ of arcs in $T$ that $\pi$ maps to a self-reverse arc $a$ in $\Delta$ we set $b=\pi_+(c)=\pi_+(\bc)$ and for $y\in Y(a)$ set $o(c,y)=(o(c), \Phi_b(h_a(y)))$ and $t(c,y)=(t(c), \Phi_b(y))$.   Then set $o(\bc,y)=(o(\bc), \Phi_b(y))$ and $t(\bc,y)=(t(\bc), \Phi_b(h_a(y)))$.  The arcs $(c,y)$ and $(\bc,y)$ in $\Sigma$ are reverse of each other.  It is an easy exercise to show that the graph $\Sigma$ thus constructed is a scaffolding for the \gga\ $\cG$.

\subsection{Universal group (revisited)}

Using the updated notions of \gga\ and scaffolding, we can define the universal group in the same way as in Section \ref{SAcceptable}. More precisely, a \emph{scaffolding isomorphism} is a graph isomorphism which preserves the equivalence relation $\sim$ and leaves $\pi$ invariant. Since $\sim$ and $\pi$ are invariant under any scaffolding isomorphism, we can for a vertex bundle $\bun{v}$ define the local map $g_{\llbracket v\rrbracket}$ as the unique permutation of $X(\pi(\llbracket v\rrbracket))$ such that $p g = g_{\llbracket v\rrbracket} p$, and the universal group as the group of acceptable scaffolding automorphisms, that is, scaffolding automorphisms whose local map is in $G(\pi(\llbracket v\rrbracket))$ for every vertex bundle $\llbracket v\rrbracket$ in the scaffolding.

%addition
\begin{remark}
    It is shown in Lemma~\ref{LArcBundle} the given an arc bundle $\bun{a}$ in a scaffolding graph for a $\ggawo$ then the universal group induces on this arc bundle an action isomorphic to the action of $H(\pi(\bun{a}))$ on $Y(\pi(\bun{a}))$.  When the arc $\pi(\bun{a})$ is self-reverse and $h_{\pi(\bun{a})}$ is the inversion agent, then the condition that $h_{\pi(\bun{a})}^2\in H(\pi(\bun{a}))$ ensures that his remains true when we have a general \gga.
\end{remark}

\subsection{Treating inversions via subdivision}
\label{sec:inversion-subdivision}

It is possible to \lq\lq remove\rq\rq\ the self-reverse arcs from a \gga\ $\cG$ by a procedure reminiscent of a \lq\lq barycentric\rq\rq\ subdivision. This allows us to seamlessly transfer results about \ggawo\ to the more general setting of \gga.

Let $\cG$ be a \gga\ and suppose that $a$ is a self-reverse arc in $\Delta$.  We replace the arc $a$ with a new vertex $v_a$ and arcs $b_a, \bb_a$, which are reverses of each other, such that the origin of $b_a$ is $v_a$ and the terminus is $t(a)$.  Set $X(v_a)=\{-1,1\}\times Y(a)$.  The group $H(a)$ acts naturally on $X(v_a)$ so that $h(\epsilon, y)= (\epsilon, h(y))$ for $h\in H(a)$.  Let $h_a$ be the inversion agent, and let $h_{a+}$ be the permutation of $X(v_a)$ given by $h_{a+}(\epsilon, y)= (-\epsilon, h_a(y))$.  
The action at $v_a$ will be the action of the permutation group $G(v_a)=\langle H(a), h_{a+}\rangle$ on $X(v_a)$.   The action at the arc $b_a$ is $(H(a), Y(a))$ and the embeddings of group actions are given by $\Phi_{b_a}: Y(a)\to X(t(a))$ equal to $\Phi_a$ and $\Phi_{\bb_a} : Y(a)\to X(v_a) ; y\mapsto (1, y)$.  Thus $A=\{(1, y)\mid y\in Y(a)\}$ is an adhesion set in $X(v_a)$ and the only other adhesion set in $X(v_a)$ is the set $h_{a+}(A)= \{(-1, y)\mid y\in Y(a)\}$. An adhesion map onto $h_{a+}(A)$ is of the form $\Psi(y)=hh_{a+}\Phi_{\bb_a}$ with $h\in H(a)$.   Then $\Psi(y)=hh_{a+}\Phi_{\bb_a}(y)=gh_{a+}(1,y)=(-1,hh_a(y))$. 

Applying this construction to every self-reverse arc in $\Delta$, we obtain a \ggawo\ $\cG_0$.  Take a scaffolding $(\Sigma_0,\sim_0,\pi_0,p_0)$ for $\cG_0$,  and let $G_0$ be the associated universal group.   %Furthermore, let $\sim_0$ denote the equivalence relation on $\V\Sigma_0$ determining the vertex bundles.

Let $[v]$ be a vertex bundle in $\Sigma_0$ that the legal colouring $p_0$ maps to $X(v_a)$ for some vertex $v_a$ in $\cG_0$ obtained from a self-reverse arc in $\Delta$.  Each vertex in $[v]$ is the end vertex of just one edge in $\Sigma_0$.  For each $y\in Y(a)$  replace the two edges in $\Sigma_0$ that have the vertices with images $(1, y)$ and $(-1,y)$ as end vertices, respectively, with a single edge connecting the other two end vertices and remove all the vertices in $[v]$. Note that the new edges we get when $[v]$ is removed from $\Sigma_0$ are in a one-to-one correspondence with the elements in $Y(a)$.  

Repeating this procedure at every such vertex bundle, we obtain a new \lq\lq scaffolding like\rq\rq\ graph $\Sigma$ on which the group $G_0$ acts by automorphisms.
The vertex set of $\Sigma$ is a subset of the vertex set of $\Sigma_0$ and thus $\sim_0$ restricted to $\V\Sigma$ gives an equivalence relation. 
Let $T_0$ be the tree associated to the scaffolding $\Sigma_0$. The vertices in $T_0$ whose image under $\pi$ is one of the subdivision vertices $v_a$ all have degree 2.  Replace each of these vertices and its adjacent edges with an edge connecting its two neighbours to get a new tree $T$; clearly $\Sigma/\!\sim$ is isomorphic to the tree $T$.  The graph morphism $\Sigma_0/\!\sim_0\to \Delta_0$ induces a graph morphism $\pi: \Sigma/\!\sim\to \Delta$.  The legal colouring $p_0$ of $\Sigma_0$ can be used to define a map 
$$p\colon \V\Sigma\cup\A\Sigma \to \left(\bigcup _{v \in \V\Delta} X(v)\right)\cup \left(\bigcup _{a \in \A\Delta} Y(a)\right)$$
such that if $x$ is a vertex or an arc in $\Sigma$ that is also on $\Sigma_0$ then $p(x)=p_0(x)$.   But, if $b$ is an arc in $\Sigma$ that is not in $\Sigma_0$ then $b$ originates in some pair of vertices $(1,y), (-1,y)$ for $y\in Y(a)$ for some self-reverse arc $a$ in $\Delta$ and we define $p(b)=y$.  

It is not difficult to check that $(\Sigma,\sim,\pi,p)$ is a scaffolding for the \gga\  $\cG$. Moreover, every $g \in G$ induces a scaffolding automorphism whose local action at every vertex bundle $\llbracket v \rrbracket$ is contained in $G(\llbracket v \rrbracket)$, and conversely, any acceptable scaffolding automorphism of $\Sigma$ corresponds to some element of $G_0$. Therefore we conclude the following.

\begin{theo}
    With the above definitions, $\Sigma$ is a scaffolding for $\cG$. The universal group $G$ of $\cG$ in its action on $\V\Sigma$ is (permutationally) isomorphic to the action of $G_0$ on $\V\Sigma$ regarded as a subset of $\V\Sigma_0$. 
\end{theo}

As mentioned above, this subdivision approach can be used to quickly transfer results about \ggawo\ to general \gga. For instance, the extension property Theorem~\ref{thm:acceptable-isomorphism-ggawo} for \ggawo's extends easily to \gga.

\begin{theo}    \label{thm:acceptable-isomorphism}
\emph{(Full extension property.)}
    Let $(\Sigma, \sim, \pi, p)$ and $(\Sigma', \sim', \pi', p')$ be scaffoldings for a \gga\ $\cG$.  Suppose $S_0$ and $S_0'$ are subtrees of $T_\Sigma$ and $T_{\Sigma'}$, respectively, and that $f \colon \Sigma_{S_0} \to \Sigma'_{S_0'}$ is an acceptable partial scaffolding isomorphism.
Then $f$ can be extended to an acceptable scaffolding isomorphism $F: \Sigma\to \Sigma'$.
\end{theo}

\begin{proof} (Sketch.)
    Construct \ggawo's $\Sigma_0$ and $\Sigma'_0$ from $\Sigma$ and $\Sigma'$, respectively, as described above.  Then there are subtrees $R_0$ and $R_0'$ in $T_0$ corresponding to $S_0$ and $S'_0$, respectively, and $f$ gives a partial isomorphism $f_0 \colon (\Sigma_0)_{S_0} \to (\Sigma'_0)_{S_0'}$.   By Theorem~\ref{thm:acceptable-isomorphism-ggawo} there exists an extension of $f_0$ to an acceptable scaffolding isomorphism $F_0: \Sigma_0\to \Sigma'_0$ that in turn gives an extension of $f$ to an acceptable scaffolding isomorphism $F : \Sigma\to \Sigma'$.  
\end{proof}

This immediately implies an extension of Corollary~\ref{cor:acceptable-isomorphism-ggawo}.

\begin{coro}
    \label{cor:acceptable-isomorphism}
      Let $(\Sigma, \sim, \pi, p)$ and $(\Sigma', \sim', \pi', p')$ be scaffoldings a \gga\ $\cG$.  Let $\llbracket v\rrbracket \in \V T_{\Sigma}$ and $\llbracket v'\rrbracket\in \V T_{\Sigma'}$ with $\pi(\llbracket v\rrbracket) = \pi(\llbracket v'\rrbracket)$, and let $f_0 \in G(\pi(\llbracket v\rrbracket))$.
          There is an acceptable isomorphism $f \colon \Sigma \to \Sigma'$ such that $\tf (\llbracket v\rrbracket) = \llbracket v'\rrbracket$ and $f_{[v]} = f_0$.
          In particular, there exists an acceptable scaffolding isomorphism $\Sigma\to \Sigma'$.  
\end{coro}

\begin{example}
\begin{enumerate}    
\item  Suppose $\cG$ is a \gga\ and $a$ is a self-reverse arc such that $|Y(a)|=1$.  Then $H(a)=\{1\}$ and the only possible choice for an inversion agent is $h_a=1$.  
\item    Consider a \gga\ $\cG$ that has as a base graph a graph with just one vertex $v$ and one self-reverse arc $a$.  We set $X(v)=Y(a)=\{1, 2, 3\}$ and $G(v)=H(a)=C_3$ acting in the natural way.   The canonical scaffolding graph for $\cG$ has six vertices in two vertex bundles and three (undirected) edges giving a perfect matching between the two vertex bundles. If $h_a=1$ then the universal group of this \gga\ will act on the set of edges like a $C_3$, but if $h_a$ is some permutation of order 2 of the set $\{1, 2, 3\}$ then the universal group would act on the set of three edges as the full symmetric group $S_3$.   In the first case the universal group would be isomorphic to $C_3\times C_2$ and in the second case it would be isomorphic to $S_3$.  
Different choices $h_a, h'_a$ of inversion agents can also produce the same universal group.  This happens precisely when $h_a {h'_a}^{-1}\in H(a)$ because then the process of constructing an ordinary \ggawo\ $\cG_0$ would give the same group actions at $v_a$.  
\item  As in the previous example, our base graph has just one vertex $v$ and one self-reverse arc $a$.  Let $(G(v), X(v))$ be some permutation group action and choose a subset $Y(a)$ of $X(v)$.  Let $H(a)$ be the permutation group induced by $G(v)$ on $Y(a)$ and define $\Phi_a: Y(a)\to X(v)$ as the inclusion map.  Furthermore, set $h_a=1$.  Denote by $\Sigma$ a canonical scaffolding for $\Sigma$ and let $T$ denote the associated tree.    If $u$ and $w$ are adjacent vertices in $T$ and $\{(u, x)\mid x\in X(v)\}$ and $\{(w, x)\mid x\in X(v)\}$ are their vertex bundles in $\Sigma$ then the map $f$ defined on the subgraph of $\Sigma$ spanned by their union such that $f(u,x)=(w,x)$ and $f(w,x)=(u,x)$ for all $x\in X(v)$ is an acceptable partial isomorphism and can thus be extended to an acceptable isomorphism $F$ of $\Sigma$.  It is clearly possible to extend $f$ in such a way that $F(t,x)=(\tilde{F}(t), x)$ for all vertices $(t,x)$ in $\Sigma$, i.e. the local map at every vertex bundle is trivial.     In this case the tree $T$ is a regular tree with degree equal to the number of distinct $G(v)$ translates of the set $Y(a)$.  The universal group $\cU(\cG)$ acts on $T$ vertex- and arc-transitively. In the example above we got the 1-regular tree.
\end{enumerate}
\end{example}

\section{Uniqueness and universality}\label{SUniqueness}

Recall that the universal group of a \gga \ $\mathcal G$ is defined as the group of all acceptable scaffolding automorphisms $\Sigma \to \Sigma$ for a specific scaffolding $\Sigma$ of $\cG$.  With the aid of Theorem~\ref{cor:acceptable-isomorphism} we show that the universal group does not depend on the choice of a scaffolding, and we are thus justified in speaking about \underline{the} universal group of a \gga\ $\cG$.

\begin{theo}\label{Tuniversal-group-unique}
Suppose $\Sigma$ and $\Sigma'$  are scaffoldings for a \gga\  $\cG$.  If $G$ and $G'$ are the universal groups of $\cG$ defined with respect to $\Sigma$ and $\Sigma'$, respectively, then $G$ and $G'$ are isomorphic and their actions on $\Sigma$ and $\Sigma'$, respectively,  are isomorphic and the isomorphism is given by an acceptable scaffolding isomorphism $\Sigma\to \Sigma'$.
\end{theo}

\begin{proof}
    By Corollary \ref{cor:acceptable-isomorphism} there is an acceptable scaffolding isomorphism $\Phi\colon \Sigma \to \Sigma'$.
    Let $\phi \colon G \to G'$ be the map defined by $\phi(g) = \Phi g\Phi^{-1}$. Note that the images of elements of $G$ are indeed in $G'$ since composition of acceptable scaffolding isomorphisms gives an acceptable scaffolding isomorphism. Moreover, $\phi$ is clearly a group homomorphism and has an inverse given by $\phi^{-1}(g') = \Phi^{-1}g'\Phi$. Therefore, $\phi$ is a group isomorphism and $(\phi,\Phi)\colon (G,\Sigma) \to (G',\Sigma')$ is an isomorphism of group actions.
\end{proof}

Suppose $\Sigma$ is a scaffolding for a \gga\ $\cG$.  
Let us call a group  $H \leq \autsc(\Sigma)$ \emph{locally-$\mathcal G$} if there is a legal colouring  $p'$ of $\Sigma$ such that for every $h \in H$ and every $\llbracket v\rrbracket \in \V T_\Sigma$, the local action of $h$ at $\llbracket v\rrbracket$ with respect to $p'$ is contained in $G(\pi(\llbracket v\rrbracket))$. Now the name ``universal group'' can be justified by showing that $\cU(\cG)$ is universal in the set of locally-$\cG$ groups, that is, it is (by definition) locally-$\cG$ and by the following theorem it contains (up to conjugation) every locally-$\cG$ group as a subgroup.

\begin{theo}\label{CUniversal}
Let $\mathcal G$ be a \gga\ and let $\Sigma$ be a scaffolding for $\mathcal G$.  Suppose the universal group of $\cG$ is defined with respect to $\Sigma$ and some legal colouring $p$.    Every locally-$\mathcal G$ subgroup of $\autsc (\Sigma)$ is conjugate (in $\autsc (\Sigma)$) to some subgroup of $\cU(\cG)$.
\end{theo}

\begin{proof}
Let $p'$ be a legal colouring of $\Sigma$ witnessing that $H \leq \autsc (\Sigma)$ is locally-$\mathcal G$. By Corollary \ref{cor:acceptable-isomorphism} there is an acceptable scaffolding isomorphism $\Phi: (\Sigma, \sim, \pi, p) \to (\Sigma, \sim, \pi, p')$.   Note that $\Phi \in \autsc (\Sigma)$.   Let $G'$ denote the universal group of $\cG$ thought of as the group of acceptable scaffolding automorphisms of $\Sigma$ with respect to the legal colouring $p'$.  Then both $G$ and $G'$ are supgroups of $\autsc (\Sigma)$ and $\Phi G'\Phi^{-1}=G$.   Because $H\leq G'$ we see that $H$ is conjugate to a subgroup of $G$.
\end{proof}

\begin{defi}\label{DIsomorphism}
Let $\cG$ and $\cG'$ be \gga's.  An isomorphism $(\theta, \Theta): \cG\to \cG'$ consists of a graph isomorphism $\theta:\Delta\to \Delta'$ and a bijection 
$$\Theta: \left(\bigcup _{v \in \V\Delta} X(v)\right)\cup \left(\bigcup _{a \in \A\Delta} Y(a)\right)\to \left(\bigcup _{v \in \V\Delta'} X'(v)\right)\cup \left(\bigcup _{a \in \A\Delta'} Y'(a)\right),$$
such that:
\begin{enumerate}[label=(\roman*)]
    \item  for each vertex $v\in \V\Delta$ the restriction $\Theta_v$ of $\Theta$ to $X(v)$ induces an isomorphism of group actions $(G(v), X(v))\to (G(\theta(v)), X(\theta(v)))$.  
    \item for each arc $a\in \A\Delta$ the restriction $\Theta_a$ of $\Theta$ to $Y(a)$ induces an isomorphism of group actions $(H(a), Y(a))\to (H(\theta(a)), Y(\theta(a)))$, and
    \item for each $a\in \A\Delta$ there is a $\theta(a)$-adhesion map $\Psi_a:Y'(\theta(a))\to X'(\theta(t(a)))$  such that $\Theta_{t(a)}\Phi_a=\Psi_a\Theta_a$.   
\end{enumerate}
\end{defi}

\begin{lemm}\label{LIso}
Let $(\theta, \Theta)$ be an isomorphism of \gga's $\cG$ and $\cG'$ as above.     
\begin{enumerate}[label=(\roman*)]
\item If $a$ is an arc in $\Delta$ and $\Psi$ is an $a$-adhesion map then there exists an $\theta(a)$-adhesion map $\Psi'$ such that $\Theta_{t(a)}\Psi=\Psi'\Theta_a$.    
\item Let $a$ be an arc in $\Delta$.  Then $\Theta_{t(a)}$ maps each $a$-adhesion set to a $\theta(a)$-adhesion set in $X'(\theta(t(a)))$.
\item There is a digraph morphism $\theta_+: \Delta_+\to \Delta'_+$ such that the restriction of $\theta_+$ to $\V\Delta_+=\V\Delta$ is equal to the restriction of $\theta$ to $\V\Delta$. 
\end{enumerate}
\end{lemm}
\begin{proof}
(i)     If $\Psi$ is an $a$-adhesion map then there is an element $g$ in $G(t(a))$ such that $\Psi=g\Phi_a$.  Since $\Theta_{t(a)}$ induces an isomorphism of group actions there is $g'\in G'(\theta(t(a)))$ such that $\Theta_{t(a)}g=g'\Theta_{\theta(t(a))}$.   Thus 
$\Theta_{t(a)}\Psi=\Theta_{t(a)}g\Phi_a=g'\Theta_{t(a)}\Phi_a=g'\Psi_a\Theta_a$. Because $\Psi_a$ is a $\theta(a)$-adhesion map we conclude that $\Psi'=g'\Psi_a$ is also a $\theta(a)$-adhesion map. 

(ii) If $A$ is an $a$-adhesion set in $X(t(a))$ then there is an element $g$ in $G(t(a))$ such that $A=g\Phi_a(Y(a))$.  As above we let $g'\in G'(\theta(t(a)))$ be such that $\Theta_{t(a)}g=g'\Theta_{\theta(t(a))}$.   Then 
$$\Theta_{t(a)}(A)=\Theta_{t(a)}g\Phi_a(Y(a))=g'\Theta_{t(a)}\Phi_a(Y(a))=g'\Psi_a\Theta_a(Y(a))=g'\Psi_a(Y'({\theta(a)}))$$ 
and as $\Psi_a$ is a $\theta(a)$-adhesion map we see that $\Theta_{t(a)}(A)$ is a $\theta(a)$-adhesion set in $X'(\theta(t(a))$.  

(iii)  Each arc $b$ in $\Delta_+$ can be associated with a $\rho(b)$-adhesion set in $X(t(b))$ and similarly for the arcs in $\Delta'_+$.  From (i) it follows that $\Theta_{t(b)}$ maps the $\rho(b)$-adhesion sets in $X(t(b))$ bijectively to the $\theta(\rho(b))$-adhesions sets in $X'(\theta(b)$.  Thus we get a bijection from the set $\rho^{-1}(\rho(b))$ of arcs in $\Delta_+$ to the set $\rho'^{-1}(\theta(\rho(b)))$ of arcs in $\Delta'_+$.  This then allows us to define a digraph morphism $\theta_+$.
\end{proof}

\begin{theo}
\label{thm:isomorphic}
Suppose that $\cG$ and $\cG'$ are \gga's and $(\theta, \Theta):\cG\to \cG'$ is an isomorphism.  If $(\Sigma, \sim, \pi, p)$ is a scaffolding for $\cG$ then $(\Sigma, \sim, \theta\pi, \Theta p)$ is a scaffolding for $\cG'$ and the universal group of $\cG$ acting on $\Sigma$ is equal to the universal group $G'$ of $\cG'$ acting on $\Sigma$. 
\end{theo}

\begin{proof}
Define $p'=\Theta p$ and $\pi'=\theta\pi$. Then $(\Sigma, \sim, \pi', p')$ is a scaffolding for $\cG'$.  The only thing that needs attention is the fourth item in part (iv) of the definition of a scaffolding.  We need to show that there is a $\pi(\bun{a})$-adhesion map $\Psi'_{\bun{a}}$  such that $\Psi'_{\bun{a}}p'(b)=p't(b)$ for all $b\in [a]$.  This follows from part (i) in Lemma~\ref{LIso}.  

Suppose $g\in G$ and $\bun{v}$ is a vertex bundle in $\Sigma$.  Then $g_{\bun{v}}=pg(p_{\bun{v}})^{-1}\in G(\pi({\bun{v}}))$.  But if we calculate the local action of $g$ at $\bun{v}$ with respect to the legal colouring $p'$ we get the local action at $\bun{v}$ as $p'g(p'_{\bun{v}})^{-1}=\Theta pg(p_{\bun{v}})^{-1}\Theta^{-1}=\Theta g_{\bun{v}} \Theta^{-1}$.   Because $g_{\bun{v}}\in G(\pi({\bun{v}})$ and $\Theta$ induces an isomorphism of the group we see that the local action of $g$ at $\bun{v}$ with respect to the legal colouring $p'$ is contained in $G'(\pi'(\bun{v})$.  Hence $G\leq G'$.  In the same way we see that $G'\leq G$ and thus $G=G'$.
\end{proof}

\begin{remark}
    The converse of Theorem \ref{thm:isomorphic} is not true. In other words, there are non-isomorphic \gga\ with isomorphic universal groups. 
    
    For a simple example, consider a \gga\ on two vertices $u$ and $v$ where $G(u) = G(v) = S_2 \times S_2$ acting in the obvious way on 4 points, and let the only adhesion set be the set of these 4 points. The scaffolding consists of 8 vertices with a perfect matching between them, and the universal group is $S_2 \times S_2$ acting on this matching.

    Another \gga\ with the same scaffolding and universal group is given by again taking two vertices $u$ and $v$ and letting $G(u) = S_2 \times S_2$ in its action on 4 points. This time, however, we let $G(v) = S_2$ acting on 2 points and let the adhesion sets in $X(u)$ be the blocks of an imprimitivity system of size 2 with respect to the action of $G(u)$. It is not hard to see that the scaffolding again consists of 8 vertices with a perfect matching between them, and the universal group is $S_2 \times S_2$ acting on this matching.
    %
    %Compare this to the fact that different graphs of groups can yield isomorphic fundamental groups.
\end{remark}

\section{Basic properties of the universal group}\label{SPermutations}

\subsection{Permutational properties}
Theorem~\ref{thm:acceptable-isomorphism} is the key to demonstrating some of the general properties of the universal group and its action on a scaffolding and the associated tree.  

\begin{theo}\label{TQuotient}
Let $\cG$ be a \gga\  and $G$ its universal group.  Set  $T = T_{\cG}$.   Then $T/G$ is isomorphic to $\Delta$.  
\end{theo}

\begin{proof}   Let $\pi_+: T\to \Delta_+$ denote a $\Star$-covering map and set $\pi=\rho\pi_+$.  
We only need to show that the fibers of $\pi$ are the orbits of the action of $G$ on $T$.  Since the actions of $G$ on any two scaffoldings are isomorphic we can choose $\Sigma$ as a canonical scaffolding as described in Sections~\ref{sec:canonical} and \ref{SInversions}.

Suppose ${v}$ and ${w}$ are vertices in $T$ such that $\pi({v})=\pi({w})$.  Let $f\colon \Sigma_v\to \Sigma_w$ be the map $(v,x)\mapsto (w,x)$.  This is an acceptable partial scaffolding isomorphism and Theorem~\ref{thm:acceptable-isomorphism} can now be applied to get an extension $F$ of $f$ to an acceptable automorphism $\Sigma\to \Sigma$.  Then $F$ is in the universal group $G$ and $F({v})={w}$. 

Suppose now that ${a}$ and ${b}$ are arcs in $T$ such that $\pi({a})=\pi({b})$.  Set ${v}=t({a})$ and $w=t({b})$.  Then $\pi({v})=\pi({w})$.  Let $f$ be an element in $G(\pi({v}))$ that takes the adhesion set $\Phi_{\pi_+({a})}(Y(\pi({a})))$ to the adhesion set $\Phi_{\pi_+({b})}(Y({\pi({b})})$.  Define a map $\Sigma_v\to \Sigma_w; (v,x)\mapsto (w, f(x))$.  Again apply Theorem~\ref{thm:acceptable-isomorphism} to get an acceptable extension $F$ of $f$ that is then an element of the universal group taking $\Sigma_{a}$ to $\Sigma_{b}$.

Thus the orbits of the universal group are just the fibers of $\pi$ and the quotient graph $T/G$ is isomorphic to the base graph $\Delta$.
\end{proof}

\begin{remark}
Start with a group $G$ acting on a tree $T$.  Set $\Delta=T/G$ and let $\pi: T\to \Delta$ denote the projection.  Then one can define an index for an arc $a$ in $\Delta$ as the number of elements in the orbit $G_{t(b)}b$, where $b$ is an arc in $T$ such that $a=\pi(b)$.  (Clearly this number does not depend on the choice of $b$.)  

Now let $\cG$ be a \gga\ and $G=\cU(\cG)$.
In Section~\ref{SDefinition} the \emph{index} $i(a)$ of an arc $a$ in the base graph $\Delta$ of $\cG$ as the number of $G(t(a))$-translates in $X(t(a))$ of $\Phi_a(Y(a))$.  This number is equal to the index of $G(t(a))_{\{\Phi_a(Y(a))\}}$ in $G(t(a))$.  Suppose now that $b$ is an arc in $T=T_{\cG}$ such that $\pi(b)=a$.  Then $i(a)$ is equal to the number of elements in the orbit of $b$ under $G_{t(b)}$ and thus $i(a)$ is equal to the index one gets from the action of $G$ on $T_{\cG}$.
\end{remark}

\begin{theo}\label{TLocal-action}
Let $G$ be the universal group of a \gga\ $\cG$ and let $\Sigma$ be a scaffolding for $\cG$. 
\begin{enumerate}
    \item The permutation action $(G^{[v]}, [v])$ is isomorphic to $(G({\pi(\bun{v}})), X(\pi(\bun{v})))$ for every $v\in \V \Sigma$.
    \item The permutation action $(G^{[a]}, [a])$ is isomorphic to $(H(\pi(\bun{a})), Y(\pi(\bun{a})))$ for every $a \in \A \Sigma$.
\end{enumerate}
\end{theo}

\begin{proof}
Again we assume that $\Sigma$ is a canonical scaffolding and $T=T_{\cG}$ the associated tree. 

    The first part is proved by applying Theorem~\ref{thm:acceptable-isomorphism} in the same way as at the start of the proof of Theorem~\ref{TQuotient}.  

   Let $a$ be an arc in $T$.  Note that the action $(G^{\Sigma_a}, \Sigma_a)$ induced by the fundamental group on the arc bundle at $a$ is determined by the action induced by $(G(t(\pi({a})))$ on $\Psi_{{a}}(Y(\pi({a})))$.  That action is isomorphic to the action induced by $G(\pi(t(a)))$ on $\Phi_{\pi({a})}(Y(\pi({a})))$ and is in turn isomorphic to the action $(G(\pi({a})), Y(\pi({a})))$.
\end{proof}

\begin{coro}
\label{coro:treeorbits}
Let $G$ be the universal group of a \gga\ $\cG$ and let $\Sigma$ be a scaffolding for $\cG$.   Set $T=T_{\cG}$.
\begin{enumerate}
    \item     The action of the universal group on $\Sigma$ is faithful but the action on the tree $T$ need not be faithful.
\item     The action of $G$ on $\Sigma$ is transitive on $\V \Sigma$ if and only if $\Delta$ has just one vertex $v$ and the action of $G(v)$ on $X(v)$ is transitive.  
The action of $G$ on $T$ is transitive on $\V T$ if and only if $\Delta$ has just one vertex.
\item The action of $G$ on $\V\Sigma$ has only finitely many orbits if and only if $\V\Delta$ is finite and $G(v)$ has only finitely many orbits on $X(v)$ for every $v\in \V\Delta$.
\end{enumerate}
\end{coro}

\begin{proof}
    The first part follows from the definition of the universal group. The second and third part follow directly from Theorems \ref{TQuotient} and \ref{TLocal-action}.
\end{proof}

\begin{coro}  Assume that for every $v \in \V \Delta$ and any pair of arcs $a,a' \in \Star_{\Delta}(v)$ the adhesion sets $\Phi_a(Y(a))$ and $\Phi_{a'}(Y(a'))$ are not $G(v)$-translates of each other.
If $\llbracket v\rrbracket$ is a vertex in $T$ the action of $G_{{\llbracket v\rrbracket}}$ on $\Star_T({\llbracket v\rrbracket})$ is isomorphic to the action of $G(\pi({\llbracket v\rrbracket}))$ on the set of adhesion sets in $X(\pi({\llbracket v\rrbracket}))$.     
\end{coro}

\begin{proof}
    The condition on the adhesion sets guarantees that the arcs in $\Star_{\Delta_+}(v)$ are in one-to-one correspondance to the adhesion sets in $X(v)$ for every $v \in \V \Delta$ and the result follows. 
\end{proof}

\begin{example}  \label{ESmall} Let $\cG$ be a \gga.   Let $\Sigma$ be a scaffolding and $T=T_{\cG}$ the associated tree.
    Theorem~\ref{TLocal-action} says that if $\bun{v}$ is a vertex in $T$ then the action of $G_{\bun{v}}$ on $[v]$ is isomorphic to the action of $G(\pi(\bun{v}))$ on $X(\pi(\bun{v}))$, but it is not guaranteed that $G(\bun{v})$ occurs as a subgroup of $G$.  This is demonstrated in the following example of a \gga\ $\cG$ and its universal group $G=\cU(\cG)$.

    The base graph for this \gga\ has just two vertices $v$ and $w$ and just two arcs that are the reverses of each other.  
    
    The action at $v$ is the action of $S_3$ on the set $X(v)=\{x, y, z\}\cup\{ 1, -1\}$ with $S_3$ acting in the usual way on $\{x,y, z\}$ and acting via the sign function on $\{-1, 1\}$.  The action at $w$ is the action of $C_4=\langle s\rangle$ on $X(w)=C_4\cup\{-1', 1'\}$ where $C_4$ acts regularly on the $C_4$ part and $s$ and $s^3$ swapping $-1'$ and $1'$ but $1$ and $s^2$ fixing both $-1'$ and $1'$.  
    
    If $a$ is the arc with $t(a)=v$ then the action
at $a$ is the non-trivial action of $C_2$ on the set $Y(a)=\{-1'', 1''\}$ and $\Phi_a(-1'')=-1$ and $\Phi_a(1'')=1$.  Also set $\Phi_{\ba}(-1'')=-1'$ and $\Phi_{\ba}(1'')=1'$.   

The augmented graph $\Delta_+$ equals $\Delta$ and $T_{\cG}$ is just the tree with two vertices.  A canonical scaffolding has the union of $X(v)$ and $X(w)$ as a vertex set and then arcs $(-1,-1'), (1,1'), (-1',-1), (1', 1)$.  Each of the sets $X(v)$ and $X(w)$ constitutes a vertex bundle.  The universal group $G$ is a subgroup of $S_3\times C_4$.  

Suppose $g$ is an element of the universal group with action $g_v\in S_3$ on $X(v)$ and action $g_w\in C_4$ on $X(w)$.  In the case that $g_v$, the local action at $v$, is an even permutation then both $-1$ and $1$ are fixed and thus $g_w$ is either equal to $1$ or $s^2$.  Thus $g$ must have order $1, 2, 3$ or $6$, and the only way to get an element of order 2 is to have $g_v=1$ and $g_w=s^2$.  If $g_v$ is an odd permutation then $g_v$ has order 2 and $g_w$ has to be equal to either $s$ or $s^3$ and thus $g$ would have order 4.  Now it is clear that the group $G(v)=S_3$ does not embed in to $G$ since $G$ has only one element of order 2.    
\end{example}

Let $G$ be a group acting on a set $X$.  The \emph{suborbits} of $G$ are the orbits of stabilizers of points.  The \emph{subdegrees} of the action are defined as the sizes of the suborbits.  We say that the action is \emph{subdegree finite} if all the subdegrees are finite, i.e.\ if $x\in X$ then every orbit of $G_x$ is finite.

\begin{prop}\label{PSubdegree-finite}
    Let $\cG$ be a \gga\ and $G=\cU(\cG)$.  
    \begin{enumerate} 
        \item The action of $G$ on $T=T_{\cG}$ is subdegree finite if and only if $G(v)$ acting on the adhesion sets in $X(v)$ has only finite orbits for every $v$ in $\V \Delta$. 
        \item The action of $G$ on a scaffolding $\Sigma$ is subdegree finite if and only if for every $v\in \V\Delta$ the action $(G(v), X(v))$ is subdegree finite and all orbits of adhesion sets in $X(v)$ under $G(v)_x$ are finite.  
    \end{enumerate}  
\end{prop}

\begin{proof}   
For the first part, we note that if the action on the tree is subdegree finite, then the stabiliser of any vertex of $T$ has finitely many orbits on the neighbours of that vertex. Since the neighbours are in one-to-one correspondence with the adhesion sets (see Remark \ref{rmk:arcbundlematching}), this shows that all orbits of adhesion sets are finite.

For the converse implication, let $\llbracket v \rrbracket$ be a vertex of $T$ and for every vertex $\llbracket u \rrbracket$ of $T$ denote by $p(\llbracket u \rrbracket)$ the parent of $\llbracket u \rrbracket$ when the tree is rooted at $\llbracket v \rrbracket$. Clearly, if some element in $G_{\llbracket v \rrbracket}$ maps $\llbracket u \rrbracket$ to $\llbracket u \rrbracket$, then it must map $p(\llbracket u \rrbracket)$ to $p(\llbracket u' \rrbracket)$. It follows that 
\[
| G_{\llbracket v \rrbracket} \llbracket u \rrbracket | \leq |G_{\llbracket v \rrbracket} p(\llbracket u \rrbracket)| \cdot |G_{p(\llbracket u \rrbracket)} \llbracket u \rrbracket|.
\]
Since the orbits on adhesion sets are assumed to be finite, $|G_{p(\llbracket u \rrbracket)} \llbracket u \rrbracket|$ is finite for all $\llbracket u \rrbracket$ and it follows by induction on the distance between $\llbracket u \rrbracket$ and $\llbracket v \rrbracket$ that $| G_{\llbracket v \rrbracket} \llbracket u \rrbracket |$ is finite for all $\llbracket u \rrbracket$.

For the second part, if we assume that the action on a scaffolding is subdegree finite then Theorem~\ref{TLocal-action}, implies that the action of $G(v)$ on $X(v)$ must be subdegree finite for all $v\in \V\Delta$. Moreover, if $G(v)_x$ had an infinite orbit on the adhesion sets for some $x \in X(v)$, then there would be some vertex $u \in \V \Sigma$ such that $G_u$ has an infinite orbit on the neighbours of $\llbracket u \rrbracket$ in $T$, and thus also an infinite orbit in its action on $\V \Sigma$.

Now assume that the conditions on the action of $G(v)$ on $X(v)$ hold for all $v\in \V\Delta$. Let $v$ be a vertex of $\Sigma$ and (as in the first part) root the tree $T$ in $\llbracket v \rrbracket$. Let $u \in \V \Sigma$. If $u \in [v]$, then $G_v u$ is finite by assumption. Otherwise, let $ \llbracket w \rrbracket =  p(\llbracket u \rrbracket)$, and let $u_0$ be a vertex of $[u]$ which has a neighbour in $[w]$. Note that by Remark \ref{rmk:arcbundlematching} this neighbour is unique and denote it by $w_0$. Now, as in the first part we note that
\[
    |G_v u| \leq |G_v w_0| \cdot |G_{w_0} u_0| \cdot |G_{u_0} u|.
\]
We note that $|G_{w_0} u_0|$ is finite by the assumption on finiteness of orbits of adhesion sets, and $|G_{u_0} u|$ is finite by subdegree finiteness of the vertex groups. It follows by induction on the distance between $\llbracket u \rrbracket$ and $\llbracket v \rrbracket$ that $| G_{\llbracket v \rrbracket} \llbracket u \rrbracket |$ is finite for all $\llbracket u \rrbracket$.
\end{proof}

\begin{remark}
    We note that the condition on orbits of adhesion sets in the second part is indeed necessary. For instance, the wreath product $S_2 \wr \mathbb Z$ in its imprimitive action on $\{1,2\} \times \mathbb Z$ is subdegree finite, but the orbit of the set $\{1\} \times \mathbb Z$ under any point stabiliser is (uncountably) infinite. Thus, if this was one of the adhesion sets in a \gga, then the action of the universal group on the scaffolding would not be subdegree finite; in fact the stabiliser of a vertex $v$ would even have uncountable orbits on the neighbours of $\llbracket v \rrbracket$ in $T$.
\end{remark}

In case all vertex groups act on finite sets, the conditions of Theorem \ref{PSubdegree-finite} are trivially satisfied, and we get the following corollary.

\begin{coro}
Suppose $\cG$ is a \gga\ and assume that the set $X(v)$ is finite for every $v\in \V\Delta$.  Then the action of the universal group on the tree $T_{\cG}$ and the action on a scaffolding are both subdegree finite. 
\end{coro}

\subsection{Properties of the permutation topology}

   When a group $G$ acts on a set $X$ one can endow $G$ with the permutation topology as explained in Section~\ref{STopology}.    In this section we study the interplay between the properties of the actions given in the \gga\ and the properties of the permutation topologies the universal group gets from the action on a scaffolding and the action on the associated tree. 
   
   The first lemma establishes a connection between the topologies arising from these two actions, respectively.

\begin{lemm}
    Let $\cG$ be a \gga, $G$ its universal group and $\Sigma$ a scaffolding for $\cG$.   Set $T=T_\cG$.  The topology arising from the action of $G$ on the tree $T$ is contained in the topology arising from the action of $G$ on a scaffolding.
\end{lemm}

\begin{proof}
    Consider the topology arising from the action of $G$ on $T$.  A neighbourhood basis of the identity is given by the groups $G_{(F)}$ where $F$ is a finite set of vertices in $T$.  Each vertex in $T$ corresponds to a vertex bundle in $\Sigma$.  From each vertex bundle corresponding to a vertex of $T$ that is in $F$ select a vertex and let $F'$ be the set of these vertices.  Clearly $G_{(F')}\leq G_{(F)}$ and the subgroup $G_{(F')}$ is open in the topology on $G$ arising from its action on the scaffolding $\Sigma$.  Thus the subgroup $G_{(F)}$ is open in the topology arising from the action on $T$ and the result follows.
\end{proof}

Since much of this paper is motivated by the study of totally disconnected, locally compact groups, we next focus on these two properties. 

\begin{prop}\label{PClosed}
    Let $\cG$ be a \gga, $G$ its universal group and $\Sigma$ a scaffolding for $\cG$. Then the following hold:
    \begin{enumerate}
        \item Both the group of scaffolding automorphisms and the universal group are totally disconnected with respect to the topology induced by the action on $\Sigma$.
        \item The group $\autsc(\Sigma)$ of all scaffolding automorphisms is closed  with respect to the topology induced by the action on $\Sigma$. 
        \item  If the permutation group $G(v)$ is closed in its action on $X(v)$ for every $v\in \V\Delta$ then the universal group $G$ is closed in its action on $\Sigma$.
    \end{enumerate}
\end{prop}

\begin{proof}   
The permutation topology is totally disconnected if and only if the action is faithful, see \cite[Section 2.1]{Moller2010} and the first part follows.  Note that the actions of the group of scaffolding automorphisms and the universal group on the associated tree $T_{\cG}$  need not be faithful, see Example~\ref{ESmall}.  

  In the second part we have to show that if $(g_i)_{i\in\NN}$ is a convergent sequence in the group of scaffolding automorphisms with limit $g$ then $g$ is a scaffolding automorphism.    Suppose $x$ and $y$ are vertices in $\Sigma$ belonging to the same vertex bundle.  Then find a number $i$ such that $g(x)=g_i(x)$ and $g(y)=g_i(y)$.  Since $g_i$ is a scaffolding automorphism we see that $g(x)$ and $g(y)$ belong to the same vertex bundle.  Hence $g$ maps vertex bundles to a vertex bundles. Similarly one can show that $g$ must map adjacent vertices to adjacent vertices and perserve the map $\pi$. Thus $g$ is a scaffolding automorphism and therefore the group of scaffolding automorphisms is closed.

    For the third part consider a sequence $(g^{(i)})_{i\in\NN}$ in the universal group $G$ converging in $\sym(\V\Sigma)$ to an element $g$.   From the above we see that $g$ is a scaffolding automorphism.   Our task is to show that $g$ is acceptable.  Let $[v]$ be a vertex bundle.  Clearly the sequence $g_{\bun{v}}^{(i)}=pg^{(i)}(p_{\bun{v}})^{-1}$ converges to $g_{\bun{v}}=pg(p_{\bun{v}})^{-1}$.    Since the elements in the sequence are all contained in $G(\pi(\bun{v}))$ and since the group $G(\pi(\bun{v}))$ is closed in the permutation topology arising from its action on $X(\pi(\bun{v}))$ follows that $g_{\bun{v}}\in G(\pi(\bun{v}))$.  Thus $g$ is an acceptable scaffolding automorphism of $\Sigma$ and is in the universal group, which is hence closed.
\end{proof}

Another natural question is if one can find conditions on the graph of group actions guaranteeing that the universal group is locally compact with respect to the permutation topology arising from its action on a scaffolding graph.  The following lemma characterizes those subsets of a closed permutation group that have compact closure.
Versions of the lemma are well known, see e.g.~\cite[Lemma~2]{Woess1991} and \cite[Lemma~2.2]{Moller2010}, but in both of these cases the setting is more restrictive.  The second part of the proof below is from \cite[Section 2.2]{Moller2010}.

\begin{lemm}\label{LCompact}
Let $G$ be a closed permutation group acting on a set $X$.  A subset $U\subseteq G$ has compact closure if and only if all orbits of $U$ are finite.    
\end{lemm}

\begin{proof}
Assume $U$ has compact closure.  Suppose the orbit $Ux=\{x_i\}_{i\in I}$ is infinite.  For $i\in I$ define $U_i=\{g\in G\mid g(x)=x_i\}$.  The sets $U_i$ are all open in the permutation topology and they form an open covering of the closure of $U$, but there clearly exists no finite subcovering.  This contradicts the assumption that the closure of $U$ is compact and thus every $U$-orbit must be finite.

Suppose that every $U$-orbit is finite.
Let $(X_i)_{i\in I}$ denote the family of $U$-orbits on
$X$.  Define $H_i=G^{X_i}$ as the permutation group induced by
$G$ on $X_i$.  Each $X_i$ is finite, so each
group $H_i$ is finite and thus discrete and compact in permutation topology.  
Set $H=\prod_{i\in I} H_i$, the full Cartesian product.  The set
$X$ is the disjoint union of the $X_i$'s and $H$ has a
natural action on $X$.  The
permutation topology on $H$ is just the product topology.  From Tychonoff's Theorem it follows that $H$ is compact.  Thinking of $U$ and $H$ as subsets of $\sym(X)$ we see that the closure of $U$ is a closed
subset of the compact set $H$ and is thus compact.  
\end{proof}

Using the above Lemma we can translate Proposition~\ref{PSubdegree-finite} into conditions implying the local compactness of the universal group in the permutation topologies arising on one hand from the action on the associated tree and on the other hand from the action on a scaffolding.

\begin{prop}\label{PLocallyCompact}
  Let $\cG$ be a \gga\ and $G=\cU(\cG)$.  Let $\Sigma$ be a scaffolding for $\cG$ and let $T=T_{\cG}$ be the associated tree.
\begin{enumerate}
    \item Assume $G$ is closed in the toplogy arising from its action on $\Sigma$.
    Suppose that
    for every vertex $v\in \V\Delta$ the group $G(v)$ has only finite orbits on the set of adhesion sets in $X(v)$, and every adhesion set in $X(v)$ contains a finite subset $F$ such that $G(v)_{(F)}$ is compact.  Then $G$ is locally compact in the topology arising from its action on a scaffolding. 
\item Suppose that $G$ is closed in the topology arising from the action of $G$ on the tree $T$.  If for every vertex $v\in \V\Delta$ the group $G(v)$ has only finite orbits in its action on the set of adhesion sets in $X(v)$ then $G$ is locally compact in the topology arising from its action on $T$.
\end{enumerate}
\end{prop}

\begin{proof}
Assume that $\cG$ satisfies the conditions in the first part. Pick a finite subset $F\subseteq X(v)$ such that $G(v)_{(F)}$ is compact, that is, $G(v)_{(F)}$ has only finite orbits on $X(v)$.  Then the argument in the proof of part 2 in Proposition~\ref{PSubdegree-finite} shows that if $F'$ is a subset of some vertex bundle whose image under $p$ is $F$ then $G_{(F')}$ has only finite orbits on the vertex set of $\Sigma$.  Thus $G_{(F')}$ has compact closure in $G$ and since $G_{(F')}$ is open in $G$ it is also closed and we conclude that $G_{(F')}$ is compact.   Hence $G$ is locally compact.

For the second part the conditions imply that $G$ in it action on the vertex set of the tree is subdegree finite and it follows that $G$ is locally compact by the same argument as before. 
\end{proof}

The next property we look at is the question of when the universal group of a \gga\ is compactly generated.   

\begin{prop}\label{PCompactly-generated}
    Let $\cG$ be a \gga\ such that each of the groups $G(v)$ is compactly generated, closed and subdegree finite. 
    Suppose further that $\V\Delta$ and $\E\Delta$ are finite and that $G(v)$ has only finitely many orbits on $X(v)$ for every $v\in \V\Delta$.  
    
    Then the universal group $G$ is compactly generated with respect to the permutation topology arising from its action on $\Sigma$. 
\end{prop}

\begin{proof}  From Proposition~\ref{PLocallyCompact} it follows that $G$ is locally compact.  It is known that if $G$ is a totally disconnected, locally compact group that acts on a connected, locally finite graph $\Gamma$ with finitely many orbits then $G$ is compactly generated, see \cite[Corollary~2.10]{KronMoller2008}. Conversely, if $G$ is compactly generated and $U$ is any compact open subgroup of $G$ then one can introduce a graph structure on $X=G/U$ so that we get a locally finite, connected graph that $G$ acts on as a group of automorphisms (see, e.g.\ \cite[Theorem 2.2]{KronMoller2008}).

Let $v$ be a vertex in $\Delta$. If we let $U$ be a point stabiliser $G(v)_x$, then $G/U$ is in one-to-one correspondence with the orbit of $x$ in $X(v)$.
Thus, for each orbit of $G(v)$ on $X(v)$ it is possible to introduce a graph structure making the orbit into a connected, locally finite graph that $G(v)$ acts transitively on.  For each pair of distinct orbits we choose vertices $x$ and $y$, one from each orbit, and then add as arcs the elements in the orbits $G(v)(x,y)$ and $G(v)(y,x)$. Subdegree finiteness and the fact that there are only finitely many orbits imply that in doing so we are only adding finitely many edges incident to any given vertex. When all of this is done we have hence defined a connected, locally finite graph $\Gamma(v)$ with vertex set $X(v)$.  The group $G(v)$ acts on $\Gamma(v)$ with finitely many orbits.   

Let now $\Sigma$ be a canonical scaffolding.  Construct a graph $\Gamma$ by adding to $\Sigma$ all the arcs $((v,x), (v,y))$ where $(x,y)$ is an arc in $\Gamma(\pi(v))$.     If $v\in \V T$ the subgraph spanned by the set $\{(v,x)\mid x\in X(\pi(v))\}$ is isomorphic to $\Gamma(v)$.  Subdegree finiteness together with the assumption that $\E \Delta$ is finite guarantees that every point in $X(v)$ is only contained in finitely many adhesion sets. Thus the graph $\Gamma$ is locally finite and $G$ acts on $\Gamma$ with finitely many orbits.  It is also clear that $\Gamma$ is connected since every vertex bundle is connected and scaffolding edges connect vertex bundles which are adjacent in the tree $T$.  We can thus conclude that $G$ is compactly generated. 
\end{proof}

The following corollary is a direct consequence of Propositions~\ref{PClosed}, \ref{PLocallyCompact} and \ref{PCompactly-generated}.

\begin{coro}\label{CFinite}
Let $\cG$ be a \gga\ with finite base graph $\Delta$ such that for every vertex $v$ in the base graph $\Delta$ the group $G(v)$ is finite.  Let $G$ be the universal group of $\cG$ and set $T=T_{\cG}$.
\begin{enumerate}
    \item The group $G$ is closed both in the topology arising from its action on a scaffolding and the topology arising from its action on $T$.
    \item The group $G$ is locally compact both in the topology arising from its action on a scaffolding and the topology arising from its action on $T$.
    \item The group $G$ is compactly generated both with respect to the topology arising from its action on a scaffolding and the topology arising from its action on $T$.
\end{enumerate}
\end{coro}

\section{Graphs of groups:  Bass--Serre theory}\label{SBass-Serre}

The Bass--Serre theory of groups acting on trees is one of the corner stones of geometric group theory.   
The basic concepts in Bass--Serre theory are \emph{graphs of groups} and the \emph{fundamental group of a graph of groups}.  The definition of a \gga\ is, as explained in the introduction, modelled in the definition of a graph of groups, but the the definition of a universal group of a \gga\ does not resemble the definition of the fundamental group of a graph of groups.   The aim in this section is to show how the fundamental group of a graph of groups can be constructed as the universal group of a \gga.  

\begin{defi}\label{DGraph-of-group}  A \emph{graph of groups} consists of a connected, non-empty graph $\Delta$, without self-reverse arcs, and for each vertex $v$ in $\Delta$ there is a given group $G(v)$ and for each arc $a$ there is a given group $H(a)$  such that $H(\ba)=H(a)$ and for each arc $a$ in $\Delta$ there is also given an injective group homomorphism $\theta_a: H(a)\to G(t(a))$.
\end{defi}

To each graph of groups we can construct a goup acting on a tree, the fundamental group, as follows. 

Start with a graph of groups.   Use the indicies of the subgroups $\theta_a(H(a))$ in $G(t(a))$ to define a graph $\Delta_+$ similar to the augmented base graph defined in Section~\ref{SDefinition}.  Then construct a $\Star$-covering tree $T$ for $\Delta_+$. with $\pi_+: T\to \Delta_+$ the associated $\Star$-covering map.  As in Section~\ref{SStar-covering} we may assume that $\pi_+$ is compatible with the graph structure of $\Delta$ and we get a graph morphism $\pi: T\to \Delta$.    This tree is clearly isomorphic to the tree constructed in \cite[Section~5.3]{Serre2003}.  
The fundamental group of a graph of groups acts on this tree such that quotient graph is equal to $\Delta$.  The stabilizer of a vertex $v$ on $T$ is isomorphic to $G(\pi(v))$ and the stabilizer of an arc $a$ in $T$ is isomorphic to $G(\pi(a))$.   Conversely, given a group $G$ acting  without inversion on a tree $T$ one can construct a graph of groups such that the action of the fundamental group on the associated tree is isomorphic to the action of $G$ on $T$, see \cite[Section 5.4]{Serre2003}.  

From a given graph of groups we construct a \gga\ $\cG$ such that the underlying graph is also $\Delta$.  For a vertex $v$ in $ \Delta$ the group action at $v$ is the regular action of $G(v)$ on itself and for an arc $a$ the group action is the regular action of $H(a)$ on itself.  The embedding of $(H(a), H(a))$ into $(G(t(a)), G(t(a)))$ is given by the group homomorphism $\theta_a$.   Let $T=T_\cG$ denote a $\Star$-covering tree for $\Delta_+$ and let $\pi: T\to \Delta$ be the associated map.  Denote by $\Sigma$ a scaffolding for $\cG$ and let $G$ denote the universal group of the \gga.  

The quotient graph $T/G$ is isomorphic to $\Delta$ by Corollary~\ref{TQuotient}.  Let $v$ be a vertex in $\Delta$ and $\llbracket w \rrbracket$ a vertex in $T$ such that $\pi(\llbracket w \rrbracket)=v$.   Suppose that $g\in G$ fixes a vertex in the vertex bundle $\llbracket w \rrbracket$.  Then $g$ fixes all the vertices in $\llbracket w \rrbracket$, because the action of $G_{\llbracket w\rrbracket}$ on $\llbracket w \rrbracket$ is isomorphic to the regular action $G(v)$ on on itself and thus regular.  This means that $G_{\llbracket w\rrbracket}$ fixes all the neighbours of $\llbracket w \rrbracket$ in $T$.  Suppose $\llbracket u \rrbracket$ is a vertex in $T$ adjacent to $\llbracket w \rrbracket$.  If $x \in [w]$ and $y \in [u]$ are adjacent vertices in $\Sigma$ then $y$ is the unique vertex in $[u]$ that is adjacent to $x$.  Thus $g$ fixes $y$.  Since $G_{\llbracket u \rrbracket}$ acts regularly on $[u]$ we see that $g$ fixes every vertex in $\llbracket u \rrbracket$.  Induction shows that $g$ acts trivially on $\Sigma$ and acts therefore trivially on $T$.  It follows that if $g$ is an element that fixes the vertex $\llbracket w\rrbracket$ in $T$ then the action of $g$ on the $T$ is determined by the action of $g$ on $[w]$.  Thus $G_{\llbracket w\rrbracket}$ is isomorphic to $G(v)$.  We also see in this way that if $\llbracket a\rrbracket$ is an arc in $T$ then the subgroup of $G$ fixing this arc fixes both the vertices $o(\llbracket a\rrbracket)$ and $t(\llbracket a\rrbracket)$ and is isomorphic to $G(\pi(\llbracket a\rrbracket))$.

From this and Corollary~\ref{TQuotient} we see that the graph of groups we construct from the action of $G$ on $T$ is identical to the graphs of groups we started with.  From \cite[Theorem 13]{Serre2003} it follows that $G$ is isomorphic to the fundamental group of the given graph of groups.
We have thus proved the following theorem.

\begin{theo}\label{TBass-Serre}
Suppose a graph of groups is given.  Construct a \gga\ as above.  Then the action of the fundamental group of the graph of groups on its associated tree is isomorphic to the action of the universal group of the \gga\  on its associated tree.  In particular the fundamental group of the graph of groups is isomorphic to the universal group of the \gga.
\end{theo}

As every group acting on a tree without inversion is the fundamental group of some graph of groups, this shows that any such group is also the universal group of some \gga. 

\begin{coro}
    Let $G$ be a group acting faithfully on a tree $T$.  Then there exists a \gga\ $\cG$ such that the action of $G$ on $T$ is isomorphic to the action induced by $\cU(\cG)$ on $T_\cG$.  
\end{coro}

\begin{proof}   If the action on the tree has inversions then we can replace $T$ with its barycentric subdivision we get an action on a tree without inversions.  Of course we only need to subdivide those edges that are inverted.  The group action on this tree is isomorphic to the action of the fundamental group of some graph of groups and thus isomorphic to the action of the universal group of a \gga\ on its associated tree.  The base graph of the \gga\ has no self-reverse arcs.  The subdivision vertices introduced to the tree have degree 2 and the stabilizers of these vertices act transitively on the two adjacent edges.  The corresponding vertices in the base graph of the \gga\ also have degree 2 and they can be removed by reversing the construction in Section~\ref{SInversions} to get a \gga\ where the action of the universal group on its associated tree is isomorphic with the action that we started with.       
\end{proof}

The above argument also gives a partial converse to Theorem~\ref{TBass-Serre}.

\begin{coro}
    Let $\cG$ be a \ggawo.   
    Assume that the pointwise stabiliser of each adhesion set (in the respective vertex group) is trivial.
    
    Define a graph of groups with the same base graph as $\cG$ as follows. The group at a vertex $v$ is the group $G_v=G(v)$, the group at an arc $a$ is the group $G_a=G(a)$ and the group homomorphism $\theta_a: G_a\to G_{t(a)}$ is induced by the embedding of the action $(G(a), Y(a))\to (G(t(a))), X(t(a)))$.  
    
    Then the action of the universal group of the \ggawo\ on the associated tree is isomorphic to the action of the fundamental group of this graph of groups on the tree.
\end{coro}

\begin{proof}
    First note that the condition on pointwise stabilizers of adhesion sets implies that the group that $G(v)$ induces on the adhesion set is equal to the setwise stabilizer of the adhesion set.  Thus, the embedding of group actions $(G(a), Y(a))\to (G(t(a))), X(t(a)))$ induces an embedding of $G(a)$ into $G(t(a))$.   Thus we have a graph of groups and the argument above shows that the action of the universal group of the \ggawo\ on the associated tree is isomorphic to the action of the fundamental group of this graph of groups on the associated tree.
\end{proof}

\section{Burger--Mozes groups and relatives}

\subsection{Burger--Mozes groups}\label{SBurger-Mozes}

\begin{defi} (Definition of Burger-Mozes groups, \cite[Section 3.2]{BurgerMozes2000}.)  \label{DBurger-Mozes}
Let $F$ be a subgroup of $\sym(X)$ for some set $X$ (in \cite{BurgerMozes2000} it seems to be assumed that $X$ is finite, but the construction works without that assumption).  Define $T$ as the $|X|$-regular tree.  A \emph{legal colouring} of $T$ is a map $c: \A T\to X$ such that (i) $c(\ba)=c(a)$ for every $a\in \A T$ and (ii) the restriction of $c$ to $\Star(v)$ is a bijection.
The (Burger--Mozes) universal group $U(F)$ is defined as the group of all automorphisms $g$ of $T$ such that $cg(c_{|\Star(v)})^{-1}$ is in $F$ for all vertices $v$ in $T$.
\end{defi} 

\medskip

Here we construct a \gga\ such that the action of the universal group on the associated tree is isomorphic to the action of the group $U(F)$ on the tree $T$ and this isomorphism of group actions is given by an isomorphism of the trees.  Let $\{O_i\}_{i\in I}$ denote the family of orbits of $F$ on $X$.  
The base graph $\Delta$ of our \gga\ has just one vertex $v$ and for each orbit $O_i$ one self-reverse loop $a_i$.  The action at $v$ is  the action of $F$ on $X$.  Choose a representative $x_i$ from each orbit.  For the arc $a_i$ in $\Delta$  we set $H(a_i)=\{1\}$, $Y(a_i)=\{x_i\}$; since $H(a_i)$ is trivial, the only possible inversion agent is $h_{a_i}=1$.  The map $\Phi_{a_i}: Y(a_i)\to X$ sends $x_i$ to $x_i$.  Now we have defined a \gga.  Recall that the index $i(a_i)$ is the number of translates of the adhesion set.  Here the adhesion set is $\{x_i\}$ and the number of adhesions sets is $|O_i|$.  Thus the augmented base graph $\Delta_+$ has one vertex $v$ and $|X|$ self-reverse loops and these loops can be labelled with the elements of $X$.  The $\Star$-covering tree is the $|X|$-regular tree $T$.  Let $\pi$ denote a $\Star$-covering map.  
A canonical scaffolding graph $\Sigma$ for this \gga\ has $\V T\times X$ as a vertex set.   Each arc bundle in $\Sigma$ has just one single arc so the arcs in $\Sigma$ are in one-to-one correspondence with the arcs in $T$.  If $a$ is the arc in $T$ then there is $x\in X$ such that the corresponding arc in $\Sigma$ has $(t(a), x)$ as a terminus and $(o(a), x)$ as an origin and $\pi(a)$ is the arc in $\Delta_+$ labelled with $x$.   Thus the map $c: \A T\to X$ that maps an arc in $T$ to the label of the arc $\pi(a)$ in $\Delta_+$ is a legal colouring in the sense of Definition~\ref{DBurger-Mozes}.  The universal group $\cU(\cG)$ is clearly equal to the universal group $U(F)$ because the condition that a scaffolding automorphism is acceptable translates into the condition that $cg(c_{|\Star(v)})^{-1}$ is in $F$ for all vertices $v$ in $T$. We have thus proved the following result.

\begin{prop}
    The Burger-Mozes group $U(F)$ is equal to the universal group $\cU(\cG)$.
\end{prop}

\subsection{Smith's box product}

The box product was introduced by Smith in in \cite{Smith2017} as a means to construct uncountably many totally disconnected, locally compact and compactly generated simple groups. Property (P) (discussed in section \ref{SPropertyP} below) and Tits' theorem \cite[Th\'eor\`eme 4.5]{Tits1970} lie at the heart of the construction.  The box product can be defined as follows.

\begin{defi}
Let $M$ and $N$ be non-trivial permutation groups acting on sets $X$ and $Y$, respectively.  Let $T$ be the $(|X|, |Y|)$-semi-regular tree, that is, an infinite tree where every vertex in one part of the bipartition has degree $|X|$, and every vertex in the other part has degree $|Y|$. Let $\V_X$ and $\V_Y$ denote the parts of the bipartition of $T$, where vertices in $\V_X$ have degree $|X|$ and vertices in $\V_Y$ have degree $|Y|$.  A \emph{legal colouring} of the arcs of $T$ is a map $c: \A T\to X\cup Y$ such that if $v\in \V_X$ then the restriction to $\Star(v)$ is a bijection to $X$ and if $v\in \V_Y$ then the restriction to $\Star(v)$ is a bijection to $Y$ and, furthermore, the restriction of $c$ to $\Star^{-1}(v)$ is constant for every $v\in \V T$.  The group $U(M,N)$ is then defined as the group of all automorphisms $g$ of $T$ such that $cg(c_{|\Star(v)})^{-1}$ is in $M$ if $v\in\V_X$ and is in $N$ if  $v\in V_Y$.    The \emph{box product} of $M$ and $N$ is defined as the permutation group $U(M,N)$ acting on $\V_Y$. 
\end{defi}

\begin{remark}
Smith describes the action of a group $G$ on the neighbourhood $N(v)$ of a vertex $v$ in terms of the set $o^{-1}(v)$ (\lq\lq the set of out-going arcs\rq\rq) but in the the present work this is described in terms of $\Star(v)=t^{-1}(v)$ (\lq\lq the set of in-going arcs\rq\rq). This is a deliberate choice ensuring that embeddings of arc groups ``follow the direction'' of the arc.

Hence, in our description of the box product construction, directions of arcs have been reversed to fit with our terminology. An analogous remark applies to the local action diagram construction below.
\end{remark}

Here is how we can define $U(M,N)$ using graphs of group actions: 
Smith shows, \cite[Lemma 22]{Smith2017}, that the quotient graph $T/U(M, N)$ is isomorphic to complete bipartite graph $K_{m,n}$ where $m$ denotes the number of orbits of $M$ on $X$ and $n$ denotes the number of orbits of $N$ on $Y$.   The base graph for our graph of group actions is $\Delta=K_{m,n}$.  Let $V_m$ denote the set of vertices of $\Delta$ that have degree $m$ and let $V_n$ denote the set of vertices that have degree $n$.  Clearly, $|V_m|=n$ and $|V_n|=m$.  The group action at each vertex in $V_m$ is set as $(M, X)$ and the action at a vertex in $V_n$ is $(N, Y)$. 
Let $\{x_v\}_{v\in V_n}$ be a family of representatives for the orbits of $M$ on $X$ and similarly let $\{y_u\}_{u\in V_m}$ be a family of representatives for the orbits of $N$ on $Y$.   For an arc $a$ in $\Delta$ such that $v=t(a)\in V_m$ and $u=o(a)\in V_n$ we associate a group action  $(\{1\}, \{z_{v,u}\})$ and maps $\Phi_a: \{z_{v,u}\}\to \{x_v\}$ and $\Phi_{\ba}:\{z_{v,u}\}\to \{y_u\}$.  

Analogously to the Burger--Mozes construction in the previous section, it can be shown that that the action of $U(M,N)$ on the $(|X|, |Y|)$-semi-regular tree is isomorphic to the action of the universal group of this graph of group actions on the associated tree. We leave the details to the reader.

\begin{prop}
    The group $U(M,N)$ is equal to the universal group $\cU(\cG)$.
\end{prop}

\subsection{Local action diagrams}\label{SLocalActionDiagrams}

The \emph{local action diagrams}  defined by Reid and Smith in \cite[Definition~3.1]{ReidSmith2020} generalise both constructions described above. They resemble in many ways our graph of group actions, and indeed can be seen as a special case of our construction. 

\begin{defi}  (Definition of local action diagrams, \cite[Definition~3.1]{ReidSmith2020}.)  A local action diagram consists of a graph $\Gamma$, a set $X(a)$ for each arc in $\Gamma$ and for each vertex $v$ a group $G(v)$ that is a subgroup of $\sym(X(v))$ where $X(v)$ is the disjoint union of the sets $X(a)$ for all arc $a$ such that $t(a)=v$ and each one of the sets $X(a)$ is an orbit of $G(v)$.  
\end{defi}

From the data provided by a local action diagram Reid and Smith construct a tree $T$ together with a surjective graph homomorphism $\pi: T\to \Gamma$ and a map $\cL: \A T\to \bigsqcup_{a\in \A\Gamma} X(a)$ such that if $v$ is a vertex in $T$ and $a$ is an arc in $\Gamma$ such that $t(a)=\pi(v)$ then the restriction $\cL_{v,a}$ of $\cL$ to $\{b\in \Star(v)\mid \pi(b)=a\}$ is a bijection to $X(a)$.  The \emph{universal group} of the local action diagram is the group of all automorphisms $g$ of $T$ such that for every vertex $v$ in $T$ the permutation $\cL_{g(v)}g\cL_v^{-1}$ of $X(\pi(v))$ is in $G(\pi(v))$.   (In \cite{ReidSmith2020} it is also assumed that $G(v)$ is closed in the permutation topology induced by its action on $X(v)$.  This additional assumption is not needed in our discussion here.)

\medskip

Constructing a graph of group actions containing the same information as a local action diagram is straightforward.  First define $\Delta$ as the graph that has the same vertex and arc sets as $\Gamma$.  The group action at a vertex $v\in \V \Delta$ is just the action of $G(v)$ on $X(v)$ and the action at an arc $a$ is the action of the trivial group on a singleton set $Y(a)=\{x_a\}$ where $x_a$ is some element from $X(a)$.  If the arc $a$ is self-reverse then the agent of inversion is $h_a=1$.  Define $\Phi_a:  Y(a)\to X(t(a))$ as the map that sends $x_a$ to $x_a$.  It is possible to show that the universal group of this graph of group actions coincides with the universal group of the action diagram.

The attentive reader will have noticed that in all three constructions discussed in this section the arc groups in the resulting \gga\ were trivial. We call such a \gga\ free and note that the groups that can be defined by a free \gga\ are precisely the groups that can be defined by local action diagrams.

\begin{defi}
A \gga\ such that the action at every arc is just the trivial group acting on a set with one element is said to be \emph{free}.
\end{defi}

\begin{prop}\label{PTitsProperty}
Let $G$ be a group acting faithfully on a tree $T$.  The action $(G,T)$ is isomorphic to the action of the universal group of a local action diagram by an isomorphism of the trees if and only if it is isomorphic to the action of the universal group of a free \gga\ on the associated tree by an isomorphism of the trees.  
\end{prop}

The proof of this proposition is deferred to Section~\ref{SReduced} as it will be much easier to state with some additional definitions in place.

\subsection{Tits' Property (P) and Property (P$_k$)}\label{SPropertyP}

\begin{defi}  (\cite[(4.2)]{Tits1970})   \label{DPropertyP}
    For a path $C$ (finite or infinite) in $T$ define the projection onto $C$ as a map $p_C:\V T\to C$ such that for a vertex $v$ in $T$ the vertex $p_C(x)$ is the vertex in $C$ that is closest to $x$.    Set $F_C=G_{(C)}$.  Furthermore, for $x$ a vertex in $C$ define $T_{C,x}$ as the set $p_C^{-1}(x)$ and then define $F_{C,x}$ as the group $G_{(C)}^{T_{C, x}}$.  (The subtree spanned by $T_{C,x}$ could be called the branch at $x$ relative to $C$ and $F_{C, x}$ is the permutation group on $T_{C, x}$ induced by $G_{(C)}$.)  Restricting to the subsets $T_{C,x}$ gives a group homomorphism $F_C\to \Pi_{x\in C} F_{C,x}$
The group $G$ is said to have Property (P) if this homomorphism is an isomorphism for every path $C$ in $T$.   
\end{defi}

The main result of Tits' paper is \cite[Th\'eor\`eme 4.5]{Tits1970} that says if $G$ is a group acting faithfully on a tree $T$ such that no proper non-trivial subtree is invariant and no end of $T$ is fixed and $G$ has Property (P) then the subgroup $G^+$ generated by stabilizers of arcs is either trivial or simple. 

\begin{prop}\label{PPropertyP}
Let $\cG$ be a free \gga.  Then the universal group acting on $T_{\cG}$ has Property (P). 
\end{prop}

\begin{proof}
Let $\Sigma$ be a scaffolding for $\cG$, let $T = T_\Sigma$, and let $G = \cU(cG)$.  Let $C$ be a path in $T$ and define $T_{C,x}$, $F_{C, x}$ and $F_C$ as above.  

The group homomorphism $\varphi\colon F_C\to \Pi_{x\in C} F_{C,x}$ is clearly injective so we just have to show it is surjective. Suppose $(f_x)_{x\in C}$ is an element in $\Pi_{x\in C} F_{C,x}$.  For $x$ in $C$ we let $g_x$ be an element in $G$ that fixes $C$ and acts  on $T_{C, x}$ like $f_x$. Set $\Sigma_{C, x}=\cup_{v\in T_{C, x}}\Sigma_v$. Note that $\V \Sigma=\cup_{x\in C} \Sigma_{C,x}$ and thus one can define a permutation $g$ of $\V \Sigma$ such that if $v\in \Sigma_{C, x}$ then $g(v)=f_x(v) = g_x(v)$.  

We show that $g\in G$. Note that if $x$ and $y$ are adjacent vertices in $C$ then there is just a single arc $a$ in $\Sigma$ such that $o(a)\in \Sigma_{C,y}$ and $t(a)\in \Sigma_{C,x}$.  Each of the elements $g_x$ fixes all such arcs and acts as a partial scaffolding automorphism on $\Sigma_{C,x}$, and thus $g$ is a scaffolding automorphism. It is acceptable because every $g_x$ in its action on $\Sigma_{C,x}$ is acceptable.  

Since clearly $\varphi(g) = (f_x)_{x\in C}$, this shows that $\varphi$ is surjective.
\end{proof}

\begin{remark}
    \begin{enumerate}
        \item  Note that the converse of the above proposition is not true in general, that is, if the universal group of a graph of group actions has Property (P) then we can not conclude that all the arc sets in this \ggawo\  are singletons.   Recall also that every possible group action on a tree arises from a \ggawo.
\item  A consequence of this is that the groups constructed with the methods from \cite{BurgerMozes2000}, \cite{Smith2017} and \cite{ReidSmith2020} all have Property (P).
 \end{enumerate}
\end{remark}

A property closely related to Property (P) is Property (P$_k$) introduced by Banks, Elder and Willis \cite[Definition 5.1]{BanksElderWillis2015}. It is defined analogously to Property (P) with the only difference that instead of fixing paths we fix \lq\lq thick\rq\rq\ paths. 

More precisely, for a set $S\subseteq \V T$ define 
\[
B(S, n) = \{y\in \V T\mid d_T(y,x)\leq n\mbox{ for some }x\in S\}.
\] 
Suppose $k$ is a given positive integer.   Amend the notation from Definition~\ref{DPropertyP} so that now $F=G_{(B(C, k-1))}$ and $F_{C,x}=F^{T_{C,x}}=G_{(B(C, k-1))}^{T_{C,x}}$.  As above we get a homomorphism $F_C\to \prod_{x\in C} F_{C,x}$.  We say that $G$ has $G$ has Property (P$_k$) if this homomorphism is an isomorphism for every path $C$ and that it has Property (IP$_k$) if this homomorphism is an isomorphism for every path $C$ of length 1. Note that  Tits' Property (P) is the same as Property (P$_1$).  Clearly, Property (P$_k$) implies Property (IP$_k$ ).
Banks, Elder and Willis \cite[Corollary~6.4]{BanksElderWillis2015} show that if $G$ is closed then the converse is also true, so Properties (P$_k$) and (IP$_k$) are equivalent for closed groups.

Above it is shown that if $\cG$ is a \gga\ such that every arc set $Y(a)$ has just one element then the action of the universal group on the tree $T_{\cG}$ has Property (P).  The following is an analogue for property (P$_k$).  

\begin{theo}\label{TIP_k}
Suppose that $\cG$ is a \gga\ with a finite base graph $\Delta$. Assume that the action of $G(v)$ on $X(v)$ is closed and subdegree finite for every $v \in \V \Delta$ and that $Y(a)$ is finite for every $a \in \A \Delta$.  Then there is a number $k$ such that $G=\cU(\cG)$ in its action on $T = T_\Sigma$ has Property (IP$_k$).
\end{theo}

\begin{proof}  
Let $\llbracket b\rrbracket$ be an arc in $T$, and let $[b]$ denote the corresponding arc bundle.  
Let 
\[
H_k = G_{(B_T(\llbracket b\rrbracket,k-1))}^{[b]}.
\]
that is, $H_k$ is the group that the pointwise stabiliser of the ball $(B_T(\llbracket b\rrbracket,k-1))$ in $T$ induces on the arc bundle $[b]$.  Then $H_i \geq H_{i+1}$ for all $i$ and since $[b]$ is finite there must be a number $k$ such that $H_i=H_{i+1}$ for all $i \geq k$.  

Let $h\in G_{(B_T(\llbracket b\rrbracket,k-1))}$ and let $h_0\in H_k$ be the permutation of $[b]$ induced by $h$.  Let $(g_i)$ be a sequence of group elements such that $g_i\in G_{(B_T(\llbracket b\rrbracket,i))}$ and $g_i$ induces the permutation $h_0$ on $[b]$.  Because adhesion sets are finite and stabilizers of vertices in $\Sigma$ are compact we see that the setwise stabilizers of adhesion sets are compact.  Therefore the sequence $(g_i)$ has a subsequence converging to some element $g$ in the universal group.   Clearly $g$ acts on $[b]$ in the same way as $h_0$ and acts trivially on $T$.

Now denote the two half trees obtained from $T$ by removing the edge $\{b, \rev{b}\}$ with endpoints $\bun{u}$ and $\bun{v}$  by $T_{\bun{u}}$ and $T_{\bun{v}}$, respectively, and denote the corresponding subgraphs of $\Sigma$ (induced by the vertex bundles) by $\Sigma_{\bun{u}}$ and $\Sigma_{\bun{v}}$.
Define $h_{\bun u}$ to be an automorphism of $\Sigma$ whose restriction to $\Sigma_{\bun{u}}$ is equal to the action of $h$ on $T_{\bun{u}}$ and whose action on $\Sigma_{\bun{v}}$ is equal to the action of $g$; note that this is indeed a scaffolding automorphism because the actions of $h_{\bun u}$ and $g$ on $[b]$ coincide, and that $h_b$ is acceptable because both $h$ and $g$ are.

This shows that if $C$ is the path consisting of the edge $\{b, \rev{b}\}$, then for every $f_{\bun u} \in F_{C,\bun{u}}$ there is some element $h_{\bun u} \in G$ which acts like $f_{\bun{u}}$ on $T_{\bun{u}}$ and acts trivially on $T_{\bun{v}}$. Analogously, for every $f_{\bun v} \in F_{C,\bun{v}}$ we can construct an element $h_{\bun u} \in G$, thus showing that each pair $(f_{\bun{u}},f_{\bun v})$ is contained in the image of the homomorpism $F_C\to \prod_{x\in C} F_{C,x}$, thus showing that this homomorphism is surjective. Since it is obviously injective, it must be a bijection.

Since $G$ has only finitely many orbits on the edges of $T$ we see that there is some number $k$ such that the above holds for every edge in $T$ and thus the action of $G$ on $T$ has Property (IP$_k$). 
\end{proof}

\begin{coro}
Suppose $\cG$ is a \gga\ such that the underlying graph is finite and for every vertex $v$ in the base graph $\Delta$ the set $X(v)$ is finite.   Then the action of $G$ on the tree $T=T_{\cG}$ has property(P$_k$) for some $k$.
\end{coro}

\begin{proof}
Clearly, the assumptions in Theorem~\ref{TIP_k} are satisfied.
Moreover, it follows from Corollary~\ref{CFinite} that $G$ in its action on $T$ is closed, and thus property (IP$_k$) implies (P$_k$) by the result of Banks, Elder and Willis mentioned above.
\end{proof}

Banks, Elder and Willis then go on to prove the following extensions of Tits' theorem.
For a group $G$ acting on a tree $T$ we let $G^{+_k}$ denote the subgroup of $G$ generated by the family  $\{G_{(B_T(b,k-1))}\}_{b\in \A T}$ of subgroups.

\begin{theo}\textnormal{ (\cite[Theorem~7.3]{BanksElderWillis2015})}
\label{T-BEW}
Let $G$ be a group acting faithfully on a tree $T$.  Suppose  $G$  for some $k$ and that $G$ neither leaves invariant a proper non-empty subtree of $T$ nor fixes an end of $T$ and satisfies Property (P$_k$).  Then every non-trivial subgroup of $G$ normalized by $G^{+_k}$ contains $G^{+_k}$; in particular $G^{+_k}$ is simple or trivial.  
\end{theo}

\begin{coro}
    Suppose that $\cG$ is a \gga\ with a finite base graph $\Delta$ such that for every vertex $v$ in $\Delta$ the stabilizer of every point in $X(v)$ in $G(v)$ is compact with respect to the permutation topology inherited from the action on $X(v)$ and for every arc $a$ in $\Delta$ the set $Y(a)$ is finite.

    Further suppose that $|\Star_{\Delta^+} (v)| > 1$ for every $v \in \V \Delta$ and that there is at least one vertex $v \in \V \Delta$ for which every arc in $\Star_{\Delta}(v)$ corresponds to at least two adhesion sets.
    
    Let $H$ be the universal group of $\cG$, seen as a subgroup of $\aut(T)$.  Then $H^{+_k}$ is simple or trivial.  
\end{coro}

\begin{proof}
   Finiteness of $\Delta$ implies that there are only finitely many orbits of vertices in the action of $H$ on $T$ by Corollary \ref{coro:treeorbits}. The condition on $\Star_{\Delta^+} (v)$ ensures that $T$ has no leaves and thus is infinite, hence $H$ cannot fix a finite subtree. The condition that there is a vertex where every arc corresponds to at least two adhesion sets prevents $H$ from fixing an end: the stabiliser of a corresponding vertex $\llbracket v \rrbracket$ in $T$ does not fix any edge incident to $\llbracket v \rrbracket$, and thus does not fix any end of $T$.
   Now the corollary is a direct consequence of Theorem~\ref{TIP_k} and Theorem \ref{T-BEW}.
\end{proof}

\section{Reduced \gga's and \lq\lq trees\rq\rq\ with parallel edges}\label{SReduced}

If one is only interested in the action of the universal group of a \gga\ on the associated tree then it is possible to replace the \gga\ with a \lq\lq simpler\rq\rq\ \gga\ such that the action of the universal group of that simpler \gga\ on the associated tree is the same as the action we started with. Below, two, closely related, ways to do this are described.

\begin{defi}   \label{DReduced}
    Let $\cG$ be a \gga.  
\begin{enumerate}
    \item     Construct a \gga\ $\cG'$ in the following way:   The base graph of $\cG'$ is equal to $\Delta$, the base graph of $\cG$.   For every vertex $v\in \V\Delta$ replace the action $(G(v), X(v))$ with the action $(G(v)^{X'(v)}, X'(v))$ where $X'(v)$ is the union of all the adhesion sets in $X(v)$.  For an arc $a$ in $\Delta$ set $(H'(a), Y'(a))$ at an arc $a$ in $\cG'$ equal to $(H(a), Y(a))$.  The embedding $\Phi'_a: Y'(a)\to X'(t(v))$ in $\cG'$ is such that $\Phi'_a(y)=\Phi_a(y)$ for all $y\in Y'(a)=Y(a)$.  The \gga\ $\cG'$ is called the \emph{reduced \gga} of $\cG$ and we say that a \gga\ $\cG$ is reduced if $\cG'=\cG$.  
    \item     Construct a \gga\ $\cG''$ in the following way:   The base graph of $\cG''$ is equal to $\Delta$, the base graph of $\cG$.   For every vertex $v\in \V\Delta$ replace the action $(G(v), X(v))$ with the action $(G(v)^{X''(v)}, X''(v))$ where $X''(v)$ is the disjoint union of the adhesion sets in $X(v)$.
    For an arc $a$ in $\Delta$ set $(H''(a), Y''(a))$ at an arc $a$ in $\cG'$ equal to $(H(a), Y(a))$.  The embedding $\Phi''_a: Y''(a)\to X''(t(v))$ in $\cG''$ is such that $\Phi''_a(y)=\Phi_a(y)\in\Phi_a(Y(a))$ for all $y\in Y''(a)=Y(a)$.  The \gga\ $\cG''$ is called the \emph{arc-reduced \gga} of $\cG$ and we say that a \gga\ $\cG$ is arc-reduced if $\cG''=\cG$.  
\end{enumerate}
\end{defi}

First note that the action of $G(v)$ on $X(v)$ gives a natural action on both $X'(v)$ and $X''(v)$ as the set $X'(v)$ is clearly invariant under $G(v)$.   If a \gga\ is arc-reduced then it is also automatically reduced.  A reduced \gga\ is also arc-reduced if and only if any two distinct adhesion sets are disjoint.  (If $a$ and $b$ are distinct arcs in $\Delta$ we think of an $a$-adhesion set as being distinct from any $b$-adhesion set.)  Suppose $T$ is the associated tree for $\cG$, $G)\cU(\cG)$ and $\Sigma$ is a canonical scaffolding.  If $w$ is a vertex in $T$ and $v$ its image under a $\Star$-covering map then $(G''(v), X''(v))$ is naturally isomorphic to the action of $G_{\{\Sigma_v\}}$ on the set of arcs in $\Sigma$ with terminal vertex in $\Sigma_v$.

\begin{lemm}\label{LReduced}
Let $\cG$ be a \gga, $T$ its associated tree and $G$ its universal group.  Define $\cG', \cG''$ as the reduced and arc-reduced \gga's of $\cG$, respectively.  Then the augmented base digraphs for $\cG'$ and $\cG''$ are equal to the augmented base digraph for $\cG$ and thus the associated trees for all three \gga's are equal.
Furthermore the induced actions of the universal groups for all three \gga's on $T$ are isomorphic.  
\end{lemm}

\begin{proof}  The statement about the isomorphic associated augmented base digraphs and trees is obvious.  

We show that the action of $G'=\cU(\cG')$ on $T$ is isomorphic to the action of $G$.  Note that the canonical scaffolding $\Sigma'$ of $\cG'$ can be embedded in the canonical scaffolding $\Sigma$ of $\cG$, and this embedding gives an action of $G$ on $\Sigma'$ by acceptable (with respect to $\cG'$) scaffolding automorphisms.   This action gives a homomorphism $\psi: G\to G'$.  Suppose $g'$ is an element in $G'$.  For every vertex $\llbracket v\rrbracket$ of, the local action $g'_{\llbracket v\rrbracket}$ of $g'$ at $\llbracket v\rrbracket$ has a corresponding element $g_{v}$ in $G(\pi(v))$. These local actions $g_{v}$ can be combined to give a group element $g\in G$ that acts on $\Sigma'$ in the same way as $g'$ does.   Thus the homomorphism $\psi: G\to G'$ is surjective.  Hence the induced action of $G$ on $T$ is isomorphic to the induced action of $G'$ on $T$. 

An analogous argument can also be used to show that the action of $\cU(\cG'')$ on $T$ is isomorphic to the action of $G$ on $T$.
\end{proof}

Now we are ready to prove Proposition~\ref{PTitsProperty} where it is convenient to use an arc-reduced \gga\ to show that if we have have a free \gga\ then the action of the universal group on the associated tree can also be defined in terms of a local action diagram.

\begin{proof}[Proof of Proposition~\ref{PTitsProperty}]
    In Section~\ref{SLocalActionDiagrams} we showed that given a local action diagram then it is possible to construct a free \gga\ such that the tree associated with the local action diagram is equal to the tree associated with this free \gga\ and the actions of the universal group of the local action diagram is the same as the action of the universal group of the \gga.

    Suppose that $\cG$ is a free \gga, $G$ is its universal group and $T=T_{\cG}$ is its associated tree.  Start by replacing $\cG$ with the arc-reduced \gga\ $\cG''$.  By Lemma \ref{LReduced} above, this doesn't change the action of the universal group on the associated tree. Then the method used to construct a \gga\ from a local action diagram can be reversed and we get a local action diagram whose universal group acts on the tree $T$ in the same way as the universal group of $\cG$ does.
\end{proof}

Finally, we note that the universal group of a \gga\  does not only act on the scaffolding and the corresponding tree, but also on an intermediate structure which can either be obtained by identifying all vertices in each vertex bundle to a single vertex (while keeping parallel edges), or by adding parallel edges to the tree $T$. We call this intermediate structure an \emph{augmented $\Star$-covering tree} and denote it by $T_+$. Clearly, the arcs of $\Sigma$ and the arcs of $T_+$ are in bijection. Thus $\cU$ induces a group of automorphisms of $T_+$ which we denote by $\cU_+$. Note that $\cU_+$ can alternatively defined as the subgroup of $\aut( T_+)$ consisting of all automorphisms whose ``local action'' on $\Star(\llbracket v\rrbracket)$ with respect to some consistent labelling of the arcs is contained in $G(\pi(\llbracket v\rrbracket)$ for every $\llbracket v \rrbracket \in \V T$.

If the \gga\ is reduced, then the action of a scaffolding isomorphism is uniquely determined by the action on the arcs, and thus $\cU$ and $\cU_+$ are in fact the same group. This further strengthens the similarities between graphs of group actions and local action diagrams. To summarise:

\begin{prop}
    Let $\cG$ be a \gga.  Then the actions of $\cU(\cG)$ and $\cU_+(\cG)$ on $T_{\cG}$ are isomorphic.  If the \gga\ $\cG$ is reduced then $T=T_+$ and $\cU(\cG)=\cU_+(\cG)$.
\end{prop}

\bibliographystyle{plain}
\bibliography{references}

\appendix

\section{Examples}\label{SExamples}

This appendix contains some basic examples of how our construction can be used to define new group actions on graphs.  

\subsection{The same local action everywhere}

Take a group $F$ acting transitively on some set $X$, finite or infinite.
Let $\Delta$ be a graph that has a single vertex $v$ and one self-reverse arc $a$.  Define a \gga\ such that the base graph is $\Delta$, the action at the vertex $v$ is the action of $F$ on $X$ and the action at the arc $a$ is the action of $F_{x_0}$ on $X\setminus\{x_0\}$ where $x_0$ is some point in $X$ and the agent of inversion is $h_a=1$.  Finally, $\Phi_a: X\setminus \{x_0\}\to X$ is just the inclusion embedding.  The adhesion sets are sets of the form $X\setminus \{x\}$ where $x$ ranges over the elements in $X$.  We see that $T=T_{\cG}$ is the $|X|$-regular tree.  Let $G$ denote the universal group of $\cG$ and let $\Sigma$ be a canonical scaffolding.  Suppose $g=[\tg; (g_v)_{v\in \V T}]$ is an element of $G$.  Suppose $e=\{v, w\}$ is an edge in $T$. There is an element $x_e\in X$ such that for every $x\in X\setminus \{x_e\}$ the sets $\{(v,x), (w,x)\}$ are edges in $\Sigma$ and these are the only edges in $\Sigma$ with one end vertex in $\Sigma_v$ and the other in $\Sigma_w$.  Now $g$ maps the edge $\{(v,x), (w,x)\}$ to the edge $\{(\tg(v), g_v(x)), (\tg(w), g_w(x))\}$.  Thus $g_w(x)=g_v(x)$ for all $x\in X$ and $g_w=g_v$.  From the above we conclude that $g_v$ is the same for all $v\in\V T$.   This type of group acting on the 3-regular tree is described in \cite[Example~3.3]{BanksElderWillis2015}.

We can also describe this group using the Burger-Mozes construction from Section~\ref{SBurger-Mozes}.    As remarked above, $T$ is the $|X|$-regular tree.  Let $c:\E T\to X$ be map such that $c(e)=x_e$.  Then $c$ is a legal colouring.  The Burger--Mozes group $U(F)$ is defined as the group of all automorphisms $g$ of $T$ such that $cg(c_{|\Star(v)})^{-1}$ is in $F$ for every vertex $v\in T$.  The group $G$, defined above, is the group of all automorphisms $g$ of $T$ such that the map $v\mapsto cg(c_{|\Star(v)})^{-1}$ is constant.  It is easy to show that this is indeed a subgroup of the automorphism group.  A special subgroup of this group is the group of all those automorphism $g$ of the tree $T$ such that $cg(c_{|\Star(v)})^{-1}=1$ for all $v\in \V T$.  This subgroup acts regularly (transitively and the only element having a fixed point is the identity) on $\V T$.  It is easy to see that this group is the free product of $|X|$ copies of $C_2$, the cyclic group with two elements.   We also see that the subgroup $G_v$ of $G$ fixing a given vertex $v$ in $T$ is isomorphic to $F$.

\subsection{Prescribed action on a ball around an edge}

The construction of Burger and Mozes \cite[Section~3.2]{BurgerMozes2000}, see also Section~\ref{SBurger-Mozes} of this paper, produces a group acting on a regular tree so that for each vertex in the tree the group acts locally like some given group $F$.  What if we want a vertex and arc transitive action on a regular tree such that the local action on the 1-ball around an edge is prescribed?

Let $T$ denote a regular tree.  Choose and edge $e=\{x,y\}$  in $T$ an define $T_0$ as the subtree spanned by $B_T(e,1)=B_T(x, 1)\cup B_T(y,1)$.  Let $H$ be a subgroup of the automorphism group of $T_0$ such that $H$ has two orbits on the vertices of $T_0$ (and thus also just two orbits on the edges of $T$).  

We assume that there is a transitive permutation group $F$ acting on $B_T(x,1)$ such that the action of group $H_{x,y}^{B_{T_0}(x,1)}$ on $B_{T_0}(x,1)$ is isomorphic to the action of $F_y$ on $B_T(x,1)$.  Furthermore, we assume that the group $F$ is generated by $F_y$ and its conjugates in $F$.

If such a group $G$ acting on $T$ exists then the quotient graph $T/G$ would have just one vertex and just one self-reverse arc.  It is convenient not to use this graph as a base graph for our \gga\ but to use instead its \lq\lq barycentric\rq\rq\  subdivision; the tree we obtain will be the barycentric subdivision of $T$, compare Section \ref{sec:inversion-subdivision}. 

To construct the group, let $\Delta$, the base graph for our \gga\ $\cG$, be the graph with two vertices $v$ and $w$ and two arcs $a$ and $\ba$ so that $v=t(a)$ and $w=o(a)$.  The action at $v$ is the action of a group $G(v)=F$ on $X(v)=B_{T_0}(x,1)$ and the action at $w$ is the action of $G(w)=H$ on $\V T_0$.   The action at the arc $a$ is the action of $H_x$ on $Y(a)=B_{T_0}(x,1)\setminus \{y\}$.  The embeddings $\Phi_a, \Phi_{\ba}$ are just the embeddings of the set $Y(a)$ into $X(v)$ and $X(w)$, respectively.   We see that there are precisely $|X(v)|$ adhesions sets in $X(v)$ and precisely two adhesion sets in $X(w)$ and the associated tree $T_{\cG}$ is the barycentric subdivision of $T$.  Thus we can think of the universal group of $\cG$ as acting on $T$ and it is easy to see that this action has the desired properties.  

Various obvious variations of the construction are possible, for instance having the group $G$ we construct act with two orbits on the vertices of $T$ and just one orbit on the arcs. 

\subsection{Odd and even}

Let $d\geq 3$ be a given positive integer.  
Our \gga\ has the graph with just one vertex $v$ and just one self-reverse arc $a$ as a base graph.  The action at the vertex $v$ is the action of $S_d$ on the set $X(v)=\{1, \ldots, d\}\times\{+1, -1\}$ such that if $g\in S_d$ then $g(k, \varepsilon)=(g(k), \sign(g)\varepsilon)$.  The action at the arc $a$ is the action of $C_2=\{+1, -1\}$ on the set $\{+1, -1\}$ such that the embedding $\Phi_a:  \{+1, -1\}\to \{1, \ldots, d\}\times\{+1, -1\}$ is given by $\Phi_a(\varepsilon)=(1, \varepsilon)$ and $h_a=1$ is the agent of inversion.  Let $G$ denote its universal group.  The associated tree $T=\cG$ is the $d$-regular tree.  Suppose we have a legal colouring $c$ of $T$ as in Section~\ref{SBurger-Mozes}.  If $u$ is a vertex in $T$ then the map $G\to S_d$ given by $g\mapsto g_u=cg(c_{|\Star(u)})^{-1}$ is surjective by Theorem~\ref{TLocal-action}.   Furthermore, the parity of $g_u$ is the same for every vertex $v\in \V T$.   To see this, we consider the action of $G$ on a canonical scaffolding graph $\Sigma$.   There are two edges in each arc bundle and these two edges have end vertices $\{(u, (k, +1)), (u', (k, +1))\}$ and $\{(u, (k, -1)), (u', (k, -1))\}$, respectively, where $u,u'$ are adjacent vertices in $T$.   Thus we can think of each edge in $\Sigma$ as being labelled with either $+1$ or $-1$.  Suppose $g=[\tg; (g_u)_{u\in \V T}]$ is an element from the universal group with $g_u\in S_d$.   If $g_u$ is even then an edge with end vertex in $\Sigma_u$ will be sent to an edge with the same label, but if $g_u$ is odd then an edge with an end vertex in $\Sigma_u$ will be sent to an edge with the opposite label.  Conversely, if an edge with an end vertex in $\Sigma_u$ is sent to an edge with the same label then $g_u$ has to be even and if such an edge is sent to an edge with the opposite label then $g_{u}$ is odd.   From this we see that if $u'$ is a vertex in $T$ adjacent to $u$ then $g_{u'}$ has the same parity as $g_v$.  Hence, all the $g_u$'s have the same parity.

One can play with this idea in various ways.  Let $\cG$ be a \gga\ described as follows:  The base graph has two vertices $v$ and $w$, and one pair of arcs $a$ and $\ba$,   The action at $v$ is the action of $G(v)=S_d$, $d\geq 3$ on the set $X(v)=\{1, \ldots, d\}\times\{+1, -1\}$ as described above and the action at $w$ is the natural action of $G(w)=S_d\times \{+1, -1\}$ on $X(w)=\{1, \ldots, d\}\times\{+1, -1\}$. 
The action at the arc $a$ is the action of $C_2=\{+1, -1\}$ on the set $\{+1, -1\}$ such that the embedding $\Phi_a:  \{+1, -1\}\to \{1, \ldots, d\}\times\{+1, -1\}$ is given by $\Phi_a(\varepsilon)=(1, \varepsilon)$ and $\Phi_{\ba}$ is defined in the same way.   The tree $T=T_{\cG}$ is just the $d$-regular tree.  For a vertex $u$ in $T$ and $g\in G$ we define $g_u$ in the same way as above.  Now we see that if $\pi$ is a $\Star$-covering map then the parity of $g_u$ is the same for all $u\in \pi^{-1}(v)$, but it is possible that the parities of $g_u$ and $g_{u'}$ are distinct for distinct vertices $u, u'\in\pi^{-1}(v)$.  

These examples could also be described using the idea of a \lq\lq tree\rq\rq\ with multiple edges as in the previous section.    In the examples above, we get a graph $T_+$ that looks like the $d$-regular tree except that if two vertices are adjacent then there are precisely two distinct edges having these two vertices as end vertices.  For each pair of adjacent vertices in $T_+$ we label one of the edges that has those vertices as end vertices with \lq\lq +\rq\rq\ and the other with \lq\lq $-$\rq\rq.  In both examples above, we see that if one \lq\lq +\rq\rq\ edge is mapped to a \lq\lq $-$\rq\rq\ edge then every \lq\lq +\rq\rq\ edge is mapped to a \lq\lq $-$\rq\rq\ edge.  In the first example and edge with end vertex $u$ is mapped to an edge of the opposite sign if and only if the local action at $u$ has odd parity, but in the second example this only holds for vertices in $T_+$ that are in $\pi^{-1}(v)$.

In these two examples we were playing with the homomorphism $S_d\to C_2$ in the cases when the universal group had just one or two orbits on the associated tree.  it is straightforward to use these ideas when the groups acting at the vertices map surjectively onto some other groups and also when we want more than two orbits on the vertex set of the tree.

\end{document}